\documentclass[12pt,a4paper]{amsart}

\usepackage{graphicx}
\usepackage{enumerate}
\usepackage{amsmath, amssymb, amsthm, amscd}
\usepackage{xcolor}
\usepackage{caption}
\usepackage{subcaption}
\usepackage{multirow}
\usepackage{stmaryrd}
\usepackage{accents}

\theoremstyle{plain}
\newtheorem{theorem}{Theorem}[section]
\newtheorem{lemma}[theorem]{Lemma}

\newtheorem{proposition}[theorem]{Proposition}

\theoremstyle{definition}
\newtheorem{definition}[theorem]{Definition}

\theoremstyle{remark}
\newtheorem{remark}[theorem]{Remark}
\numberwithin{equation}{section}

\def\bold#1{\mbox{\boldmath $#1$}}
\newcommand{\uu}[1]{\bold{#1}}

\newcommand{\abs}[1]{\lvert#1\rvert}
\newcommand{\D}{\partial}
\newcommand{\dd}{\mathrm{d}}
\newcommand{\dive}{\mathrm{div}}
\newcommand{\bdive}{\uu{\mathrm{div}}}

\newcommand{\Dt}{\partial_t}

\newcommand{\divM}{\mathrm{div}_{\mathcal{M}}}
\newcommand{\diam}{\text{diam}}

\newcommand{\gm}{\gamma}

\newcommand{\bgrd}{\uu{\nabla}}

\newcommand{\mbb}{\mathbb}

\newcommand{\mcal}{\mathcal}

\newcommand{\norm}[1]{\lVert#1\rVert}
\newcommand{\Norm}[1]{{\left\vert\kern-0.25ex\left\vert\kern-0.25ex\left\vert #1 
    \right\vert\kern-0.25ex\right\vert\kern-0.25ex\right\vert}}

\newcommand{\half}{\frac{1}{2}}
\newcommand{\veps}{\varepsilon}

\newcommand{\s}{\sigma}

\newcommand{\M}{\mcal{M}}
\newcommand{\E}{\mcal{E}}
\newcommand{\Q}{Q}

\newcommand{\Ds}{D_\sigma}
\newcommand{\Eds}{\bar{\mcal{E}}(D_\sigma)}
\newcommand{\Eint}{\mcal{E}_{\mathrm{int}}}
\newcommand{\Ek}{\mcal{E}(K)}
\newcommand{\Lm}{L_{\M}(\Omega)}

\newcommand{\Hez}{\uu{H}_{\E,0}(\Omega)}

\newcommand{\dt}{\delta t}
\newcommand{\chark}{\mcal{X}_K}

\newcommand{\fsigk}{F_{\sigma , K}}
\newcommand{\fesig}{F_{\epsilon , \sigma}}

\newcommand{\intr}{\mathrm{int}}
\newcommand{\extr}{\mathrm{ext}}

\newcommand{\absq}[1]{\abs{#1}^2}
\def\bold#1{\mbox{\boldmath $#1$}}

\newcommand\ubar[1]{%
\underaccent{\bar}{#1}}

\makeindex             

\begin{document}

\title[Asymptotic Preserving and Energy Stable Scheme]{An Asymptotic Preserving and Energy Stable Scheme for the Euler-Poisson System in the Quasineutral Limit}         

\author[Arun]{K.~R.~Arun}
\thanks{K.~R.~A.\ gratefully acknowledges Core Research Grant - CRG/2021/004078 from Science and Engineering Research Board, Department of Science \& Technology, Government of India. R.~G.\ would like to thank Minstry of Education, Government of India for the PMRF fellowship support.} 
\address{School of Mathematics, Indian Institute of Science Education and Research Thiruvananthapuram, Thiruvananthapuram 695551, India}  
\email{arun@iisertvm.ac.in, rahuldev19@iisertvm.ac.in, mainak17@iisertvm.ac.in}

\author[Ghorai]{Rahuldev Ghorai}

\author[Kar]{Mainak Kar}

\date{\today}



\maketitle

\begin{abstract}
      An asymptotic preserving and energy stable scheme for the Euler-Poisson system under the quasineutral scaling is designed and analysed. Correction terms are introduced in the convective fluxes and the electrostatic potential, which lead to the dissipation of mechanical energy and the entropy stability. The resolution of the semi-implicit in time finite volume in space fully-discrete scheme involves two
    steps: the solution of an elliptic problem for the potential and an explicit evaluation for the density and velocity. The proposed scheme possesses several physically relevant attributes, such as the
    the entropy stability and the consistency with the weak formulation of the continuous Euler-Poisson system. The AP property of the scheme, i.e.\ the boundedness of the mesh parameters with respect to the Debye length and its consistency with the quasineutral limit system, is shown. The results of numerical case studies are presented to substantiate the robustness and efficiency of the proposed method.     
\end{abstract}

\section{Introduction}
\label{sec:Int}

In physics, the matter in the universe is broadly classified into four different states: solids, liquids, gases and plasma. Often, plasma is visualised as a gas albeit its constituent particles are not neutral, instead they are negatively charged electrons and positively charged ions. Consequently, a plasma responds to both electric and magnetic fields and its dynamics is completely different from that of neutral gases. Mathematically, a plasma is modeled primarily by two different kinds of models: kinetic models and fluid models. Both these models can be used to compute the particle locations, their velocities, and the electro-magnetic fields present in the plasma. In a kinetic model, the plasma is represented by a distribution function of the particles at each point in a phase space. On the other hand, fluid models are developed based on macroscopic quantities, such as the mass density of plasma particles, their average velocity, pressure and so on. Even though kinetic models are found to be more accurate than fluid models, their applicability in numerical computations is way less than that of fluid models due to computational expenses associated with the higher dimensions in kinetic models. We refer the reader to, e.g.\ \cite{Che13,KT86} for more details on the physics of plasma, their mathematical models and analysis.

In this paper, we focus our attention on a one-fluid Euler-Poisson (EP) system which is a macroscopic hydrodynamic model commonly used in the mathematical and numerical modelling of plasma flows \cite{Che13} and semiconductor devices \cite{Jue01}. The system of governing equations consists of the Euler equations for the conservation laws of mass and momentum and a Poisson equation for the electrostatic potential. Though more realistic models will involve multiple equations for the conservation of mass and momentum corresponding to each species, such as the electrons and ions, the one-fluid model captures many of the essential mathematical properties of the fluid models of plasma. The literature on the analysis of the Euler-Poisson system leading to the existence and uniqueness is quite vast; we refer the interested reader to \cite{Bez93,BK18,Gam93,Mak86,MP90,MU87,MN95} and the references cited therein.

The one-fluid EP system upon a non-dimensionalisation of the variables gives rise to two characteristic paprameters, namely the Debye length and the electron plasma period \cite{KT86}, which play a significant role in its analysis. The Debye length measures charge imbalances in the plasma, while the electron plasma period quantifies oscillations caused by electrostatic restoring forces when such imbalances occur. We focus on the scenarios where these parameters become exceedingly small compared to typical macroscopic scales, known as the quasineutral regime, where local electric charge is absent in the plasma. As a result, even accidental charge imbalances, like numerical artifacts, can trigger high-frequency plasma oscillations. Accurately resolving these micro-scale phenomena poses a challenge for standard explicit numerical schemes. It requires using space and time steps smaller than both the Debye length and the electron plasma period. Failure to meet these requirements leads to numerical instabilities; see, e.g.\ \cite{CDV17,CDV07} for more discussions.

In order to construct accurate and robust numerical schemes for the EP system which sustain large variations of the Debye length, we use the `Asymptotic Preserving' (AP) methodology which was initially presented in the context of kinetic models of diffusive transport \cite{Jin99}. The working principle behind the AP framework can be explained as follows. Let $\mcal{P}_{\veps}$ be the problem at hand, where $\veps$ represents the perturbation parameter. As $\veps$ tends to zero, assume that the solution of $\mcal{P}_{\veps}$ converges to the solution of a well-posed problem $\mcal{P}_{0}$. This well-posed problem is referred to as the singular limit or the limit problem. A numerical scheme $\mcal{P}_{\veps}^{h}$ for $\mcal{P}_{\veps}$ with discretisation parameter $h$ is said to be AP if it converges to a numerical scheme $\mcal{P}_{0}^{h}$ which is a consistent discretisation of the limit problem $\mcal{P}_{0}$ as $\veps \to 0$ and the stability constraints on $\mcal{P}_{\veps}^{h}$ remain unaffected even when the order of magnitude of $\veps$ changes drastically.  In this regard, the AP schemes offer a remarkable versatility as they enable the use of uniform discretisations for both $\mcal{P}_{\veps}$ and $\mcal{P}_{0}$ and they are well capable of capturing the transition between $\mcal{P}_\veps$ and $\mcal{P}_0$ automatically without any further intervention. For these reasons, the AP methodology presents itself as a natural selection for the numerical approximation of the quasineutral limit in the sense that it preserves the limit at a discrete level; see, e.g.\ \cite{BC07, BCL+02, BD06, CDV07, Deg13, FJ11} for applications of AP schemes for EP and several other models.

A typical approach to derive AP schemes is the construction of a discretization of the given model using a semi-implicit time-stepping; see e.g. \cite{AS20,BAL+14,BRS18,Deg13,NBA+14,TPK20} for some developments in this direction. The computational advantages of adopting the semi-implicit formalism is that, it avoids inverting large and dense matrices, a common challenge posed by fully-implicit schemes. At the same time, it effectively overcomes the restrictive stability conditions often encountered in fully-explicit schemes. In the context of numerical approximations for hyperbolic models of compressible fluids, another critical requirement is to establish the entropy stability as it becomes vital in simulating physically valid weak solutions that are known to develop discontinuities over finite time. An entropy stable scheme with semi-implicit time stepping can be achieved by introducing an implicit correction to an explicit term present in the numerical flux. The choice for the implicit correction term is made aposteriori based on the requirements for entropy stability. We refer to \cite{DVB17, DVB20} for such constructions of entropy stable semi-implicit scheme for hyperbolic systems with source terms.

In this paper, our primary objective is to design and analyse an AP, semi-implicit and entropy stable finite volume scheme for the EP system on a MAC grid. The essence of energy stability lies in introducing a shifted momentum in the convective fluxes and a shifted potential in the source term in the momentum equation; see \cite{AGK22,DVB17,GVV13} for related treatments. The momentum shift, directly proportional to the combined pressure gradient and electrostatic source term and the potential shift, proportional to the divergence of momentum, play a vital role in stabilising the flow by dissipating mechanical energy at all Debye lengths. In order to overcome stiff stability restrictions, a semi-implicit time discretisation has been used. The fully-discrete scheme satisfies the essential apriori entropy stability inequalities analogous to those found in the continuous model. Suitably combining the stabilised mass equation and the Poisson equation leads to a well-posed linear elliptic problem for the electrostatic potential. After resolving the elliptic equation, we can solve for the density and velocity explicitly from the mass and momentum updates respectively. Through a meticulous weak consistency analysis we confirm that the numerical scheme aligns with the weak formulation of the EP system as the mesh is refined. At the end, we establish the AP property of the scheme both theoretically and numerically.

Rest of this paper is organised as follows. In Section~\ref{sec:cont-case} we recall some apriori energy stability estimates and the momentum and potential stabilisation technique based on the apriori energy 
stability considerations. The discretisation of the domain and the discrete differential operators are introduced in Section~\ref{sec:MAC_disc_diff}. The semi-implicit numerical scheme and its energy stability attributes are presented in Section~\ref{sec:SI_scheme}. In Section~\ref{sec:weak_cons} we present the weak consistency results. The results of numerical case studies are presented in Section~\ref{sec:num_res}. Finally, the paper is concluded with some remarks in Section~\ref{sec:conclusion}.

\section{Euler-Poisson System and its Quasineutral Limit}
\label{sec:cont-case}
We start with the following Euler-Poisson system, parametrised by the Debye length:
\begin{subequations}
\label{eq:ep}
\begin{align}
  \D_t\rho^\veps+\dive (\rho^\veps\uu{u}^\veps)&=0, \label{eq:cons_mas}\\ 
  \D_t(\rho^\veps\uu{u}^\veps)+\bdive
  (\rho^\veps\uu{u}^\veps\otimes\uu{u}^\veps)+\bgrd
  p(\rho^\veps) &=\rho^\veps\bgrd\phi^\veps, \label{eq:cons_mom}\\
  \veps^{2}\Delta\phi^\veps&=\rho^\veps-1, \label{eq:poisson}\\
  \rho^\veps\vert_{t=0}=\rho^\veps_0, \quad
  \uu{u}^\veps\vert_{t=0}=\uu{u}^\veps_0,\label{eq:eq_ic}
\end{align}
\end{subequations}
for $(t, \uu{x})\in Q_T:=(0, T ) \times \Omega$, where $T>0$ and $\Omega$ is an open, bounded and connected subset of $\mathbb{R}^d$, $d\geq1$. Apart from the initial conditions \eqref{eq:eq_ic}, the Poisson equation \eqref{eq:poisson} is supplemented with periodic, Dirichlet, Neumann or mixed boundary conditions of the form
\begin{equation}
\label{eq:cont_bdary}
\phi^\veps = \phi^D_\veps\;\text{on}\;\Gamma^D\subset\D\Omega,\;    \bgrd\phi^\veps\cdot\uu{\nu} = 0\;\text{on}\;\Gamma^N = \D\Omega\backslash\Gamma^D, 
\end{equation} 
where $\uu{\nu}$ denotes the unit outward normal vector on $\D\Omega$. The parameter $\veps\in (0,1]$ is the scaled Debye length which is characteristic length over which ions and electrons can be separated in a plasma \cite{Che13}. The unknowns $\rho^\veps,\uu{u}^\veps,\phi^\veps$ are the electron density, electron velocity and electrostatic potential respectively. The pressure $p$ is assumed to follow a barotropic equation of state $p(\rho):=\rho^\gamma$, with $\gamma \geq 1$ being the ratio of specific heats.

We start by obtaining some apriori energy estimates satisfied by the solutions of the system \eqref{eq:ep}. As a first step, we define the internal energy per unit volume or the so-called Helmholtz function 
$\psi_\gamma\colon\mbb{R}_{+}\to \mbb{R}$, via 
\begin{equation}
   \label{eq:psi_gamma}
   \psi_\gamma(\rho) =
   \begin{cases}
     \rho\ln\rho, & \mbox{if} \ \gamma=1, \\
     \dfrac{\rho^\gamma}{\gamma-1}, & \mbox{if} \ \gamma>1.
   \end{cases}  
\end{equation}

The internal energy $\mcal{I}_\veps$, potential energy $\mcal{P}_\veps$ and kinetic energy $\mcal{K}_\veps$ of the system \eqref{eq:cons_mas}-\eqref{eq:poisson} are defined by
\begin{equation}
  \mcal{I}_{\veps}(t):=\int_{\Omega}\psi_{\gamma}(\rho^\veps)\dd\uu{x}, \quad 
  \mcal{P}_{\veps}(t):=\frac{\veps^2}{2}\int_{\Omega}\abs{\bgrd\phi^\veps}^2
  \dd\uu{x}, \quad
  \mcal{K}_{\veps}(t):= \int_{\Omega}\frac{1}{2}\rho^\veps{\abs{\uu{u}^\veps}}^2
  \dd\uu{x}. \label{eq:all_engy}
\end{equation}

\subsection{Apriori Energy Estimates}
\label{sec:apr_est}
\begin{proposition}
  \label{prop:engy_balance}
  The regular solutions of \eqref{eq:cons_mas}-\eqref{eq:poisson} satisfy the
  following identities. 
  \begin{itemize}
  \item A renormalisation identity:
    \begin{equation}
      \label{eq:renorm}
        \Dt\psi_\gm(\rho^\veps)
      +\dive\left(\psi_\gm(\rho^\veps)\uu{u}^\veps\right)
      +p^\veps\dive(\uu{u}^\veps)=0.
    \end{equation}    
    \item A positive renormalisation identity:
    \begin{equation}
      \label{eq:porenorm}
      \Dt\Pi_\gm(\rho^\veps)
      +\dive\left(\psi_\gamma(\rho^\veps)-\psi_\gamma^\prime(1)\rho^\veps\right)\uu{u}^\veps+p^\veps\dive\uu{u}^\veps=0,
    \end{equation}
    where
    $\Pi_\gm(\rho):=\psi_\gamma(\rho)-\psi_\gamma(1)-\psi_{\gamma}^{\prime}(1)(\rho-1)$ 
    denotes the relative internal energy, which is an affine approximation of $\psi_\gamma$ with respect to the constant state $\rho=1$.
  \item A kinetic energy identity:
    \begin{equation}
      \label{eq:kinbal}
      \Dt\Big(\frac{1}{2}\rho^\veps{\abs{\uu{u}^\veps}}^2\Big)
      +\dive\Big(\frac{1}{2}\rho^\veps{\abs{\uu{u}^\veps}}^2\uu{u}^\veps\Big)
      +\bgrd p^\veps\cdot\uu{u}^\veps
      =\rho^\veps\uu{u}^\veps\cdot\bgrd\phi^\veps.
    \end{equation}
    \item A potential energy identity:
    \begin{equation}
      \label{eq:potbal}
     \Dt\Big(\frac{\veps^2}{2}\abs{\bgrd\phi^\veps}^{2}\Big)
      -\veps^2\dive(\phi^\veps\Dt(\bgrd\phi^\veps))=\phi^\veps\dive(\rho^\veps\uu{u}^\veps).
    \end{equation}
  \item Adding \eqref{eq:porenorm}, \eqref{eq:kinbal} and \eqref{eq:potbal} we get the total energy or entropy identity:
  \begin{equation}
  \label{eq:eng_id}
  \Dt\Big(\Pi_\gm(\rho^\veps)+\frac{1}{2}\rho^\veps{\abs{\uu{u}^\veps}}^2+\frac{\veps^2}{2}\abs{\bgrd\phi^\veps}^{2}\Big)
      +\dive\Big((\psi_{\gm}(\rho^\veps) -\psi_\gamma^\prime(1)\rho^\veps+ p^\veps + \frac{1}{2}\rho^\veps{\abs{\uu{u}^\veps}}^2-\rho^\veps\phi^\veps)\uu{u}^\veps-\veps^2\phi^\veps\Dt(\bgrd\phi^\veps)\Big)=0.
  \end{equation}
  \end{itemize}
  Taking the integral over the domain $\Omega$, we obtain the following total energy conservation:  
  \begin{equation}
    \label{eq:apr_est_stab}
    \frac{\dd}{\dd t}
    E_{\veps}(t)=\frac{\dd}{\dd t}(\mcal{I}_\veps+\mcal{P}_\veps+\mcal{K}_\veps)(t)=0. 
  \end{equation}
\end{proposition}
Here, $E_{\veps}\colon\mbb{R} \to \mbb{R}$ denotes the total energy of the system \eqref{eq:cons_mas}-\eqref{eq:poisson} and later, the stability of solutions will be established using the energy conservation \eqref{eq:apr_est_stab}.
\begin{proof}
Proofs of \eqref{eq:renorm} and \eqref{eq:kinbal} are classical and hence omitted; see, e.g.\ \cite{HLS21} and the references therein. In order to prove \eqref{eq:potbal}, we multiply \eqref{eq:cons_mas} by $\phi^\veps$ and using \eqref{eq:poisson} we write
\begin{align*}
0&=\phi^\veps\D_t\rho^\veps+\phi^\veps\dive (\rho^\veps\uu{u}^\veps)\\
&=\veps^2\phi^\veps\D_t(\Delta\phi^\veps)+\phi^\veps\dive (\rho^\veps\uu{u}^\veps),\\
&=-\Dt\Big(\frac{\veps^2}{2}\abs{\bgrd\phi^\veps}^{2}\Big)
      +\veps^2\dive(\phi^\veps\Dt(\bgrd\phi^\veps))+\phi^\veps\dive(\rho^\veps\uu{u}^\veps).
\end{align*}
\end{proof}

\subsection{Quasineutral Limit}
\label{sec:QN_limit}

The quasineutral limit of the EP system \eqref{eq:cons_mas}-\eqref{eq:poisson} is obtained by passing to the limit $\veps\to0$; see, e.g.\ \cite{Deg13} and the references therein. Letting $\veps\to0$ formally yields the following incompressible Euler equations:
\begin{equation}
\label{eq:incom_Euler}
    \Dt \uu{U}+\bdive(\uu{U}\otimes\uu{U})+\bgrd\Phi=\uu{0}, \quad \dive(\uu{U})=0.
\end{equation}

\subsection{Stabilisation}
\label{sec:stab}

The primary goal of this paper is to develop a semi-implicit scheme for the EP system \eqref{eq:ep} and study its stability in the sense of numerical control of total energy. In other words, we aim at obtaining a discrete equivalent of the energy stability \eqref{eq:eng_id}. In order to enforce the energy stability of the numerical solution, we adopt the formalism used in \cite{AGK22,DVB17,DVB20,PV16}. The idea is to apply a stabilisation, therein a shift in the momentum is introduced in the convective fluxes of mass and momentum and a shift in electrostatic potential on the right hand side of the momentum balance which yields the modified system:
\begin{subequations} 
\label{eq:r_EP}
\begin{align}
  \D_t\rho^\veps+\dive (\rho^\veps\uu{u}^\veps-\uu{q}^\veps)&=0, \label{eq:r_cons_mas}\\ 
  \D_t(\rho^\veps\uu{u}^\veps)+\bdive
  (\uu{u}^\veps\otimes(\rho^\veps\uu{u}^\veps-\uu{q}^\veps))+\bgrd
  p(\rho^\veps) &=\rho^\veps\bgrd(\phi^\veps-\Lambda^\veps), \label{eq:r_cons_mom}\\
  \veps^{2}\Delta\phi^\veps&=\rho^\veps-1.\label{eq:r_poisson}
\end{align}
\end{subequations}
The stabilisations $\uu{q}^\veps$ and $\Lambda^\veps$ are to be obtained aposteriori after carrying out an energy stability analysis. Analogous to the original system \eqref{eq:ep}, the stabilised system \eqref{eq:r_EP} satisfies the following energy identities. 
\begin{proposition}
  \label{prop:r_engy_balance}
  The regular solutions of \eqref{eq:r_EP} satisfy the following identities. 
  \begin{itemize}
  \item A renormalisation identity
    \begin{equation}
      \label{eq:r_renorm}
      \Dt\psi(\rho^\veps)
      +\dive\bigg(\psi(\rho^\veps)\Big(\uu{u}^\veps-\frac{\uu{q}^\veps}{\rho^\veps}\Big)\bigg)
      +p(\rho^\veps)\dive\Big(\uu{u}^\veps-\frac{\uu{q}^\veps}{\rho^\veps}\Big)=0.
    \end{equation}
    \item A positive renormalisation identity
    \begin{equation}
      \label{eq:r_po_renorm}
      \Dt\Pi_\gamma(\rho^\veps)
      +\dive\bigg(\Big(\psi_\gamma(\rho^\veps) - \psi^{\prime}_\gamma(1)\rho^\veps\Big)\Big(\uu{u}^\veps-\frac{\uu{q}^\veps}{\rho^\veps}\Big)\bigg)
      +p(\rho^\veps)\dive\Big(\uu{u}^\veps-\frac{\uu{q}^\veps}{\rho^\veps}\Big)=0.
    \end{equation}
  \item A kinetic energy identity
    \begin{align}
      \label{eq:r_kinbal}
      \Dt\Big(\frac{1}{2}\rho^\veps{\abs{\uu{u}^\veps}}^2\Big)
      &+\dive\Big(\frac{1}{2}{\abs{\uu{u}^\veps}}^2(\rho^\veps\uu{u}^\veps-\uu{q}^\veps)\Big)
      +\Big(\uu{u}^\veps-\frac{\uu{q}^\veps}{\rho^\veps}\Big)\cdot\bgrd p(\rho^\veps)\\
      &=\bgrd \phi^\veps\cdot(\rho^\veps\uu{u}^\veps-\uu{q}^\veps)
      +\frac{\uu{q}^\veps}{\rho^\veps}\cdot(\rho^\veps\bgrd\phi^\veps-\bgrd p(\rho^\veps))-\rho^\veps\uu{u}^\veps\cdot\bgrd\Lambda^\veps.\nonumber    \end{align}
  \item A potential energy identity
    \begin{equation}
      \label{eq:r_potbal}
     \Dt\Big(\frac{\veps^2}{2}\abs{\bgrd\phi^\veps}^{2}\Big)
      -\veps^2\dive(\phi^\veps\Dt(\bgrd\phi^\veps))=\phi^\veps\dive(\rho^\veps\uu{u}^\veps-\uu{q}^\veps).
    \end{equation}
  \item Adding \eqref{eq:r_po_renorm}, \eqref{eq:r_kinbal} and \eqref{eq:r_potbal} we get the total energy or entropy identity
  \begin{equation}
      \label{eq:r_eng_id}
  \begin{aligned}
  &\Dt\Big(\Pi_\gm(\rho^\veps)+\frac{1}{2}\rho^\veps{\abs{\uu{u}^\veps}}^2+\frac{\veps^2}{2}\abs{\bgrd\phi^\veps}^{2}\Big)\\
  &+\dive\Big(\big(\psi^\prime(\rho^\veps)-\psi^\prime(1)\rho^\veps+\frac{1}{2}{\abs{\uu{u}^\veps}}^2-\phi^\veps\big)(\rho^\veps\uu{u}^\veps-\uu{q}^\veps)-\veps^2\phi^\veps\Dt(\bgrd\phi^\veps)+\rho^\veps\uu{u}^\veps\Lambda^\veps\Big)\\ 
  &=\frac{\uu{q}^\veps}{\rho^\veps}\cdot(\rho^\veps\bgrd\phi^\veps-\bgrd p(\rho^\veps))+\Lambda^\veps\dive(\rho^\veps\uu{u}^\veps) 
  \end{aligned}
  \end{equation}
  \end{itemize}
  \end{proposition}
  \begin{proof}
  The proof can be done in the exactly same way as in Proposition~\ref{prop:engy_balance}.
  \end{proof}
Thus, at least at the continuous level, we immediately see that if we take $\uu{q}^\veps=\eta\rho^\veps(\bgrd p(\rho^\veps)-\rho^\veps\bgrd\phi^\veps)$ with $\eta > 0$ and $\Lambda^\veps=-\alpha\dive(\rho^\veps\uu{u}^\veps)$ with $\alpha > 0$ then we can have the decay of total energy. Guided by the above result, we introduce a discretisation of the stabilised EP system \eqref{eq:r_EP} so as to obtain a discrete equivalent of the stability.

\section{Discretisation of the Domain with MAC Grid and Discrete Differential Operators}
\label{sec:MAC_disc_diff}

In order to approximate the Euler-Poisson system \eqref{eq:r_EP} in a finite volume framework, we assume that the computational space-domain $\Omega\subseteq \mbb{R}^d$ such that the closure of $\Omega$ is a union of closed rectangles ($d=2$) or closed orthogonal parallelepipeds ($d=3$).

\subsection{Mesh and Unknowns}
\label{subsec:msh_unkn}
In this subsection, we introduce a discretisation approach for the domain $\Omega$, employing a marker and cell (MAC) grid along with associated discrete function spaces; see \cite{Cia91, GHL+18, GHM+16, HW65} for more details. A MAC grid is a pair $\mcal{T}=(\mcal{M},\mcal{E})$, where $\mcal{M}$ is called the primal mesh which is a partition of $\bar{\Omega}$ consisting of possibly non-uniform closed rectangles ($d=2$) or parallelepipeds ($d=3$) and $\mcal{E}$ is the collection of all edges of the primal mesh cells. For each $\sigma\in\mcal{E}$, we construct a dual cell $\Ds$ which is the union of half-portions of the primal cells $K$ and $L$, where $\s=\bar{K}\cap\bar{L}$. Furthermore, we decompose $\E$ as $\mcal{E}=\cup_{i=1}^d\mcal{E}^{(i)}$, where $\mcal{E}^{(i)}=\mcal{E}_\intr^{(i)}\cup\mcal{E}_\extr^{(i)}$. Here, $\mcal{E}_\intr^{(i)}$ and $\mcal{E}_\extr^{(i)}$ are, respectively, the collection of $d-1$ dimensional internal and external edges that are orthogonal to the $i$-th unit vector $e^{(i)}$ of the canonical basis of $\mbb{R}^d$. We denote by $\mcal{E}(K)$, the collection of all edges of $K\in\mcal{M}$ and $\tilde{\mcal{E}}(D_\sigma)$, the collection of all edges of the dual cell $D_\sigma$. Also, we denote by $\E^D$ the collection of all external edges $\s\in\cup_{i=1}^d\E^{(i)}_\extr$ such that $\s\subset\Gamma^D$.

Now, we define a discrete function space $L_{\mcal{M}}(\Omega) \subset L^{\infty}(\Omega)$, consisting of scalar valued functions which are piecewise constant on each primal cell $K\in\mcal{M}$. Analogously, we denote by $\uu{H}_{\mcal{E}}(\Omega)=\prod_{i=1}^{d} H^{(i)}_{\mcal{E}}(\Omega)$, the set of vector valued (in $\mbb{R}^d$) functions which are constant on each dual cell $D_\sigma$ and for each $i=1,2,\dots,d$. The space of vector valued functions vanishing on the external edges is denoted as  $\uu{H}_{\mcal{E},0}(\Omega)=\prod_{i=1}^d H^{(i)}_{\mcal{E},0}(\Omega)$, where  $H^{(i)}_{\mcal{E},0}(\Omega)$ contains those elements of $H^{(i)}_{\mcal{E}}(\Omega)$ which vanish on the external edges. For a primal grid function $q\in L_{\mcal{M}}(\Omega)$ such that $q =\sum_{K\in\mcal{M}}q_K\chark$, and for each $\sigma = K|L \in\cup_{i=1}^d\mcal{E}^{(i)}_\intr$, the dual average $q_{D_\sigma}$ of $q$ over $D_\sigma$ is defined via the relation
  \begin{equation}
   \label{eq:mass_dual}
     \abs{D_\sigma}q_{D_\sigma}=\abs{D_{\sigma,K}}q_K+\abs{D_{\sigma,L}}q_L.
  \end{equation}

\subsection{Discrete Convection Fluxes and Differential Operators}
\label{sec:dic_convect}

In this section we introduce the discrete convection fluxes and discrete differential operators on the functional spaces described above. As a first step, for each $\sigma=K|L\in\mcal{E}^{(i)}_{\mathrm{int}}$, $i=1,2,\dots,d$, we assert an interface value $\rho_\sigma=\rho_{KL}$ of $\rho$ on $\sigma=K|L$, whose existence is guaranteed by the following lemma, see \cite[Lemma 2]{GH+21}. 
\begin{lemma}
    \label{lem:rho_sig}
    Let $\psi$ be a strictly convex and continuously differentiable function over an open interval $I$ of $\mathbb{R}$. Let $\rho_K, \rho_L\in I$. Then there exists unique $\rho_{KL}\in \llbracket \rho_K, 
    \rho_L\rrbracket$ such that
    \begin{equation}
    \psi(\rho_K) + \psi^\prime(\rho_K)(\rho_{KL}-\rho_L) = \psi(\rho_L) + \psi^\prime(\rho_L)(\rho_{KL}-\rho_L),\; \mbox{if}\; \rho_K\neq\rho_L,
    \end{equation}
    and $\rho_{KL}=\rho_K=\rho_L$ otherwise. In particular for $\psi=p=\rho^\gamma$, for each $K,L\in\mcal{M}$ there exists a unique $\rho_{KL}\in \llbracket \rho_K, \rho_L\rrbracket$ such that 
    \begin{equation}
    \label{eq:rhosigchoice}
    \begin{aligned}
    \rho_K^\gamma - \rho_L^\gamma &= \rho_{KL}[\psi^\prime_\gamma(\rho_K)-\psi^\prime_\gamma(\rho_L)], \
    \mbox{if} \ \rho_K\neq\rho_L,\\
    \rho_{KL}=\rho_K &=\rho_L, \ \mbox{otherwise}.
    \end{aligned}
    \end{equation}
    \end{lemma}
    Throughout the paper, $\llbracket a, b\rrbracket$ stands for the interval
$[\min(a, b), \max(a, b)]$, for any real numbers $a$ and $b$. 
\begin{definition}\label{def:disc_conv_flux}
  Assume a discretisation of $\Omega$ with MAC grid and the discrete function spaces as defined above. 
  \begin{itemize}
  \item For each $K \in \mcal{M}$ and $\sigma\in \mcal{E}(K)$, the stabilized mass flux $F_{\sigma,K}$ is defined by
    \begin{equation}
      \label{eq:mass_flux}
      F_{\sigma,K} :=\abs{\sigma}(\rho_{\s}u_{\sigma, K} - Q_{\sigma, K}),
    \end{equation}
    with $(\rho,\uu{u}, \uu{Q})\in L_{\mcal{M}}(\Omega) \times \uu{H}_{\mcal{E},0}(\Omega)\times \uu{H}_{\mcal{E},0}(\Omega)$. Here, $u_{\sigma,K}=u_{\sigma} \uu{e}^{(i)}\cdot\uu{\nu}_{\sigma,K}$ and $Q_{\sigma,K}=Q_{\sigma} \uu{e}^{(i)}\cdot\uu{\nu}_{\sigma, K}$, where ${\nu}_{\sigma, K}$ denotes the unit vector normal to the edge $\sigma=K|L\in\mcal{E}^{(i)}_{\mathrm{int}}$ in the direction outward to the cell $K$. The choice for the stabilization term $\uu{Q}\in\uu{H}_{\mcal{E},0}(\Omega)$ will be determined after an energy stability analysis of the overall finite volume scheme. 
  \item For a fixed $i=1,2,\dots,d$, for each $\sigma\in\mcal{E}^{(i)}, \epsilon\in\bar{\E}_{D_\sigma}$ and $v\in H^{(i)}_{\mcal{E},0}$, the upwind momentum convection flux is given by the expression
    \label{mom_flux_up} 
    \begin{align}
      \sum_{\epsilon\in\Eds}\fesig v_{\epsilon, \mathrm{up}},
    \end{align}
    where $\fesig$ is the mass flux across the edge $\epsilon$ of the dual cell $\Ds$ which is $0$ if $\epsilon\in\mcal{E}_{\mathrm{ext}}$. On the other hand, if $\epsilon = D_\sigma|D_\sigma^{\prime}$, it is a suitable linear combination of the primal mass convection fluxes $F_{\s,K}$ and $F_{\s^\prime,K}$ at the neighbouring edges with constant coefficients; see, e.g.\ \cite{GHL+18} for more details of the above construction. 
\item For any dual face $\epsilon=\Ds|D_{\sigma^{\prime}}$ we have $F_{\epsilon,\sigma}=-F_{\epsilon,\sigma^{\prime}}$. Note that same is true for the primary mass fluxes, i.e.\ for any $\sigma=K|L$, we have $\fsigk=-F_{\sigma^{\prime},K}$.
\item $v_{\epsilon,\mathrm{up}}$ is determined by the following upwind choice:
\label{eq:mom_up}
\begin{equation}
v_{\epsilon,\mathrm{up}}=\begin{cases}
v_{\sigma}, &\fesig\geq 0,\\
v_{\sigma^{\prime}}, &\mathrm{otherwise,}
\end{cases}
\end{equation}
where $\epsilon\in\Eds$, $\epsilon=\Ds|D_{\sigma^{\prime}}$.
\end{itemize}
\end{definition}

\begin{definition}[Discrete gradient and discrete divergence]
\label{def:disc_grad_div}
    The discrete gradient operator  $\bgrd_{\mcal{E}}:L_{\mcal{M}}(\Omega)\rightarrow\uu{H}_{\mcal{E}}(\Omega)$ is defined by the map $q \mapsto \bgrd_{\mcal{E}}q=\Big(\D^{(1)}_{\mcal{E}}q,\D^{(2)}_{\mcal{E}}q,\dots,\D^{(d)}_{\mcal{E}}q\Big)$, where for each $i=1,2,\dots,d$, $\partial^{(i)}_{\mcal {E}}q$ denotes
\begin{equation}
\label{eq:disc_grad}
\partial^{(i)}_{\mcal {E}}q=\sum_{\sigma\in \mcal{E}^{(i)}_\intr}(\partial^{(i)}_{\mcal{E}}q)_{\sigma}\mcal{X}_{D_{\sigma}}, \ \mbox{with} \  (\partial^{(i)}_{\mcal{E}}q)_{\sigma}= \frac{\abs{\sigma}}{\abs{D_\sigma}}(q_{K_\s}-q_K)\uu{e}^{(i)}\cdot \uu{\nu}_{\sigma,K}, \; \s\in\E(K),\;K\in\M.
\end{equation}
The values $q_{K_\s}$ corresponding to the dual edge $\s\in\E(K)$ is given by
\begin{equation}
\label{eq:Ksig_choice}
q_{K_\s}=\begin{cases}
q_L, &\mathrm{if}\;\s\in\cup_{i=1}^d\E^{(i)}_\intr,\;\s=K|L\\
q_K, &\mathrm{if}\;\s\in\cup_{i=1}^d\E^{(i)}_\extr,\;\s\subset\Gamma^N\\
q_\s= \frac{1}{\abs{\s}}\int_{\s}q^D(\uu{x})\dd\uu{\gamma}(\uu{x}),&\mathrm{if}\;\s\in\cup_{i=1}^d\E^{(i)}_\extr,\;\s\subset\Gamma^D,
\end{cases}
\end{equation}
where $q^D$ is the boundary value of $q$ on $\Gamma^D$ and $\uu{\gamma}$ denotes the Lebesgue measure on $\mbb{R}^{d-1}$. Note that if $q\in L_\M(\Omega)$ has a purely Neumann or no-flux boundary condition, the gradient vanishes on the external edges and we have $\bgrd_\E q\in\uu{H}_{\E,0}(\Omega)$.

The discrete divergence operator $\divM:\uu{H}_{\E}(\Omega)\rightarrow L_{\mcal{M}}(\Omega)$ is defined as $\uu{v} \mapsto \divM \uu{v}=\sum_{K\in\M}(\dive_{\mcal{M}} \uu{v})_K \mcal{X}_{K}$, where for each $K\in\M$, $(\dive_{\mcal{M}} \uu{v})_K $ denotes
\label{def:disc_div}
\begin{equation}
\label{eq:disc_div}
(\divM \uu{v})_K =\frac{1}{\abs{K}}\sum_{\sigma\in\mcal{E}(K)}\abs{\sigma} v_{\sigma,K}.
\end{equation}
The above discrete operators satisfy the following `gradient-divergence duality'; see \cite{EG+10,GHL+18} for further details. 
\end{definition}
\begin{proposition}
  For any $(q,\uu{v})\in\Lm\times\Hez$, the gradient-divergence duality is given by
  \begin{equation}
    \label{eq:disc_dual}
    \int_{\Omega}q(\divM \uu{v})\dd\uu{x}+\int_{\Omega}\bgrd_{\mcal{E}}q\cdot\uu{v}\dd\uu{x}=0.
  \end{equation}
\end{proposition}

\label{def:disc_lap}
\begin{definition}[Discrete Laplacian] 
The discrete Laplacian  $\Delta_{\M}:L_{\mcal{M}}(\Omega)\rightarrow L_{\mcal{\M}}(\Omega)$ is defined as $q \mapsto \Delta_{\M}q =\sum_{K\in\M}(\Delta_{\M} q)_K \mcal{X}_{K}$, where for each $K\in\M$, $(\Delta_{\M} q)_K $ denotes
\begin{equation}
(\Delta_{\M} q)_K  = (\dive_{\M}(\bgrd_{\E}{q}))_{K},
\end{equation}
where the operators $\dive_{\M}$ and $\bgrd_{\E}$ are defined in \eqref{eq:disc_div} and \eqref{eq:disc_grad} respectively.
\end{definition}

\section{Energy Stable Semi-implicit Scheme}
\label{sec:SI_scheme}
In the following, we present a linearly implicit in time, fully-discrete in space finite volume scheme for the stabilised Euler-Poisson system \eqref{eq:r_EP}. 

\subsection{The Scheme}
\label{sec:scheme}
Let us consider a discretisation $0=t_0<t_1<\cdots<t_N=T$ of the time interval $(0,T)$ and let $\dt=t_{n+1}-t_n$, for $n=0,1,\dots,N-1$, be the constant timestep. We consider the following fully discrete scheme for $0\leq n\leq{N-1}$:
\begin{subequations}
\label{eq:dis_ep}
\begin{gather}
    \frac{1}{\dt}(\rho_{K}^{n+1}-\rho_{K}^{n})+\frac{1}{\left|K\right|}\sum_{\s\in\E(K)}F_{\sigma,K}^{n}=0, \ \forall K\in \M, \label{eq:dis_cons_mas}\\
    \frac{1}{\dt}(\rho_{\Ds}^{n+1}u_{\sigma}^{n+1}-\rho_{\Ds}^{n}u_{\sigma}^{n})+\frac{1}{\left|\Ds\right|}\sum_{\epsilon\in\bar\E(\Ds)}F_{\epsilon,\sigma}^{n}u_{\epsilon}^{n}+(\partial^{(i)}_{\E}p^n)_{\s}=\rho_{\s}^{n}(\partial^{(i)}_{\E}\phi^{n+1,*})_{\s}, \ \forall\s\in\E_{\mathrm{int}}^{(i)}, \, \forall 1\leq i\leq d,  \label{eq:dis_cons_mom}\\
    \veps^2(\Delta_{\M}\phi^{n+1})_{K}=\rho_{K}^{n+1}-1, \ \forall K\in \M.  \label{eq:dis_poisson}
\end{gather}
\end{subequations}
Here, the stabilised mass flux and the potential are given by 
\begin{align}
\label{eq:stabterms}
F_{\s,K}^n&=\abs{\s}(\rho_{\s}^n{u}_{\s,K}^n-Q_{\s,K}^{n+1}),\\
\phi^{n+1,*}_{K}&=\phi^{n+1}_{K}-\Lambda^{n}_{K},
\end{align}
with the choices of $Q_\s^{n+1}$ and $\Lambda^{n}_{K}$ to be made later. Using the updates \eqref{eq:dis_cons_mas}-\eqref{eq:dis_cons_mom}, the following dual mass update and the velocity update can be easily deduced. 
\begin{gather}
    \frac{1}{\dt}(\rho_{\Ds}^{n+1}-\rho_{\Ds}^{n})+\frac{1}{\left|\Ds\right|}\sum_{\epsilon\in\bar\E(\Ds)}F_{\epsilon,\sigma}^{n}=0,  \label{eq:dis_cons_mass_dual}\\
    \frac{1}{\dt}(u_{\s}^{n+1}-u_{\s}^n)+\frac{1}{|\Ds|}\sum_{\epsilon\in\bar\E(\Ds)}(F_{\epsilon,\s}^n)^{-}\frac{u_{\s^{\prime}}^{n}-u_{\s}^n}{\rho_{\Ds}^{n+1}}=\frac{1}{\rho_{\Ds}^{n+1}}\big(\rho_{\s}^{n}(\partial_{\E}^{(i)}\phi^{n+1,*})_{\s}-(\partial_{\E}^{(i)}p^n)_{\s}\big).  \label{eq:dis_vel_dual}
\end{gather}
Here, $a^{\pm} = \frac{1}{2}(a \pm \abs{a})$ denotes, respectively, the positive and negative parts of a real number $a$.
Finally, we take the initial approximation for $\rho$ and $\uu{u}$ as the average of the initial conditions $\rho_{0}$ and $\uu{u}_{0}$ on primal cells and dual cells respectively, i.e.\ 
\begin{align}
    \label{eq:dis_ic_den}
    \rho_{K}^{0}&=\frac{1}{|K|}\int_{K}\rho_{0}(\uu{x})\dd\uu{x}, \ \forall K\in\M,
    \\
    \label{eq:dis_ic_vel}
    u_{\s,i}^{0}&=\frac{1}{|\Ds|}\int_{\Ds}(\uu{u}_{0}(\uu{x}))_{i}\dd\uu{x}, \ \forall \s\in\E_\intr^{(i)}, \forall \, 1\leq i\leq d. 
\end{align}
The boundary conditions imposed on $\phi$ are discretised on the external interfaces $\s\in\cup_{i=1}^d\E^{(i)}_\extr$, $1\leq i\leq d$ via
\begin{align}
\label{eq:dis_bdary_dir}
    \phi^D_\s &= \frac{1}{\abs{\s}}\int_{\s}\phi^D(\uu{x})\dd\uu{\gamma}(\uu{x}),\;\s\subset\Gamma^D, \\
\label{eq:dis_bdary_neu}
    (\D^{(i)}_\E\phi^n)_\s &= 0,\;\s\subset\Gamma^N,\;0\leq n\leq N.
\end{align}
In all the subsequent calculations we assume that $\phi^D$ is defined on $\overline{\Omega}$ and that $\phi^D\in L^\infty(\Omega)\cap H^1(\Omega)$.

\subsection{Discrete Identities}
\label{sec:id}

The aim of this subsection is to prove apriori energy estimates satisfied by the scheme \eqref{eq:dis_ep} that are discrete counterparts of the stability estimates stated in Proposition~\ref{prop:r_engy_balance}. For the simplicity of exposition, throughout the computations carried out in this subsection, we have assumed that the velocity components vanish on the boundary whereas the potential has a homogeneous Neumann boundary condition. Analogous computations can be carried out when periodic boundary conditions are imposed on all the unknowns. The goal of the apriori energy estimation is ultimately to derive suitable expressions for the stabilisation terms $Q_\s^{n+1}$ and $\Lambda^{n}_{K}$ so that the overall scheme is energy stable. 
\begin{lemma}[Discrete renormalisation identity]
\label{lem:disc_renorm}
Let $\psi_\gm$ be defined as in \eqref{eq:psi_gamma}. Then a solution to the system \eqref{eq:dis_cons_mas}-\eqref{eq:dis_poisson} satisfies the following identity for each $K\in\M$:
\begin{equation}
    \label{eq:dis_renorm}
    \frac{|K|}{\dt}\big(\psi_\gm(\rho_{K}^{n+1})-\psi_\gm(\rho_{K}^{n})\big)+\sum_{\s\in\E(K)}\abs{\s}\psi_\gm(\rho_\s^n)
\Big(u_{\s,K}^n-\frac{Q_{\s,K}^{n+1}}{\rho^n_\s}\Big)+ \abs{K}p^n_K\bigg(\dive_\M\Big(\uu{u}^n-\frac{\uu{Q}^{n+1}}{\rho^n}\Big)\bigg)_K + \mcal{F}_K
    =\mcal{R}_K,
   \end{equation}
with
\begin{subequations}
    \begin{align}
    \label{eq:stab_div}
    &\bigg(\dive_\M\Big(\uu{u}^n-\frac{\uu{Q}^{n+1}}{\rho^n}\Big)\bigg)_K
    = \frac{1}{\abs{K}}\sum_{\s\in\E(K)}\abs{\s} \Big(u_{\s,K}^n-\frac{Q_{\s,K}^{n+1}}{\rho^n_\s}\Big),\\
    \label{eq:renorm_flux}
    &\mcal{F}_K = \sum_{\s\in\E(K)}\abs{\s}\big(\psi^\prime_\gamma(\rho^n_K)(\rho^n_\s - \rho^n_K)+\psi_\gamma(\rho^n_K)-\psi_\gm(\rho_\s^n)\big)\Big(u^n_{\s,K}-\frac{Q_{\s,K}^{n+1}}{\rho^n_\s}\Big),\\
     \label{eq:renorm_id_const_b}
     &\mcal{R}_K =\frac{\abs{K}}{2\dt}(\rho_{K}^{n+1}-\rho_{K}^{n})^2\psi_{\gm}^{\prime\prime}\big(\bar{\rho}^{n+\frac{1}{2}}_K\big),
    \end{align}
\end{subequations}
where $\bar\rho_{K}^{n+\frac{1}{2}}\in \llbracket\rho_{K}^{n+1},\rho_{K}^{n}\rrbracket$.
\end{lemma}
\begin{proof}
We first multiply the mass update \eqref{eq:dis_cons_mas} by $\abs{K}\psi_{\gm}^{\prime}(\rho_{K}^n)$. Upon performing a second order Taylor's series expansion we obtain
\begin{equation}
\label{eq:dis_renorm_Taylor}
   \frac{\abs{K}}{\dt}\big(\psi_\gm(\rho_{K}^{n+1})-\psi_\gm(\rho_{K}^{n})\big) + \psi^\prime_\gm(\rho^n_K)\sum_{\s\in\E(K)}F^n_{\sigma, K}
   = \mcal{R}_K.
\end{equation}
Rearranging the second term on the left hand side of \eqref{eq:dis_renorm_Taylor} and using the relation $\rho\psi^{\prime}_\gamma(\rho) - \psi_\gamma(\rho) = p(\rho),\;\rho>0,$ we obtain \eqref{eq:dis_renorm}.
\end{proof}
\begin{lemma}[Discrete positive renormalisation identity]
Any solution to the system \eqref{eq:dis_cons_mas}-\eqref{eq:dis_poisson} satisfies the following identity:
\begin{equation}
\label{eq:dis_po_renorm}
\begin{aligned}
    \frac{|K|}{\dt}\big(\Pi_\gm(\rho_{K}^{n+1})-\Pi_\gm(\rho_{K}^{n})\big)+\sum_{\s\in\E(K)}\abs{\s}\big(\psi_\gm(\rho_\s^n)-\psi_\gm^\prime(1)\rho_\s^n\big)\Big(\uu{u}_{\s,K}^n-\frac{Q_{\s,K}^{n+1}}{\rho^n_\s}\Big)\\
    +\abs{K}p^n_K\bigg(\dive_\M\Big(\uu{u}^n-\frac{\uu{Q}^{n+1}}{\rho^n}\Big)\bigg)_K+\mcal{F}_K=\mcal{R}_K,
   \end{aligned}
\end{equation}
where $\mcal{F}_{K},\mcal{R}_{K}$ are as in Lemma~\ref{lem:disc_renorm}. 
\end{lemma}
\begin{proof}
The proof follows from the definition of $\Pi_\gamma$ and straightforward calculations.
\end{proof}

\begin{lemma}[Discrete kinetic energy identity]
\label{lem:dis_kin_energy_loc}
Any solution to the system \eqref{eq:dis_cons_mas}-\eqref{eq:dis_poisson} satisfies the following identity for $1\leq i\leq d,\; \s\in\E_\intr^{(i)}$ and $0\leq n\leq{N-1}$:
\begin{align}
\label{eq:dis_kinid_loc}
&\frac{\abs{\Ds}}{2\dt}(\rho^{n+1}_{\Ds}(u^{n+1}_\sigma)^2 -\rho^n_{\Ds}(u^n_\sigma)^2) +\sum_{\substack{\epsilon\in\tilde{\E}(\Ds)}}F^n_{\epsilon, \sigma}\frac{(u^n_{\epsilon,\mathrm{up}})^2}{2}+\abs{\Ds}\Big(u^n_\sigma-\frac{Q_\s^{n+1}}{\rho_\s^n}\Big)(\partial^{(i)}_\E p^n)_\sigma\\& = \abs{\Ds}(\rho_\s^n u^n_\sigma-Q_\s^{n+1})(\partial^{(i)}_\E \phi^{n+1})_\sigma+\abs{\Ds}\frac{Q_\s^{n+1}}{\rho_\s^n}\big(\rho_\s^n(\partial^{(i)}_\E \phi^{n+1})_\s-(\partial^{(i)}_\E p^{n})_\s\big)\nonumber
\\
&-\abs{\Ds}\rho_\s^n u_\s^n(\partial^{(i)}_\E \Lambda^{n})_\s+\mcal{R}_{\s},\nonumber
\end{align}
where
\begin{align}
\label{eq:dis_kinloc_rem}
\mcal{R}_{\s} =\frac{\abs{\Ds}}{2\dt}\rho^{n+1}_{\Ds}\abs{u^{n+1}_\sigma - u^n_\sigma}^2
+\frac{1}{2}\sum_{\substack{\epsilon\in\tilde{\E}(\Ds)\\ \epsilon = \Ds|D_\s^{\prime}}}(F^n_{\epsilon, \sigma})^{-}(u^n_{\s^\prime} - u^n_\s)^2.
\end{align}
\end{lemma}
\begin{proof}
Multiplying the momentum balance equation \eqref{eq:dis_cons_mom} with $\abs{\Ds}u^n_\sigma$ and using the dual mass balance \eqref{eq:dis_cons_mass_dual} we readily obtain \eqref{eq:dis_kinid_loc}.
\end{proof}

\begin{lemma}[Discrete potential energy identity]
A solution to the system \eqref{eq:dis_cons_mas}-\eqref{eq:dis_poisson} satisfies the following equality for $1\leq i\leq d,\s\in\E_{int}^{(i)}$ and $0\leq n\leq{N-1}\colon$
\begin{equation}
\label{eq:dis_potbal}
\begin{aligned}    
    \frac{\veps^2}{4\dt}\sum_{i=1}^{d}\sum_{\s\in\E^{(i)}(K)}\abs{\Ds}\left(\abs{(\partial^{(i)}_\E\phi^{n+1})_{\s}}^2-\abs{(\partial^{(i)}_\E\phi^{n})_{\s}}^2\right)+\mcal{\bar{R}}_K\\-\veps^2\sum_{i=1}^{d}\sum_{\substack{\s\in\E^{(i)}(K)\\\s=K|L}}\abs{\s}\frac{\phi_K^{n+1}+\phi_L^{n+1}}{2}(\partial^{(i)}_\E(\phi^{n+1}-\phi^n))_{\s,K}
    =\phi_{K}^{n+1}\sum_{\s\in\E(K)}\abs{\s}(\rho_{\s}^{n}u_{\s,K}^{n}-Q_{\s,K}^{n+1}).
\end{aligned}
\end{equation}
Here, the non-negative remainder term $\mcal{\bar{R}}_K$ is given by
\begin{equation}
   \mcal{\bar{R}}_K=\frac{\veps^2}{4\dt}\sum_{i=1}^d\sum_{\s\in\E^{(i)}(K)}\abs{\Ds}\abs{(\partial^{(i)}_\E\phi^{n+1})_{\s}-(\partial^{(i)}_\E\phi^{n})_{\s}}^2.
\end{equation}
\end{lemma}
\begin{proof}
The result follows after multiplying the discrete mass equation by $|K|\phi_{K}^{n+1}$ and using the relation $(a-b)b=\frac{a^2}{2}-\frac{b^2}{2}-\frac{(a-b)^2}{2}$, for any two real numbers $a$ and $b$.
\end{proof}
\begin{theorem}[Total energy estimate]
\label{lem:dis_tot_energy}
A solution to the system \eqref{eq:dis_cons_mas}-\eqref{eq:dis_poisson} satisfies the global entropy inequality
\begin{align}
\label{eq:dis_totbal}
  \sum_{K\in\M}\frac{\abs{K}}{\dt}\big(\Pi_\gamma(\rho_{K}^{n+1})-\Pi_\gamma(\rho_{K}^{n})\big)&+\sum_{i=1}^{d}\sum_{\s\in\E_{int}^{(i)}}\frac{1}{2}\frac{|\Ds|}{\dt}\left(\rho_{\Ds}^{n+1}(u_{\sigma}^{n+1})^2-\rho_{\Ds}^{n}(u_{\sigma}^{n})\right)^2\\
  &+\frac{\veps^2}{2}\sum_{i=1}^d\sum_{\s\in\E^{(i)}_{int}}\frac{\abs{\Ds}}{\dt}\left(|(\partial^{(i)}_\E\phi^{n+1})_{\s}|^2-|(\partial^{(i)}_\E\phi^{n})_{\s}|^2\right) \leq 0,\nonumber
\end{align}
under the following conditions. 
\begin{enumerate}[(i)]
    \item A CFL restriction on the time-step:
    \begin{equation}
        \label{eq:cfl}
        \dt\leq\min\Big\{\frac{\rho_{\Ds}^{n+1}\abs{\Ds}}{4\sum_{\epsilon\in\bar\E(\Ds)}(-(F^n_{\epsilon,\s})^{-})}, \frac{1}{4}\Big(\frac{\rho^{n+1}_{\Ds}}{a^n_\sigma}\Big)^{\half}\Big\},
    \end{equation}
    with 
    \begin{equation}
    \label{eq:eqn_ansigma}
        a^n_\sigma = \frac{2C}{\Delta_\sigma}\frac{\abs{\sigma}}{\abs{\Ds}}(\rho^n_\sigma)^2,
    \end{equation}
    where $\Delta_\s$ is given by $\frac{1}{\Delta_\s}=\frac{1}{2}\big(\frac{|\D K|}{|K|}+\frac{|\D L|}{|L|}\big)$ for $\s=K|L$, and the constant $C>0$ being an upper bound of $\psi_\gamma^{\prime\prime}$.
    \item The following choices for the stabilisation terms $\uu{Q}^{n+1}$ and $\Lambda^{n}$:
    \begin{align}
        \label{eq:dis_Q}
        Q_{\s}^{n+1} &=\eta_\s\dt\rho_\s^n\big((\partial^{(i)}_{\E}p^n)_{\s}-\rho_{\s}^n(\partial^{(i)}_{\E}\phi^{n+1})_{\s}\big),\;\forall\s\in\E^{(i)}_\intr,\; 1\leq i \leq d, \\
        \label{eq:dis_Lambda}
        \Lambda_K^n &=-\alpha\dt\frac{C}{|K|}\sum_{\s\in\Ek}|\s|\rho_\s^n u_{\s,K}^n,\;\forall K\in\M,
    \end{align} 
\end{enumerate}
with $\eta_\s\geq\frac{1}{\rho_{\Ds}^{n+1}}$ for each $\s\in\E^{(i)}_\intr, \ 1\leq i \leq d$ and $\alpha\geq2$.
\end{theorem}
\begin{proof}
We take sum over $K\in\M$ in \eqref{eq:dis_po_renorm}. Since the second term on the left hand side of \eqref{eq:dis_po_renorm} is locally conservative, it doesn't contribute when summation over $K$ is carried out. Summing over $K\in\M$, the third term on the left hand side of \eqref{eq:dis_po_renorm} yields
\begin{equation}
\label{eq:dis_po_renorm_fk}
    \begin{aligned}
        \sum_{K\in\M}\mcal{F}_K &=\sum_{K\in\M}\sum_{\s\in\E(K)}\abs{\s}\big(\psi^\prime_\gamma(\rho^n_K)(\rho^n_\s - \rho^n_K)+\psi_\gamma(\rho^n_K)\big)\Big(u^n_{\s,K}-\frac{Q_{\s,K}^{n+1}}{\rho^n_\s}\Big) \\
        &=\sum_{K\in\M}\psi^{\prime}_\gamma(\rho^n_K)\sum_{\sigma\in\E(K)}\abs{\s}\rho^n_\s\Big(u^n_{\s,K}-\frac{Q_{\s,K}^{n+1}}{\rho^n_\s}\Big)\\
        &+\sum_{K\in\M}\big(\psi_\gamma(\rho^n_K) - \rho^n_K\psi^{\prime}_\gamma(\rho^n_K)\big)\sum_{\s\in\E(K)}\abs{\s}\Big(u^n_{\s,K}-\frac{Q_{\s,K}^{n+1}}{\rho^n_\s}\Big).
    \end{aligned}
\end{equation}
Rearranging the summands in the first term and using the identity $ \rho\psi^{\prime}_\gamma(\rho)-\psi_\gamma(\rho) = p(\rho)$ in the second term on the right hand side of $\eqref{eq:dis_po_renorm_fk}$, we further obtain
\begin{equation}
\label{eq:dis_po_flx}
    \begin{aligned}
        \sum_{K\in\M}\mcal{F}_{K} &=-\sum_{i=1}^d\sum_{\s\in\E^{(i)}_\intr}\abs{\Ds}\rho^n_\s\Big(u^n_{\s}-\frac{Q_{\s}^{n+1}}{\rho^n_\s}\Big)\big(\D^{(i)}_\E\psi^\prime_\gamma(\rho^n)\big)_\s\\
        &-\sum_{K\in\M}\abs{K}p^n_K\bigg(\dive_\M\Big(\uu{u}^n-\frac{\uu{Q}^{n+1}}{\rho^n}\Big)\bigg)_K.
    \end{aligned}
\end{equation}
The choice for $\rho^n_\s$ on each internal interface $\s=K|L$ given by Lemma~\ref{lem:rho_sig} and the discrete div-grad duality \eqref{eq:disc_dual} implies that the first term on the right hand side of \eqref{eq:dis_po_flx} is zero. Thus, the summation over all $K\in\M$ in \eqref{eq:dis_po_renorm} finally gives the identity
\begin{equation}
\begin{aligned}
\label{eq:po_renorm_glob}
    \sum_{K\in\M}\frac{\abs{K}}{\dt}\big(\Pi_\gm(\rho_{K}^{n+1})-\Pi_\gm(\rho_{K}^{n})\big)+\sum_{K\in\M}\abs{K}p^n_K\bigg(\dive_\M\Big(\uu{u}^n-\frac{Q^{n+1}}{\rho^n}\Big)\bigg)_K=\sum_{K\in\M}\mcal{R}_K.
   \end{aligned}
\end{equation}
In order to estimate the remainder term on the right hand side, we employ the mass balance \eqref{eq:dis_cons_mas} and the elementary inequality $(a+b)^2\leq 2a^2+2b^2$ and write
\begin{equation}
\label{eq:po_renorm_rmn}
\mcal{R}_K\leq C\bigg[\frac{\dt}{\abs{K}}\Big(\sum_{\s\in\E(K)}|\s|\rho^n_{\s}u^n_{\s,K}\Big)^2+\frac{\dt}{\abs{K}}\Big(\sum_{\s\in\E(K)}\abs{\s}\Q_{\s,K}^{n+1}\Big)^2\bigg], \;\forall K\in\M,
\end{equation}
where $C>0$ is an upper bound of $\psi^{\prime\prime}_\gamma$. We further estimate the second term on the right hand side of \eqref{eq:po_renorm_rmn} using the Jensen's inequality and splitting the resulting term as a sum of symmetric and anti-symmetric parts, so that for each $K\in\M$,
\begin{equation}
\label{eq:sym_split}
\begin{aligned}
    \frac{1}{\abs{K}}\Big(\sum_{\s\in\E(K)}\abs{\s}\Q_{\s,K}^{n+1}\Big)^2
    &\leq \frac{\abs{\D K}}{\abs{K}}\sum_{\s\in\E(K)}\abs{\s}(\Q_\s^{n+1})^2 \\
    &=\half\sum_{\substack{\s\in\E(K)\\\s=K|L}}\abs{\s}\Big(\frac{\abs{\D K}}{\abs{K}}+\frac{\abs{\D L}}{\abs{L}}\Big)(\Q_\s^{n+1})^2+\half\sum_{\substack{\s\in\E(K)\\\s=K|L}}\abs{\s}\Big(\frac{\abs{\D K}}{\abs{K}}-\frac{\abs{\D L}}{\abs{L}}\Big)(\Q_\s^{n+1})^2.
\end{aligned}
\end{equation}
Using the estimates \eqref{eq:po_renorm_rmn}-\eqref{eq:sym_split} in \eqref{eq:po_renorm_glob}, dropping the anti-symmetric terms and rearranging the summands leads to
\begin{equation}
\label{eq:dis_po_renorm_glob}
    \sum_{K\in\M}\frac{|K|}{\dt}\big(\Pi_\gm(\rho_{K}^{n+1})-\Pi_\gm(\rho_{K}^{n})\big)+\sum_{K\in\M}\abs{K}p^n_K(\dive_\M\uu{u}^n)_K\leq\mcal{A}+\mcal{R}+\mcal{R}^{\prime},
   \end{equation}
where
\begin{subequations}
 \begin{align}
 \mcal{A}&=\dt\sum_{K\in\M}\frac{C}{|K|}\left(\sum_{\s\in\E(K)}|\s|\rho^n_{\s} u^n_{\s,K}\right)^2,\\
 \mcal{R}&=\dt\sum_{K\in\M}\sum_{\s\in\E(K)}\frac{C}{\Delta_{\s}}\abs{\s}(\Q_{\s,K}^{n+1})^2,\\
 \mcal{R}^{\prime}&=-\sum_{i=1}^{d}\sum_{\s\in\E^{(i)}}\abs{\Ds}\Q_{\s}^{n+1}(\partial^{(i)}_{\E}\psi_{\gm}^{\prime})^{n}_{\s},
\end{align}
\end{subequations}
with $\frac{1}{\Delta_\s} = \half\Big(\frac{\abs{\D K}}{\abs{K}}+\frac{\abs{\D L}}{\abs{L}}\Big)$. Calculations analogous to the above, performed on a collocated grid, can be found in \cite{CDV17}.

Next, taking sum over all $\s\in\E^{(i)}_\intr,~i=1,2,\dots,d$, in the discrete kinetic energy identity \eqref{eq:dis_kinid_loc}, dropping the locally conservative term on the dual mesh and rearranging the summands gives
\begin{equation}
\begin{aligned}
\label{eq:sum_kinid_glob}
    &\sum_{\s\in\E_{\intr}}\frac{\abs{\Ds}}{2\dt}(\rho^{n+1}_{\Ds}(u^{n+1}_\sigma)^2 -\rho^n_{\Ds}(u^n_\sigma)^2)+\sum_{i=1}^{d}\sum_{\s\in\E_{\intr}^{(i)}}\abs{\Ds}u^n_\sigma(\partial^{(i)}_\E p^n)_\sigma \\
    &= \sum_{i=1}^{d}\sum_{\s\in\E_{int}^{(i)}}\abs{\Ds}\rho_\s^n u^n_\sigma(\partial^{(i)}_\E \phi^{n+1})_\sigma -\sum_{i=1}^{d}\sum_{\s\in\E_{int}^{(i)}}\abs{\Ds}\rho_\s^n u^n_\sigma(\partial^{(i)}_\E \Lambda^n)_\sigma +\sum_{i=1}^{d}\sum_{\s\in\E_{\intr}^{(i)}}\mcal{R}_{\s}.
\end{aligned}    
\end{equation}
In order to estimate the third term on the right hand side of \eqref{eq:sum_kinid_glob}, we square the velocity update \eqref{eq:dis_vel_dual} and use the inequality $(a+b)^2 \leq 2a^2 + 2b^2$, so that for each $\s\in\E^{(i)}_\intr$, $1\leq i\leq d$,
\begin{equation}
\begin{aligned}
\label{eq:dis_kinloc_rhs_rem}
\half\rho^{n+1}_{\Ds}\absq{u^{n+1}_\s - u^n_\s} &\leq \frac{\dt^2}{\rho^{n+1}_{\Ds}}((\partial^{(i)}_\E p^n)_\s - \rho^n_\s(\partial^{(i)}_\E\phi^{n+1})_\s)^2\\
 &+2\frac{\dt^2}{\rho^{n+1}_{\Ds}}\Big(\frac{1}{\abs{\Ds}}\sum_{\epsilon\in\E(\Ds)}(F^n_{\epsilon,\sigma})^{-}(u^n_{\s^{\prime}} - u^n_\s)\Big)^2+2\dt^2\frac{(\rho^n_\s)^2}{\rho^{n+1}_{\Ds}}(\partial^{(i)}_\E\Lambda^n)_\s^2.
\end{aligned}
\end{equation}
Estimating the remainder term $\mcal{R}_{\s}$, cf.\ \eqref{eq:dis_kinloc_rem}, using \eqref{eq:dis_kinloc_rhs_rem} and applying the Cauchy-Schwartz inequality to the second term on the right hand side of \eqref{eq:dis_kinloc_rhs_rem} analogously as done in \cite[Lemma 3.1]{DVB17}, the equation \eqref{eq:sum_kinid_glob} finally yields
\begin{equation}
\begin{aligned}
\label{eq:dis_kinid_est}
 &\sum_{\s\in\E_{int}}\frac{|\Ds|}{2\dt}\left(\rho_{\Ds}^{n+1}(u_{\sigma}^{n+1})^2-\rho_{\Ds}^{n}(u_{\sigma}^{n})^2\right)+\sum_{i=1}^{d}\sum_{\s\in\E_{int}^{(i)}}|\Ds|u_{\s}^{n}(\partial^{(i)}_{\mcal{E}}p^n)_{\sigma}\leq \sum_{i=1}^{d}\sum_{\s\in\E_{int}^{(i)}}|\Ds|\rho_{\s}^{n}u_{\s}^{n}(\partial^{(i)}_{\mcal{E}}\phi^{n+1})_{\s}\\
&+\sum_{i=1}^d\sum_{\s\in\E^{(i)}_\intr}\bigg(\half + \frac{2\dt}{\rho^{n+1}_{\Ds}\abs{\Ds}}(F^n_{\epsilon,\s})^{-}\bigg)\sum_{\epsilon\in\tilde{\E}(\Ds)}(F^n_{\epsilon,\s})^{-}(u^n_{\s^{\prime}} - u^n_\s)^2+\tilde{\mcal{A}}+\tilde{\mcal{R}}+\mcal{A}^{\prime}.
\end{aligned}
\end{equation}
Here,
\begin{subequations}
\begin{align}
\mcal{A}^{\prime}&=4\dt\sum_{K\in\M}(\Lambda_K^n)^2\sum_{\s\in\Ek}\frac{|\s|^2}{|\Ds|}\frac{(\rho_\s^n)^2}{\rho_{\Ds}^{n+1}},\\
\tilde{\mcal{R}}&=\dt\sum_{i=1}^{d}\sum_{\s\in\E_{int}^{(i)}}|\frac{|\Ds|}{\rho_{\Ds}^{n+1}}\left((\partial^{(i)}_{\E}p^n)_{\s}-\rho_{\s}^n(\partial^{(i)}_{\E}\phi^{n+1})_{\s}\right)^2,\\
\tilde{\mcal{A}}&=\sum_{K\in\M}\Lambda_K^n\sum_{\s\in\Ek}|\s|\rho_{\s}^{n}u_{\s,K}^{n}.
\end{align}
\end{subequations}

Next, taking sum over all $K\in\M$ in the potential balance \eqref{eq:dis_potbal}, dropping the locally conservative third term on the right hand side and using the discrete div-grad duality \eqref{eq:disc_dual} yields
\begin{equation}
\label{eq:dis_potbal_glob}
\begin{aligned}    
    \frac{\veps^2}{2\dt}\sum_{i=1}^{d}\sum_{\s\in\E^{(i)}}\abs{\Ds}\left(\abs{(\partial^{(i)}_\E\phi^{n+1})_{\s}}^2-\abs{(\partial^{(i)}_\E\phi^{n})_{\s}}^2\right)&+\sum_{K\in\M}\sum_{\s\in\E(K)}|\s|\rho_{\s}^{n}u_{\s,K}^{n}\phi_{K}^{n+1}\\
 &\leq\sum_{i=1}^{d}\sum_{\s\in\E^{(i)}}|\Ds|\Q_{\s}^{n+1}(\D^{(i)}_{\E}\phi^{n+1})_{\s}.
\end{aligned}
\end{equation}
The above inequality stems from the identity \eqref{eq:dis_potbal}, owing to the fact that $\mcal{\bar{R}}_K\geq 0$ for each $K\in\M$. 

Finally, adding \eqref{eq:dis_po_renorm_glob}, \eqref{eq:dis_kinid_est} and \eqref{eq:dis_potbal_glob} and using the discrete div-grad duality \eqref{eq:disc_dual} we obtain the inequality
\begin{equation}
\label{eq:dis_tot_est}
\begin{aligned}
    &\sum_{K\in\M}\frac{\abs{K}}{\dt}\big(\Pi_\gamma(\rho_{K}^{n+1})-\Pi_\gamma(\rho_{K}^{n})\big)+\sum_{i=1}^{d}\sum_{\s\in\E_{int}^{(i)}}\frac{1}{2}\frac{|\Ds|}{\dt}\left(\rho_{\Ds}^{n+1}(u_{\sigma}^{n+1})^2-\rho_{\Ds}^{n}(u_{\sigma}^{n})^2\right)\\
    &+\frac{\veps^2}{2}\sum_{i=1}^d\sum_{\s\in\E^{(i)}_\intr}\frac{\abs{\Ds}}{\dt}\left(|(\partial^{(i)}_\E\phi^{n+1})_{\s}|^2-|(\partial^{(i)}_\E\phi^{n})_{\s}|^2\right)\leq \hat{\mcal{A}} + \hat{\mcal{R}} + \mcal{S},
\end{aligned}
\end{equation}
where
\begin{align}
     \hat{\mcal{A}} &= C\dt\sum_{K\in\M}\frac{1}{|K|}\bigg(\sum_{\s\in\Ek}|\s|\rho_{\s}^{n}u_{\s,K}^n\bigg)^2\big(q^n_K\alpha^2-\alpha+1\big),\\
    \hat{\mcal{R}} &= \dt\sum_{i=1}^{d}\sum_{\s\in\E_{int}^{(i)}}|\Ds|\left((\partial_{\E}^{(i)}p^n)_{\s}-\rho_\s^n(\partial_{\E}^{(i)}\phi^{n+1})_{\s}\right)^2\bigg(m_\s^n\eta_\s^2-\eta_\s+\frac{1}{\rho_{\Ds}^{n+1}}\bigg), \\
     \mcal{S} &= \sum_{i=1}^d\sum_{\s\in\E^{(i)}_\intr}\bigg(\half + \frac{2\dt}{\rho^{n+1}_{\Ds}\abs{\Ds}}(F^n_{\epsilon,\s})^{-}\bigg)\sum_{\epsilon\in\E(\Ds)}(F^n_{\epsilon,\s})^{-}(u^n_{\s^{\prime}} - u^n_\s)^2,
\end{align}
provided the stabilisation terms  $\uu{Q}^{n+1}$ and $\Lambda^{n}$ are chosen as in \eqref{eq:dis_Q} and \eqref{eq:dis_Lambda}, respectively. Here we introduce the notations $m_\s^n=\dt^2 a^n_\sigma$, for each $\s\in\E^{(i)}_\intr$, $1\leq i\leq d$, with $a^n_\sigma$ given by \eqref{eq:eqn_ansigma} and $q^n_K=4\dt^2\frac{C}{\abs{K}}\sum_{\s\in\Ek}\frac{\absq{\s}}{\abs{\Ds}}\frac{(\rho_\s^n)^2}{\rho_{\Ds}^{n+1}}$, for each $K\in\M$. In order to obtain the required entropy inequality \eqref{eq:dis_totbal}, the following conditions must hold in \eqref{eq:dis_tot_est}:
\begin{align}
    \label{eq:tstep_ke}
    \half + \frac{2\dt}{\rho^{n+1}_{\Ds}\abs{\Ds}}(F^n_{\epsilon,\s})^{-} &\geq 0,\;\forall\s\in\E^{(i)}_\intr,\; 1\leq i\leq d,
    \\
    \label{eq:Q_quad}
    m_\s^n\eta_\s^2-\eta_\s+\frac{1}{\rho_{\Ds}^{n+1}} &\leq 0,\;\forall\s\in\E^{(i)}_\intr,\; 1\leq i\leq d,
    \\
    \label{eq:Lambda_quad}
    q^n_K\alpha^2-\alpha+1 &\leq 0,\;\forall K\in\M.
\end{align}
Note that time timestep restriction 
\begin{equation}
\label{eq:tstep_ke_suff}
    \dt\leq \frac{\rho_{\Ds}^{n+1}\abs{\Ds}}{4\sum_{\epsilon\in\bar\E(\Ds)}(-(F^n_{\epsilon,\s})^{-})}
\end{equation}
ensures the inequality \eqref{eq:tstep_ke}, whereas the timestep restriction
\begin{equation}
\label{eq:tstep_quad_suff}
    \dt\leq\frac{1}{4}\Big(\frac{\rho^{n+1}_{\Ds}}{a^n_\sigma}\Big)^{\half}
\end{equation}
yields the existence of real solutions for $\eta_\s$ and $\alpha$ in the quadratic expressions \eqref{eq:Q_quad} and \eqref{eq:Lambda_quad}, respectively. Furthermore, the timestep restriction \eqref{eq:tstep_quad_suff} along with the choices $\eta_\sigma \geq \frac{1}{\rho^{n+1}_{\Ds}}$, for each $\s\in\E^{(i)}_\intr$, $1\leq i\leq d$, and $\alpha = 2$ clearly shows that the inequalities \eqref{eq:Q_quad} and \eqref{eq:Lambda_quad} are satisfied.  
\end{proof}
\begin{remark}
  The choice for $\eta_\sigma$ in (ii) in the above theorem is implicit in nature. However, we observe that the following implicit time-step restriction 
  \begin{equation}
    \dfrac{\dt}{|\Ds|}\sum_{\epsilon\in\tilde\E(\Ds)}
    \frac{\abs{F^n_{\epsilon,\s}}}{\rho_{\Ds}^{n+1}}\leq\frac{1}{4}
    \label{eq:suftime}  
  \end{equation}
  gives a sufficient condition which, along with the time-step restriction \eqref{eq:tstep_quad_suff}, yields the required CFL condition \eqref{eq:cfl}. From \eqref{eq:suftime} and the dual mass balance
  \eqref{eq:dis_cons_mass_dual}, we deduce that
  \begin{equation}
    \frac{5}{4}\rho_{\Ds}^{n+1} - \rho_{\Ds}^n \geq  \rho_{\Ds}^{n+1}
    - \rho_{\Ds}^n +
    \frac{\dt}{\abs{\Ds}}\sum_{\epsilon\in\tilde\E(\Ds)}\abs{F^n_{\epsilon,\s}}\geq
    0. 
  \end{equation}
  Hence, we get
  \begin{equation}
    \label{eq:rho_5/4}
    \frac{\rho_{\Ds}^n }{\rho_{\Ds}^{n+1}}\leq \frac{5}{4}.
  \end{equation}
  Therefore, at each interface $\sigma$, choosing $\eta_\sigma$ such that
  \begin{equation}
  \label{eq:eta_expl}
    \eta_\sigma=\eta/\rho_{\Ds}^n  \mathrm{with}\;\eta>\frac{5}{4},
  \end{equation}
  will guarantee the condition (i) and hence the stability of the scheme. In other
  words, the value of $\eta_\sigma$ can be obtained explicitly. Analogous considerations can also be found in \cite{AGK22, DVB20}.   
\end{remark}

\subsection{Existence of a Numerical Solution}
\label{sec:exist_soln}

Note that in view of the expression for $\uu{Q}^{n+1}$ from Theorem~\ref{lem:dis_tot_energy}, the mass update \eqref{eq:dis_cons_mas} and the Poisson equation \eqref{eq:dis_poisson} are coupled. In practice, we replace the linearly implicit density term $\rho^{n+1}$ in \eqref{eq:dis_cons_mas} in terms of $\phi^{n+1}$ using the Poisson equation \eqref{eq:dis_poisson} and obtain a modified discrete elliptic problem for $\phi^{n+1}$. Once $\phi^{n+1}$ is obtained after solving this elliptic problem, the mass update and the momentum update can be explicitly evaluated to get the density $\rho^{n+1}$ and velocity $\uu{u}^{n+1}$. In what follows, we establish the existence of a discrete solution to the numerical scheme \eqref{eq:dis_ep}.  
\begin{theorem}
    \label{thm:existence}
    Let $(\rho^n,\uu{u}^n,\phi^n)\in
    L_{\mcal{M}}(\Omega)\times\uu{H}_{\mcal{E},0}(\Omega)\times L_{\mcal{M}}(\Omega)$. Then there exists a solution $(\rho^{n+1},\uu{u}^{n+1},\phi^{n+1})\in L_{\mcal{M}}(\Omega)\times\uu{H}_{\mcal{E},0}(\Omega)\times L_{\mcal{M}}(\Omega)$ of the scheme \eqref{eq:dis_ep}.
\end{theorem}
\begin{proof}
    Eliminating the implicit term $\rho^{n+1}_K$ between \eqref{eq:dis_cons_mas} and \eqref{eq:dis_poisson} yields the discrete elliptic equation
    \begin{equation}
    \label{eq:reform_poisson}
   -\frac{1}{\abs{K}}\sum_{\s\in\E(K)}\abs{\s}\big(\veps^2+\eta_\s\dt^2(\rho_\s^n)^2\big)(\partial^{(i)}_{\E}\phi^{n+1})_{\s,K}=\hat{\rho}^{n}_{K},\;K\in\M,
\end{equation}
where
\begin{equation*}
    \hat{\rho}^{n}_{K}=-\rho_{K}^{n}+1+\frac{\dt}{\abs{K}}\sum_{\s\in\E(K)}\abs{\s}(\rho_{\s}^n u_{\s,K}^n-\eta_\s\dt\rho_\s^n(\partial^{(i)}_{\E}p^n)_{\s,K}).
\end{equation*}
In order to simplify the presentation, we assume a mixed boundary condition of the type \eqref{eq:cont_bdary} with $\uu{\gamma}(\Gamma^D) > 0$ and show that \eqref{eq:reform_poisson}, coupled with the boundary conditions \eqref{eq:dis_bdary_dir}-\eqref{eq:dis_bdary_neu}, is indeed a linear system for the unknown $\phi^{n+1}$. The values of $\rho^n_\sigma$ and $\eta_\sigma$ on any external edge $\s\in\E_\extr$ are defined according to Lemma~\ref{lem:rho_sig} and Theorem~\ref{lem:dis_tot_energy} after appropriate boundary conditions on $\rho$ are applied. The linear system can be obtained as
\begin{equation}
    \mathbb{A}^{n}\uu{\phi}^{n+1}=\uu{b}^{n},
\end{equation}
where $\uu{\phi}^{n+1}=(\phi_K^{n+1},\phi_{\s}^{n+1})_{K\in\M, \s\in\E^{D}}$, $\uu{b}^{n}=(\hat{\rho}_K^{n},\phi^{D}_\s)_{K\in\M, \s\in\E^{D}}$ and $\mbb{A}^n\in\mbb{R}^{\#(\M\cup\E^{D})\times\#(\M\cup\E^{D})}$ is a sparse symmetric matrix defined by
   \begin{align}
    \mbb{A}^n_{K,K_\s} &=-\frac{1}{\abs{K}}\frac{\abs{\s}^2}{\abs{\Ds}}(\veps^2+\eta_\s\dt^2(\rho_\s^n)^2), \;K\in\M,\; \s\in\E(K) \\
    \mbb{A}^n_{K,K} &=\frac{1}{\abs{K}}\sum_{\s\in\Ek}\frac{\abs{\s}^2}{\abs{\Ds}}(\veps^2+\eta_\s\dt^2(\rho_\s^n)^2), \; K\in\M,\\
    \mbb{A}^n_{\s,\s} &=1, \ \mbb{A}^n_{\s,\ell}=0 \; \text{ if } \ell\neq\s, \;\s\in\E^D.
   \end{align} 
In the notations introduced above, $K_\s = L$ if $\s\in\E_\intr$, $\s= K|L$ and $K_\s = \s$ if $\s\in\E(K)\cap\E^D$, whereas $\ell$ denotes any arbitrary index in $\M\cup\E^D$. For any $K\in\M$ we have  
\begin{equation}
\label{eq:diagdom_K}
    \mbb{A}^n_{K,K} = \sum_{\s\in\E(K)}\abs{\mbb{A}^n_{K,K_\s}},
\end{equation}
and for any $\s\in\E(K)\cap\E^D$ we trivially have
\begin{equation}
\label{eq:diagdom_sig}
\mbb{A}^n_{\s,\s}>\sum_{\ell\in\E(K)\cap\E^D}\abs{\mbb{A}^n_{\s,\ell}} + \abs{\mbb{A}^n_{\s,K}},
\end{equation}
thanks to the fact that $\eta_\s\dt^2(\rho^n_\s)^2\geq0$, on all the interfaces $\s\in\E$, which makes the matrix $\mbb{A}^n$ diagonally dominant. We further notice that 
\begin{equation}
\label{eq:diagdom}
\mbb{A}^n_{k,k} > 0\;\mathrm{and} \;\mbb{A}^n_{k,\ell} \leq 0\; \mathrm{for\;each}\; k, \ell\in\M\cup\E^D,\; k \neq\ell.    
\end{equation}
From \eqref{eq:diagdom_K}-\eqref{eq:diagdom} we conclude that $\mbb{A}^n$ is an M-matrix \cite[Criterion 4.18]{Hac17}, which guarantees the existence of a solution $\phi^{n+1}=\sum_{K\in\M}\phi^{n+1}_K\mcal{X}_K\in L_\M(\Omega)$. Subsequently, $\rho^{n+1}$ and $\uu{u}^{n+1}$ can be explicitly evaluated from \eqref{eq:dis_cons_mas} and \eqref{eq:dis_cons_mom}.
\end{proof}

\section{Weak Consistency of the Scheme}
\label{sec:weak_cons}

Goal of this section is to show a Lax-Wendroff-type weak consistency of the scheme \eqref{eq:dis_ep} which is essentially proving the consistency of the numerical solution with a weak solution of the Euler-Poisson system when the mesh parameters tend to zero. To this end, we consider an initial data $(\rho^{\veps}_0,\uu{u}^{\veps}_0)\in L^\infty(\Omega)^{1+d}$ and recall the following definition of a weak solution. 
\begin{definition}
\label{def:weak_soln}
A triple $(\rho^\veps,\uu{u}^\veps, \phi^{\veps})\in L^\infty(Q_T)\times L^\infty(Q_T)^d\times L^\infty(0,T;H^1(\Omega))$ is a weak solution to the Euler-Poisson system \eqref{eq:cons_mas}-\eqref{eq:poisson} with initial-boundary conditions \eqref{eq:eq_ic} if $\rho^\veps>0$ a.e.\ in $Q_T$ and the following identities holds:  
\begin{gather}
    \int_0^T\int_\Omega\left(\rho^\veps\D_t\phi+\rho^\veps \uu{u}^\veps\cdot\bgrd \psi\right)\dd\uu{x}\dd t
     =-\int_\Omega\rho^\veps_0\psi(0,\cdot)\dd\uu{x} \label{eq:weak_soln_mas} ,\ \text{for all } \psi \in C_c^\infty([0,T)\times\bar{\Omega}),\\
     \int_0^T\int_\Omega\left(\rho^\veps \uu{u}^\veps\cdot\D_t\uu{\psi}+(\rho^\veps \uu{u}^\veps\otimes\uu{u}^\veps):\bgrd \uu{\psi}\right)\dd\uu{x}\dd t+\int_0^T\int_{\Omega}p^\veps\dive\uu{\psi}\;\dd\uu{x}\dd t
     =\int_0^T\int_\Omega\rho\bgrd{\phi}\cdot\uu{\psi}\dd\uu{x}\dd t\nonumber\\
     -\int_{\Omega}\rho^\veps_0 \uu{u}^\veps_0\cdot\uu{\psi}(0,\cdot)\dd\uu{x}, \ \text{for all } \uu{\psi} \in C_c^\infty([0,T)\times\bar{\Omega})^d, \label{eq:weak_soln_mom}\\
    -\veps^2\int_0^T\int_\Omega\bgrd\phi\cdot\bgrd \psi\dd\uu{x}\dd t
     =\int_0^T\int_\Omega(\rho^\veps-1)\psi\dd\uu{x}\dd t \label{eq:weak_soln_poisson},\ \text{for all } \psi \in C_c^\infty([0,T)\times\bar{\Omega}).
    \end{gather}
\end{definition}
The following theorem is a Lax-Wendroff-type consistency formulation for the semi-implicit scheme \eqref{eq:cons_mas}-\eqref{eq:poisson}. In the theorem, we take the liberty to suppress $\veps$ wherever necessary for the sake of simplicity of writing.  
\begin{theorem}
\label{thm:weak_cons}
Let $\Omega$ be an open bounded set of $\mbb{R}^d$. Assume that $\big(\mcal{T}^{(m)},\delta t^{(m)}\big)_{m\in\mbb{N}}$ is a sequence of discretisations such that both $\lim_{m\rightarrow +\infty}\delta t^{(m)}$ and $\lim_{m\rightarrow +\infty}h^{(m)}$ are $0$. Let $\big(\rho^{(m)},u^{(m)},\phi^{(m)}\big)_{m\in\mbb{N}}$ be the corresponding sequence of discrete solutions with respect to an initial data $(\rho^\veps_0,u^\veps_0,\phi^\veps_{0})\in L^\infty(\Omega)^{d+2}$. We assume that $(\rho^{(m)},u^{(m)}, \phi^{(m)})_{m\in\mbb{N}}$ satisfies the following:
\begin{enumerate}[(i)]
\item $\big(\rho^{(m)},u^{(m)}\big)_{m\in\mbb{N}}$ is uniformly bounded in $L^\infty(Q)^{1+d}$, i.e.\ 
\begin{align}
\ubar{C}<(\rho^{(m)})^n_K &\leq \bar{C}, \ \forall K\in\mcal{M}^{(m)}, \ 0\leq n\leq N^{(m)}, \ \forall m\in\mbb{N}\label{eq:dens_abs_bound}, \\
|(u^{(m)})^n_\sigma| &\leq C, \ \forall \sigma\in\mcal{E}^{(m)}, \ 0\leq n\leq N^{(m)}, \ \forall m\in\mbb{N}\label{eq:u_abs_bound}. 
\end{align}
where $\ubar{C}, \bar{C}, C>0$ are constants independent of discretisations. 
\item $\big(\rho^{(m)},u^{(m)}, \phi^{(m)}\big)_{m\in\mbb{N}}$ converges to $(\rho^\veps,\uu{u}^\veps, \phi^\veps)\in L^\infty(0, T; L^\infty(\Omega)\times L^\infty(\Omega)^d\times H^1(\Omega))$ in $L^r(Q_T)^{1+d+1}$ for $1\leq r<\infty$.
\end{enumerate}
We also assume that the sequence of discretisations $\big(\mcal{T}^{(m)},\delta t^{(m)}\big)_{m\in\mbb{N}}$ satisfies the mesh regularity conditions:
\begin{equation}
\label{eq:CFL_meshpar}
\frac{\delta t^{(m)}}{\min_{K\in\mcal{M}^{(m)}}\abs{K}}\leq\theta,\ \max_{K \in \mcal{M}^{(m)}} \frac{\diam(K)^2}{\abs{K}}\leq\theta,\;\forall m\in\mbb{N},
\end{equation}
where $\theta>0$ is independent of discretisations. 
Then $(\rho^\veps,u^\veps,\phi^\veps)$ satisfies the weak formulation \eqref{eq:weak_soln_mas}-\eqref{eq:weak_soln_poisson}.
\end{theorem}
\begin{proof}
Our approach follows analogous lines as in \cite[Lemma 4.1]{GHL22} and hence we skip most of the calculations except the ones related to the velocity stabilisation terms. Proceeding as in \cite[Lemma 4.1]{GHL22}, we multiply the mass update \eqref{eq:dis_cons_mas} by $\psi^n_K\abs{K}$, where $\psi^n_K$ denotes the extrapolation of the smooth and compactly supported test function on $\psi$ at $(t^n,\uu{x}_K)$. For the sake of of simplicity, we choose $m$ large enough such that $(\psi^{(m)})^n_K = 0$ for all $K\in\M^{(m)}$ that has an edge lying on the boundary. By an abuse of notation, we denote by $\M$ the collection of all the primal cells $K$ such that $K\cap\partial\Omega = \emptyset$ in the subsequent analysis. We obtain the following remainder term $R^{(m)}_1$ from the stabilisation terms in the mass update \eqref{eq:dis_cons_mas} upon summing over the space and time discretisation parameters
\begin{align}
R^{(m)}_1&=\sum_{n=0}^{N^{(m)}-1}\dt^n\sum_{K\in\M}\diam(K)\sum_{\s\in\E(K)}\abs{\s}\abs{Q^{n+1}_\sigma}\\
&\leq R_{1,1}^{(m)}+R_{1,2}^{(m)}, \nonumber
\end{align}
where
\begin{align}
\label{eq:pr_grd_trunc}
R_{1,1}^{(m)}&=C(\psi, \bar{C})\sum_{n=0}^{N^{(m)}-1}\dt^n\sum_{K\in\M}\diam(K)\sum_{\s\in\E(K)}\abs{\s}\eta\dt^n\frac{\abs{\s}}{\abs{\Ds}}\abs{\rho^n_{L}-\rho^n_{K}},\\
\label{eq:source_trunc}
R_{1,2}^{(m)}&=C(\psi, \bar{C})\sum_{n=0}^{N^{(m)}-1}\dt^n\sum_{K\in\M}\diam(K)\sum_{\s\in\E(K)}|\s|\eta\dt^n\frac{\abs{\s}}{\abs{\Ds}}\abs{\phi^{n+1}_{L}-\phi^{n+1}_{K}},
\end{align}
where $C(\psi, \bar{C})>0$ is a constant depending only on $\psi$ and $\bar{C}$. Following similar techniques as introduced in \cite[Section 4]{GHL19} to study the convergence of discrete space translates, it can be shown that the right hand sides of \eqref{eq:pr_grd_trunc} and \eqref{eq:source_trunc} tends to $0$ as $m\rightarrow \infty$ under the uniform boundedness assumptions \eqref{eq:dens_abs_bound} in (i), strong convergence assumptions detailed in (ii) and the regularity conditions on the mesh parameters \eqref{eq:CFL_meshpar}; see also \cite[Lemma A.1]{GHL22} for a similar treatment on discrete space translates.

In order to obtain the consistency of the momentum update, we multiply \eqref{eq:dis_cons_mom} by $\psi^n_\s\abs{\Ds}$ where $\psi^n_\s$ denotes the point value of the extrapolation of the compactly supported, smooth test function $\uu{\psi}$ at $(t^n,\uu{x}_\s)$. The weak consistency of the discrete convection operator and the pressure gradient in the momentum update \eqref{eq:dis_cons_mom} follows under the given assumptions using a similar method as in the proof of \cite[Theorem 4.1]{HLN+23}. The residual terms appearing due to the velocity stabilisation again converges to $0$ by a similar argument as in the case of the mass equation. It only remains to establish the consistency of the right hand side term in the momentum update \eqref{eq:dis_cons_mom}. Upon summation over all $\s\in\E^{(i)}_\intr$, $1\leq i\leq d$ and all $n\in\{0,1,\dots,N^{(m)-1}\}$, the source term of \eqref{eq:dis_cons_mom} yields the remainders
\begin{align}
T^{(m)}&=\sum_{n=0}^{N^{(m)}-1}\dt^n\sum_{i=1}^d\sum_{\s\in\E^{(i)}_\intr}\rho_\s^n\psi_\s^n(\partial^{(i)}_\E \phi^{n+1})_\s\abs{\Ds},\\
R_2^{(m)}&=-\sum_{n=0}^{N^{(m)}-1}\dt^n\sum_{K\in\M}\Lambda_K^n\sum_{\s\in\E(K)}|\s|\rho_\s^{n}\psi_\s^n.
\end{align}
Note that by assuming the uniform boundedness \eqref{eq:dens_abs_bound} of $(\rho^{(m)})_{m\in\mbb{N}}$, the discrete Poisson equation \eqref{eq:dis_poisson} with the boundary conditions \eqref{eq:dis_bdary_dir}-\eqref{eq:dis_bdary_neu} yields  a uniform bound for the sequence of discrete solutions $(\phi^{(m)})_{m\in\mbb{N}}$ upon carrying out a similar analysis as in \cite[Lemma A1]{CH+21} using an M-matrix argument. More precisely we can obtain
\begin{equation}
\label{eq:ub_phi}
    \norm{\phi^{(m)}}_{L^\infty(\Omega)}\leq C(\Omega),\;\forall m\in\mbb{N},
\end{equation}
where $C(\Omega)>0$ is a constant that depends only on the domain. Moreover, a similar analysis as in \cite[Lemma 3.5]{CH+21}, suitably adopted for a MAC grid, gives the following discrete $H^1$-seminorm estimate for any solution $\phi^n$ of \eqref{eq:dis_poisson}
\begin{equation}
 \label{eq:ub_phi_hsem}
 \sum_{i=1}^d\sum_{\s\in\E^{(i)}}\abs{\Ds}(\D^{(i)}_\E \phi^n)^2_\s\leq C, \ 0\leq n\leq N^{(m)},\;\forall m\in\mbb{N},
\end{equation}
where $C>0$ is a constant independent of the discretisation, owing to \eqref{eq:dens_abs_bound}, \eqref{eq:ub_phi} and the assumption that $\phi^D\in H^1(\Omega)$. Performing a similar analysis as in \cite[Proposition 4.5]{CH+21}, using \eqref{eq:ub_phi}, \eqref{eq:ub_phi_hsem}, the convergence assumption (ii) and the uniqueness of weak-* limits we obtain 
\begin{equation}
\label{eq:weakstar_grad}
    \bgrd_{\E^{(m)}}\phi^{(m)}\overset{\ast}{\rightharpoonup} \bgrd\phi^\veps\;\text{in}\;L^\infty(0, T; L^2(\Omega)^d) \ \text{as}\;m\rightarrow\infty.
\end{equation}
Following this observation we conclude that
\begin{equation}
    \lim_{m\rightarrow\infty}T^{(m)} = \int_0^T\int_\Omega\rho\bgrd{\phi}\cdot\uu{\psi}\dd\uu{x}\dd t.
\end{equation}
Moreover, the weak-* convergence \eqref{eq:weakstar_grad} further yields that $\phi^\veps$ satisfies \eqref{eq:weak_soln_poisson}; see \cite[Proposition 4.5]{CH+21} for an analogous result with detailed calculations.
Finally, the residual term $R^{m}_{2}$ tend to zero as $m$ tends to $+\infty$, given the choice of $\Lambda^n_K$ from \eqref{eq:dis_Lambda} and following similar techniques of approximating space translates as in \cite[Lemma A.1]{GHL22}.
\end{proof}

\section{Quasi-neutral Limit}
\label{sec:quas_lim}
In this section, we formally derive the quasi-neutral limit of the semi-implicit scheme \eqref{eq:dis_ep}. We show that the formal $\veps\to 0$ limit of the scheme \eqref{eq:dis_ep} is a semi-implicit scheme for the momentum stabilised incompressible Euler system; see also \cite{ACS22, DLV08, Deg13} for analogous treatments. We start by introducing the notion of a well-prepared initial data.
\begin{definition}
\label{def:wp_id}
An initial datum $(\rho^\veps_0,\uu{u}^\veps_0)\in L^\infty(\Omega)^{1+d}$, $\rho^\veps>0$ is called well-prepared if it satisfies
\begin{equation}
\label{eq:wp_id}
\norm{\uu{u}^\veps_0}_{L^2(\Omega)^d} + \frac{1}{\veps^2}\norm{\rho^\veps_0 - 1}_{L^\infty(\Omega)}\leq C,
\end{equation}
where $C>0$ is a constant independent of $\veps$.
\end{definition}
\begin{proposition}
\label{prop:qn_formal}
Let $(\rho^\veps, \uu{u}^\veps, \phi^\veps)_{\veps>0}\in L_\M(\Omega)\times\uu{H}_\E(\Omega)\times L_\M(\Omega)$ be a sequence of discrete solutions of the semi-implicit scheme \eqref{eq:dis_cons_mas}-\eqref{eq:dis_poisson} with respect to a corresponding sequence of initial data that satisfies \eqref{eq:wp_id}. Let $(\rho,\uu{u},\phi)\in L_\M(\Omega)\times\uu{H}_\E(\Omega)\times L_\M(\Omega)$ be a formal limit of the sequence $(\rho^\veps, \uu{u}^\veps, \phi^\veps)_{\veps>0}$ as $\veps\rightarrow 0$. Then $(\rho,\uu{u},\phi)$ satisfies the following scheme for each $0\leq n < N$:
\begin{gather}
    (\dive_\M(\uu{u}^n-\uu{Q}^{n+1})_{K}=0,\; \forall K \in \mcal{M},
    \label{eq:dis_incomp_mas}
    \\
    \frac{1}{\dt}\big(u_\s^{n+1}-u_\s^{n}\big)+\frac{1}{\left|\Ds\right|}\sum_{\epsilon\in\tilde{\E}(\Ds)}F_{\epsilon,\sigma}(1,\uu{u}^{n}-\uu{Q}^{n+1})u_{\epsilon,\mathrm{up}}^{n}+(\partial^{(i)}_{\E}\phi^{n+1}_*)_{\sigma} = 0, \ 1\leq i\leq d, \ \forall \s\in\E_\intr^{(i)}, \label{eq:dis_incomp_mom}
    \\
    \rho^{n+1}_K = 1,\; \forall K\in\M,\label{eq:const_dens_incomp}
\end{gather}
where the correction $Q^{n+1}_\sigma = -\eta\dt(\D_\E^{(i)}\phi^{n+1})_\s$ for $\s\in\Eint^{(i)}$ with a constant $\eta>0$ defined as in \eqref{eq:eta_expl} and $\phi^{n+1}_* = -\sum_{K\in\M} (\phi^{n+1}_K - (\dive_\M~\uu{u}^n)_K)\mcal{X}_K\in L_\M(\Omega)$. 
\end{proposition}
\begin{remark}
Note that the updates \eqref{eq:dis_incomp_mas}-\eqref{eq:const_dens_incomp} gives a semi-implicit scheme for the stabilised incompressible Euler system
\begin{subequations}
\label{eq:vel_stab_incomp}
\begin{align}
    \label{eq:vstab_incomp_vel}
    \D_t \uu{U} + \bdive~(\uu{U}\otimes(\uu{U}-\uu{Q})) + \bgrd(\Phi - \Lambda) &= \uu{0}, 
    \\
    \label{eq:vstab_incomp_div}
    \dive\;(\uu{U} - \uu{Q}) &= 0,
\end{align}
\end{subequations}
with $\uu{Q} = \eta\bgrd\Phi$ and $\Lambda = \dive\;\uu{U}$. 
\end{remark}
\begin{proof}
 From the discretisation of the initial data \eqref{eq:dis_ic_den} and well prepared condition \eqref{eq:wp_id} we have
\begin{equation}
\label{eq:dis_wp_qn}
    \abs{\rho^{\veps, 0}_K - 1} \leq C\veps^2,\; \forall K\in\M.
\end{equation}  
Upon taking a formal limit of the modified elliptic problem \eqref{eq:reform_poisson} with a constant timestep $\dt$ and following \eqref{eq:dis_wp_qn} we obtain
\begin{equation}
\label{eq:dis_incomp_reform}
    -\frac{\eta\dt}{\abs{K}}\sum_{\s\in\E(K)}\abs{\s}(\partial^{(i)}_{\E}\pi^{1})_{\s,K}= \frac{1}{\abs{K}}\sum_{\s\in\E(K)}\abs{\s}u^0_{\s,K},\;\forall K\in\M,
\end{equation}
which gives \eqref{eq:dis_incomp_mas} for $n = 1$. Subsequently, taking the formal limits of the updates \eqref{eq:dis_cons_mas}, \eqref{eq:dis_cons_mom} and using \eqref{eq:dis_incomp_reform} we obtain \eqref{eq:dis_incomp_mom}-\eqref{eq:const_dens_incomp} for $n=1$. Performing the same for the consecutive time-steps, we obtain that $\rho^{\veps,n} - 1 = O(\veps^2),\; n>0$ and $(\rho,\uu{u},\phi)$ satisfies \eqref{eq:dis_incomp_mas}-\eqref{eq:const_dens_incomp} for each $1\leq n < N$.  
\end{proof}

\section{Numerical Results}
\label{sec:num_res}
In this section, we obtain some numerical results with the semi-implicit scheme \eqref{eq:dis_cons_mas}-\eqref{eq:dis_poisson}. Note that the stability analysis performed in Section~\ref{sec:SI_scheme}, cf.\
Theorem~\ref{lem:dis_tot_energy}, requires that the timesteps $\dt$ be chosen according to the conditions \eqref{eq:tstep_quad_suff}-\eqref{eq:suftime}.
Since the above stability condition \eqref{eq:tstep_quad_suff} is implicit and difficult to carry out, we derive a sufficient condition along the lines of \cite{CDV17, DVB17,DVB20}  which is easy to implement.

\begin{proposition}
Suppose $\dt>0$ be such that for each $\s\in\E^{(i)},\; i\in\{1,\dots,d\},\; \s = K|L$, the following holds: 
\begin{equation}
\label{eq:timestep_suff}
    \dt\max\Bigg\{\frac{\abs{\D K}}{\abs{K}},\frac{\abs{\D L}}{\abs{L}}\Bigg\}\Bigg(\abs{u^n_\s} + \sqrt{\eta}\sqrt{\left|p^{n}_L - p^{n}_K\right|+\rho^{n}_\s\left|\phi^{n+1}_L - \phi^{n+1}_K\right|}\Bigg)\leq \min\Big\{1,\; \frac{1}{5}\mu^{n}_{K,L}\Big\},
\end{equation}
where $\abs{\D K}=\sum_{\s\in\E(K)}\abs{\s}$ and $\mu^{n}_{K,L} = \displaystyle\frac{\min\{\rho^n_K,\rho^n_L\}}{\rho^{n}_\s}$. Then $\dt$ satisfies the inequality \eqref{eq:suftime}.
\end{proposition}
\begin{proof}
The proof follows the same lines of \cite[Proposition 3.2]{CDV17}, where a similar result has been obtained for an explicit scheme; see also \cite{AGK22,DVB17,DVB20} for analogous treatments. 
\end{proof}

\begin{remark}
    We use the the classical explicit scheme, as in \cite[Definition 1.7]{Deg13} as a reference scheme in the numerical experiments carried out in the following subsections. Following the notations from \cite{HL+20}, the fully-discrete classical scheme for on a collocated MAC grid reads
    \begin{subequations}
    \label{eq:cl_dis_ep}
     \begin{gather}
     \frac{1}{\dt}(\rho_{K}^{n+1}-\rho_{K}^{n})+\frac{1}{\abs{K}}\sum_{\s\in\E(K)}F_{\sigma,K}^{n}=0, \ \forall K\in \M, \label{eq:cl_dis_cons_mas}\\
    \frac{1}{\dt}(\rho_{K}^{n+1}\uu{u}_{K}^{n+1}-\rho_{K}^{n}\uu{u}_{K}^{n})+\frac{1}{\abs{K}}\sum_{\s\in\E(K)}F_{\sigma, K}^{n}\uu{u}_{\s}^{n}+(\bgrd_{\M}p^n)_K=\rho_{K}^{n+1}(\bgrd_{\M}\phi^{n+1})_K, \ \forall K\in\M,  \label{eq:cl_dis_cons_mom}\\
    \veps^2(\Delta_{\M}\phi^{n+1})_{K}=\rho_{K}^{n+1}-1, \ \forall K\in \M,  \label{eq:cl_dis_poisson}
 \end{gather}
\end{subequations}
with $F_{\sigma,K}^{n}=F_{\sigma}(\rho^n,\uu{u}^n)\cdot\uu{\nu}_{\s,K}$. In all our subsequent numerical experiments, $F_\s$ is chosen as the Rusanov flux
\begin{gather}
    F_{\sigma}(\rho,\uu{u})=\frac{\rho_K\uu{u}_K+\rho_L\uu{u}_L}{2}-\frac{\max(s_K,s_L)}{2}(\rho_L-\rho_K), \ \text{for }\s=K|L\in\E,
\end{gather}
where $s_K=\abs{u^n_K + \sqrt{\gm \rho_K^{\gm-1}}}$.
\end{remark}
\begin{remark}
Note that the classical scheme \eqref{eq:cl_dis_ep} is stable under a rather stringent time-step restriction $\dt\sim\mcal{O}(\veps)$; see e.g.\ \cite{ACS22, CDV07}, whereas the time-step restriction \eqref{eq:tstep_quad_suff} coupled with \eqref{eq:timestep_suff} allows larger time-steps for the semi-implicit scheme \eqref{eq:dis_ep}. To this end, we have implemented the time-steps $\dt<\veps$ for the reference classical scheme and the minimum of the time-steps obtained from  \eqref{eq:timestep_suff} and \eqref{eq:tstep_quad_suff} for the semi-implicit scheme in all the subsequent numerical experiments.
\end{remark}
\subsection{1D Periodic Perturbation of a Quasineutral State}
\label{sec:1d_perturb_periodic}
We consider the following initial data in terms of a small perturbation of a quasineutral state from \cite{CDV07}. The quasineutral state under consideration has a constant density $\rho = 1$ and a uniform constant horizontal velocity. The constant density in the Poisson equation yields $\phi = 0$. A small periodic perturbation of magnitude $\delta = \veps^2$ is added to the uniform horizontal velocity and the initial condition reads
\begin{equation}
    \rho(0, x) = 1.0, \; u(0, x) = 1.0 +\delta \cos (\kappa\pi x), \; \phi(0,x) = 0  
\end{equation}
where the frequency is $\kappa = 16$. The choice of the amplitude makes the initial data well-prepared, cf.\ \eqref{eq:wp_id}. The domain of the plasma is $[0,1]$ with periodic boundary conditions and the adiabatic constant is chosen $\gamma = 2$. In this test case we demonstrate the semi-implicit scheme's capability of recovering a quasineutral background state. We have chosen $\veps = 10^{-4}$ and we test the semi-implicit scheme \eqref{eq:dis_ep} against the classical explicit scheme \eqref{eq:cl_dis_ep} on a coarse mesh of $100$ mesh points. Note that the mesh does not resolve the parameter $\veps$. The plasma frequency for this configuration is $\omega = 1/\veps = 10^4$. Figures \ref{fig:1d_perturb_periodic_t0p01} and \ref{fig:1d_perturb_periodic_t0p1} shows comparison of the density and potential profiles at times $t = 0.01$ and $t = 0.1$ which corresponds to $100$ and $1000$ cycles respectively. The figures clearly indicate that the semi-implicit scheme is capable of capturing a solution near the quasineutral state and recovers the quasineutrality in the long-time asymptotic.
\begin{figure}[htbp]
    \centering
    \includegraphics[height=0.20\textheight]{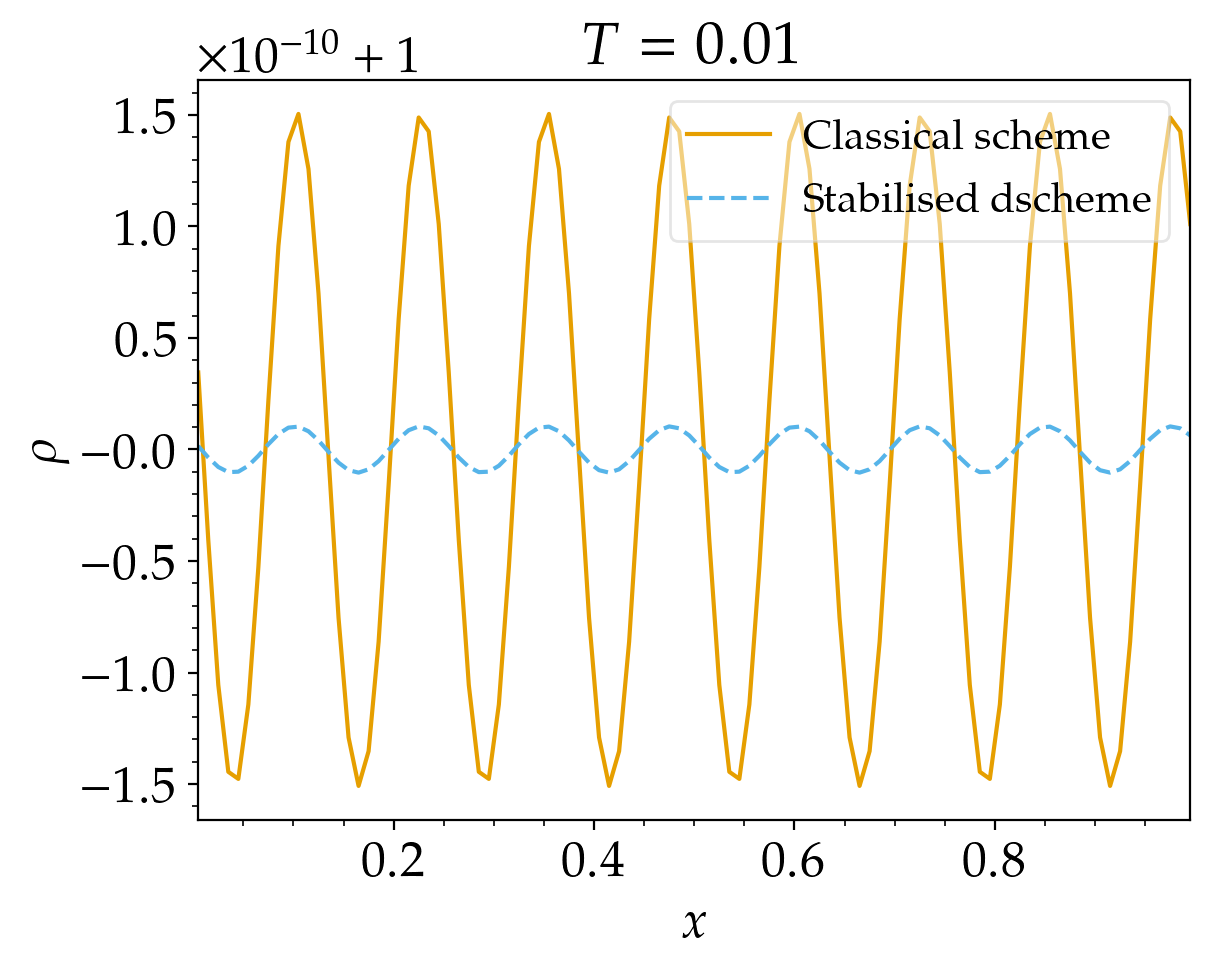}
    \includegraphics[height=0.20\textheight]{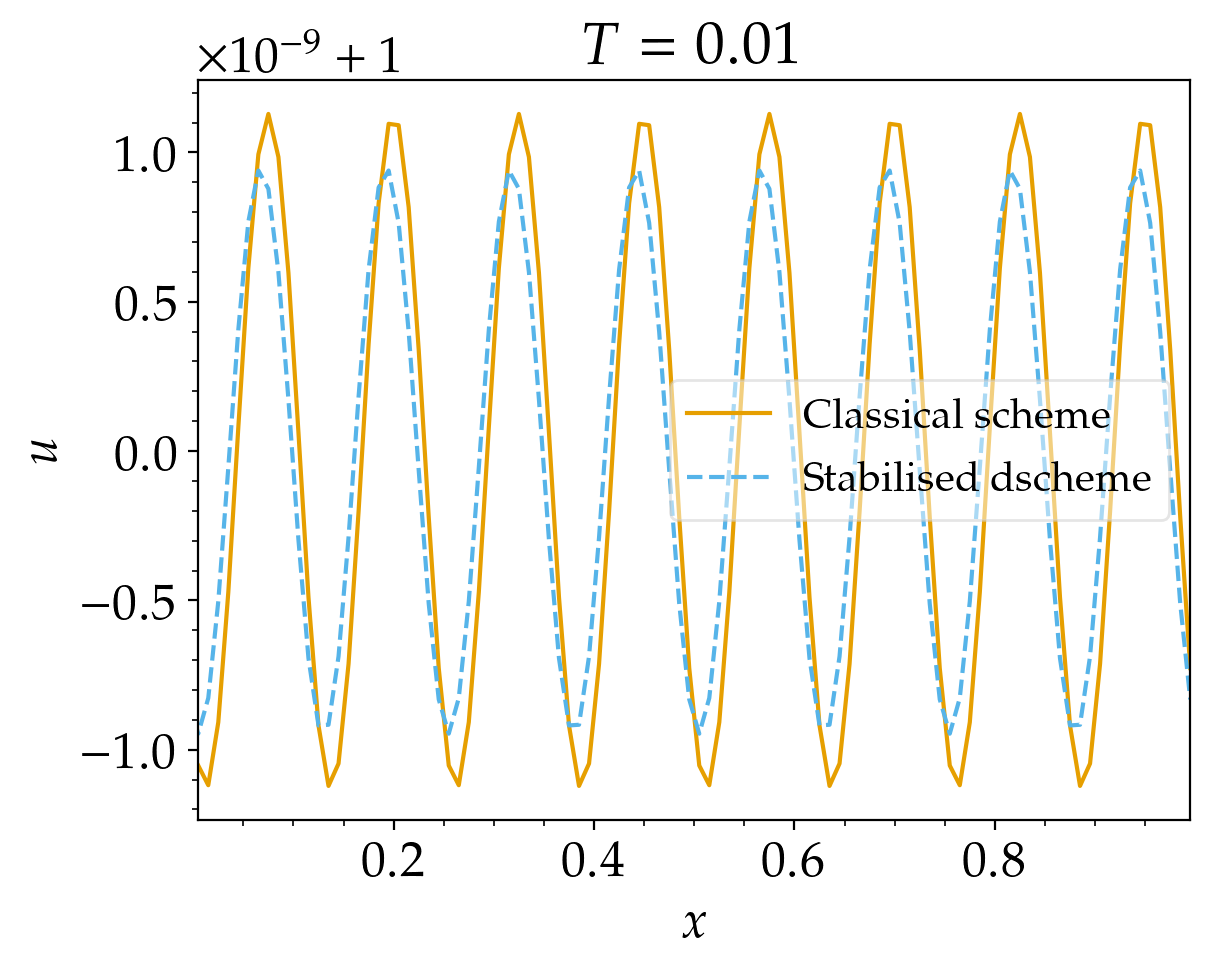}
    \includegraphics[height=0.20\textheight]{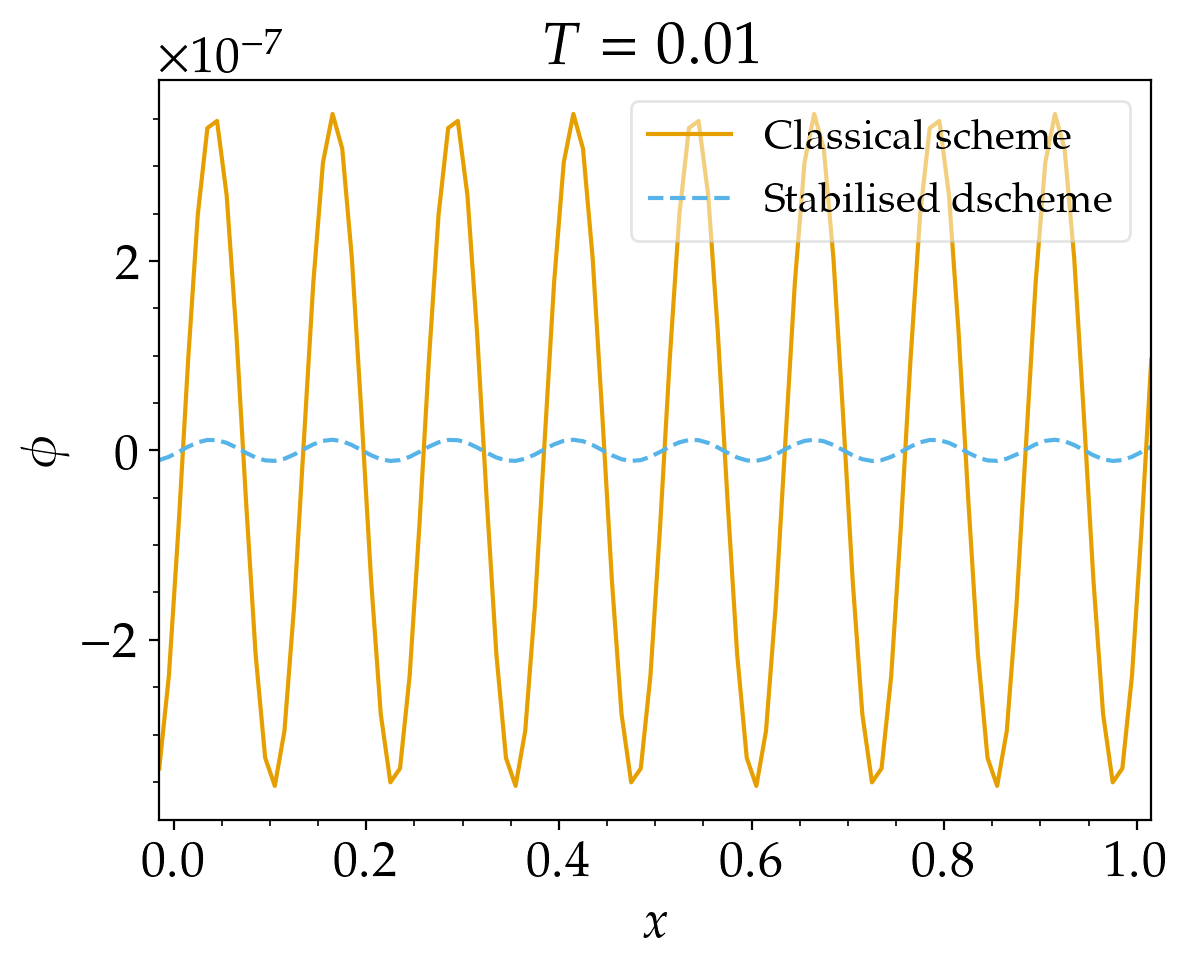}
    \caption{1D Periodic Perturbation of a Quasi-neutral State. Density, velocity and potential plots with respect to classical scheme at $t = 0.01$.} 
    \label{fig:1d_perturb_periodic_t0p01}
\end{figure}

\begin{figure}[htbp]
    \centering
    \includegraphics[height=0.20\textheight]{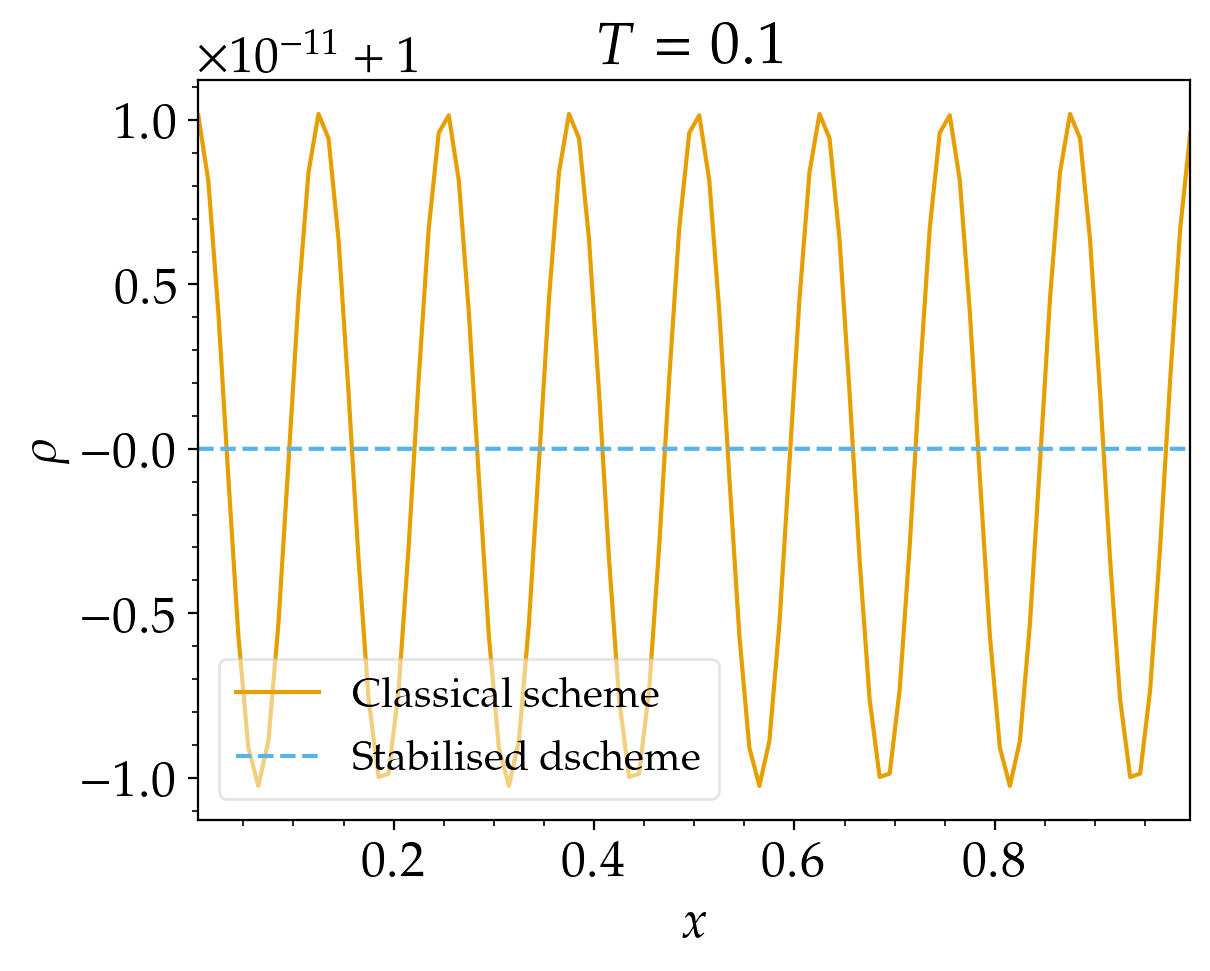}
    \includegraphics[height=0.20\textheight]{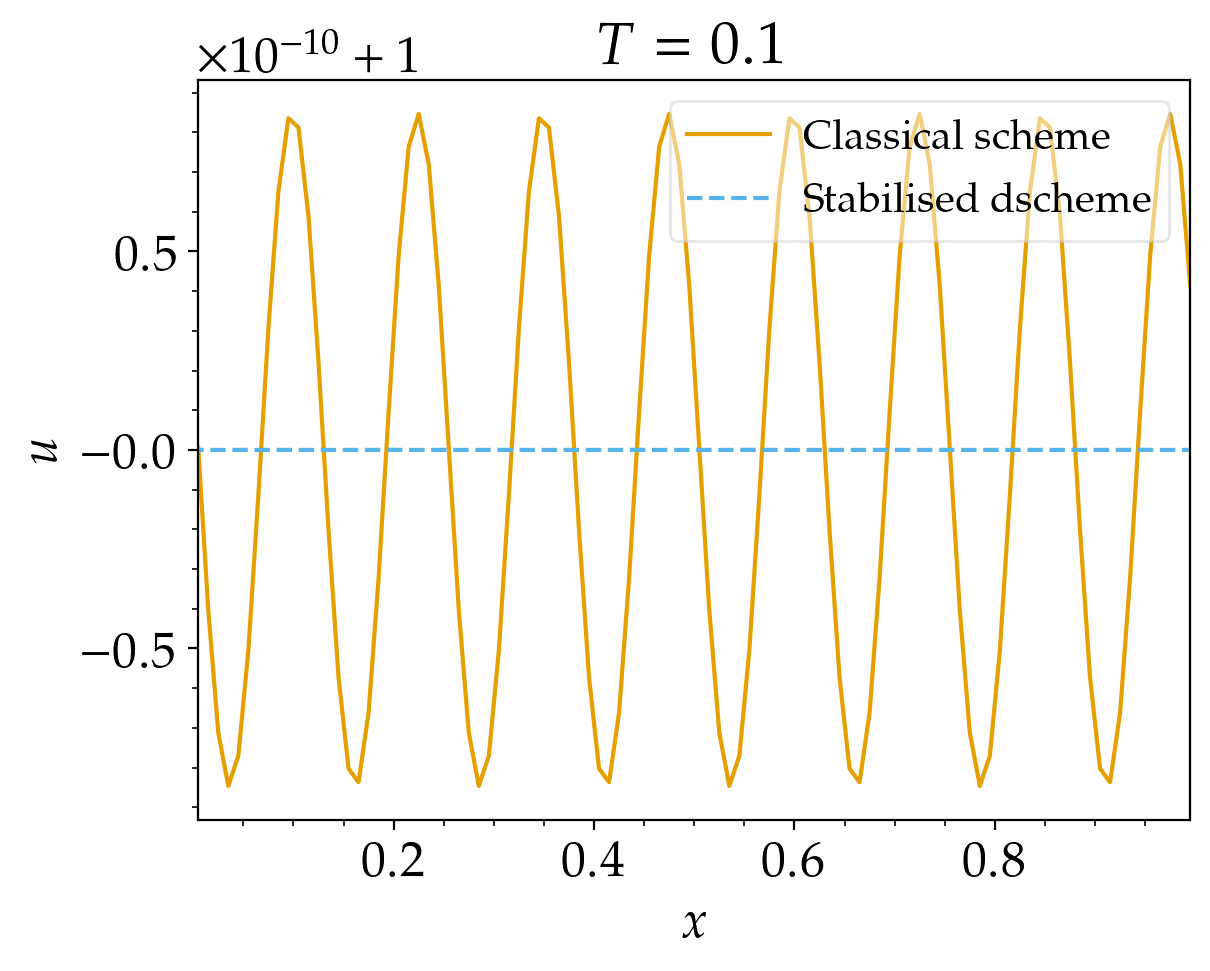}
    \includegraphics[height=0.20\textheight]{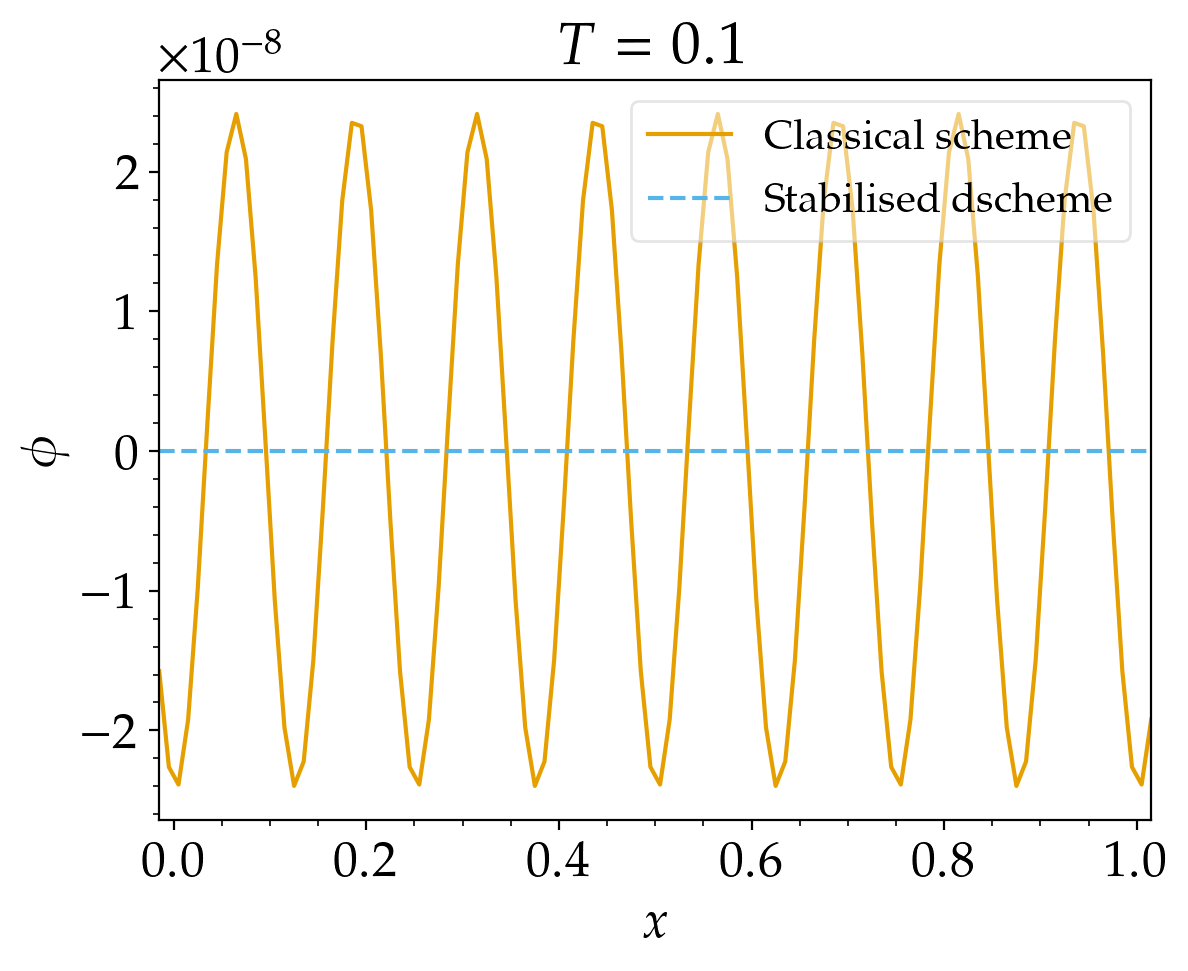}
    \caption{1D Periodic Perturbation of a Quasi-neutral State. Density, velocity and potential plots with respect to classical scheme at $t = 0.1$.} 
    \label{fig:1d_perturb_periodic_t0p1}
\end{figure}

\subsection{1D Maxwellian Perturbation of a Stationary Solution}
\label{sec:1d_perturb_maxwell}
We consider the following initial data inspired by \cite{Neg13} which is given as a small perturbation of a static equilibrium with a constant background density. The initial condition reads
\begin{equation}
    \rho(0, x) = 1.0 + \delta \sin(\kappa\pi x) 
\end{equation}
with the frequency $\kappa = 2220$. For this problem the domain of the plasma is $[0,1]$ with periodic boundary conditions and we choose the adiabatic constant $\gamma = 5/3$. The choices for the small amplitude $\delta$ has been taken as $\veps$ and $\veps^2$ to simulate a non-well-prepared and well-prepared data respectively, cf.\ \eqref{eq:wp_id}. The aim of this numerical test is to exhibit the scheme's ability to recover the steady state. We choose $\veps = 10^{-4}$ and test the semi-implicit scheme \eqref{eq:dis_ep} against the classical scheme \eqref{eq:cl_dis_ep} on a coarse mesh of $100$ mesh points that does not resolve $\veps$. We plot the density, velocity and potential profiles at time $t = 0.1$ in the Figures \ref{fig:1d_perturb_maxwell_eps}, \ref{fig:1d_perturb_maxwell_eps2} for amplitudes $\delta = \veps$ and $\delta = \veps^2$ respectively. The figures clearly indicate that the semi-implicit scheme effectively recovers the equilibrium in the long time asymptotic than the classical scheme on a coarse mesh despite being implemented with a more restrictive time-stepping condition.
\begin{figure}[htbp]
    \centering
    \includegraphics[height=0.18\textheight]{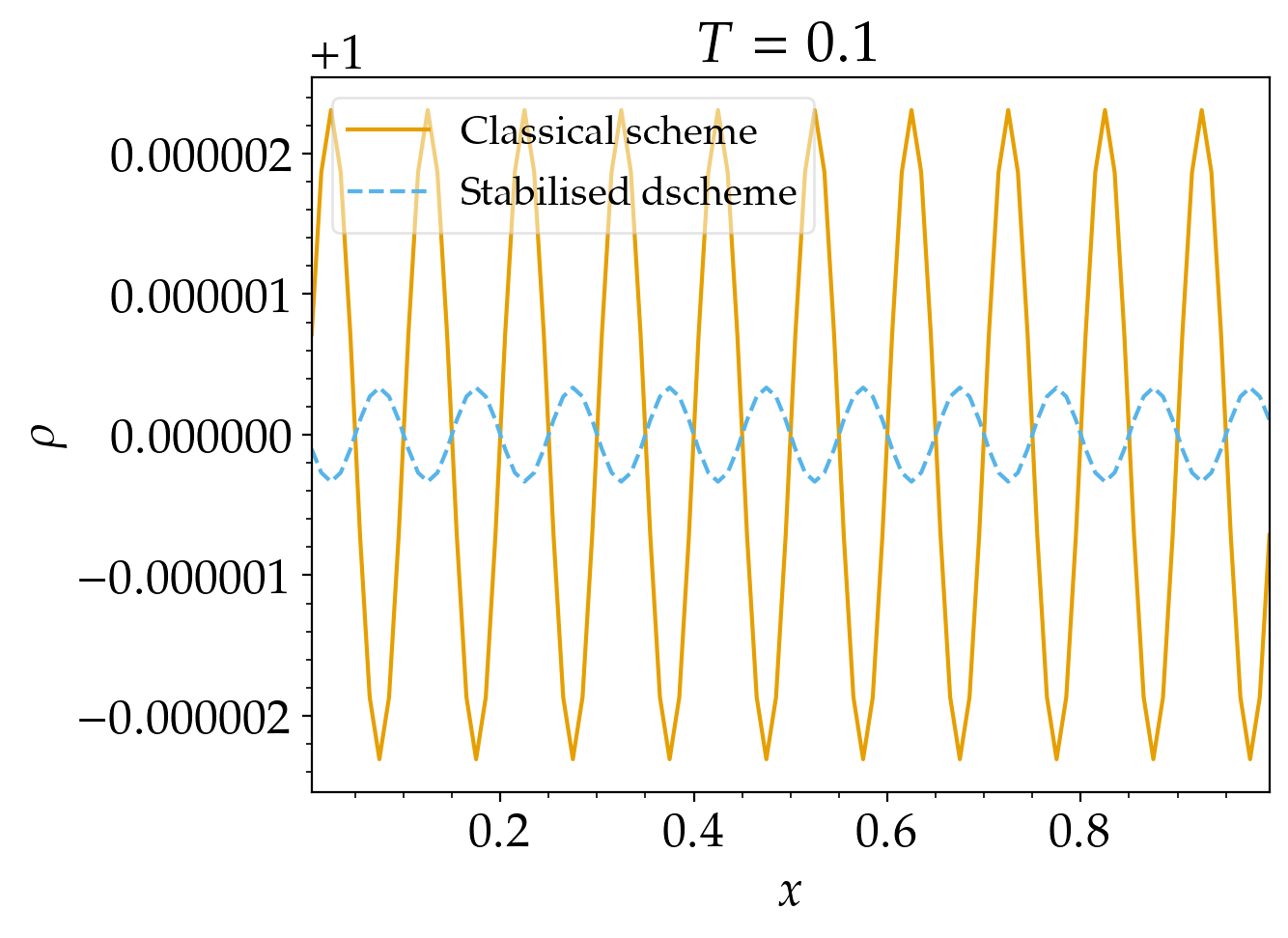}
    \includegraphics[height=0.18\textheight]{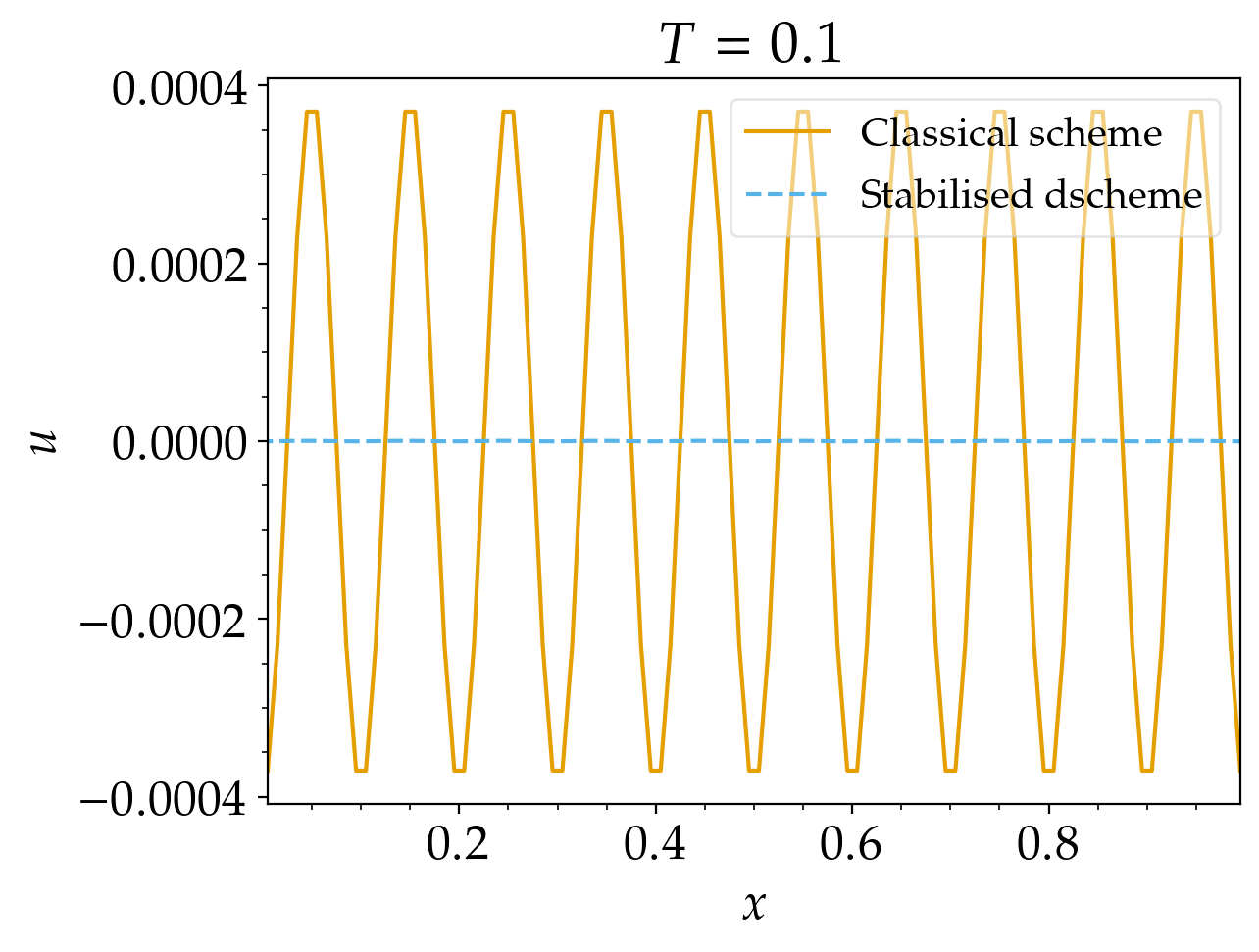}
    \includegraphics[height=0.18\textheight]{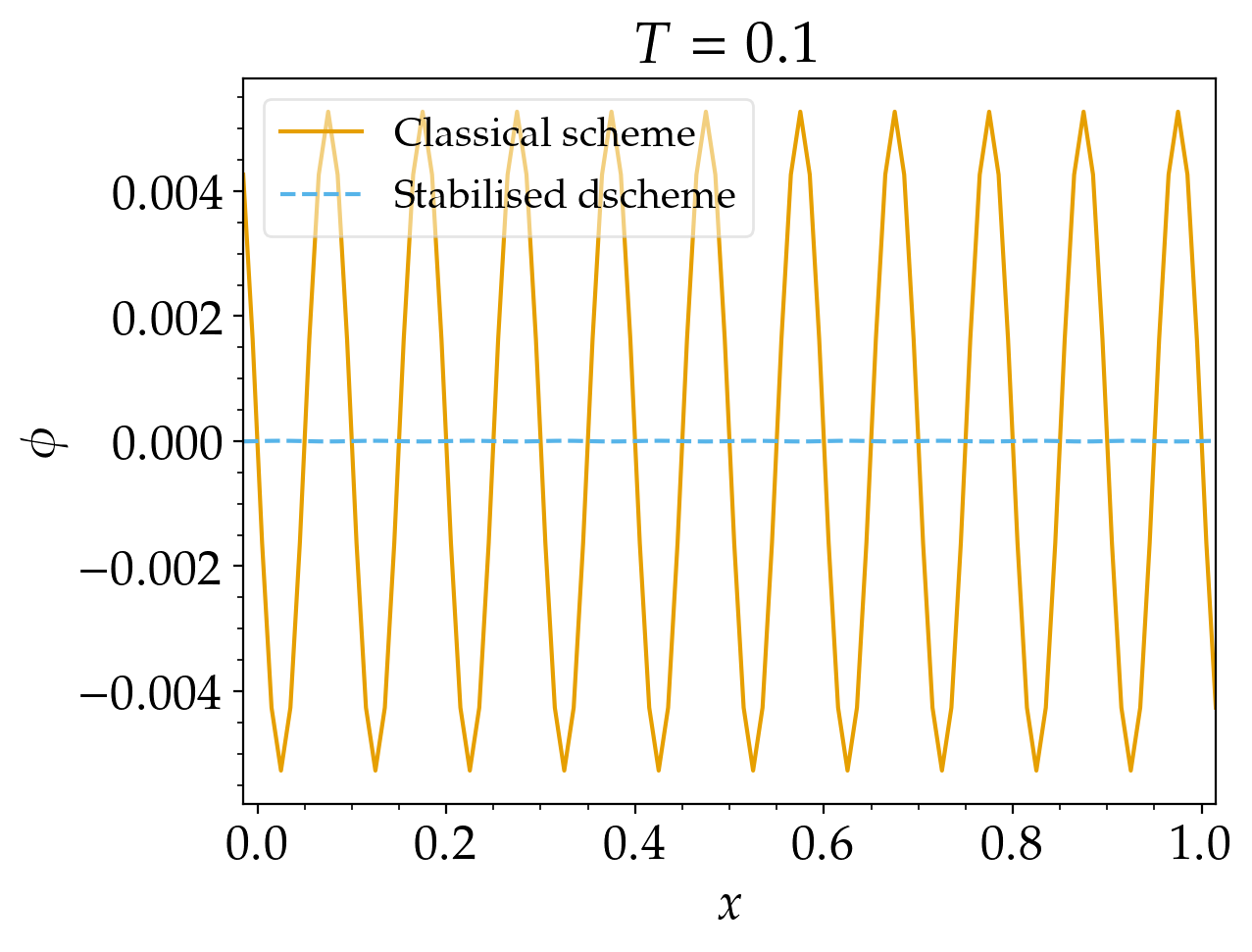}
    \caption{1D Maxwellian Perturbation of a Stationary Solution. Density, velocity and potential plots with respect to classical scheme for $\delta = \veps$} 
    \label{fig:1d_perturb_maxwell_eps}
\end{figure}
\begin{figure}[htbp]
    \centering
    \includegraphics[height=0.20\textheight]{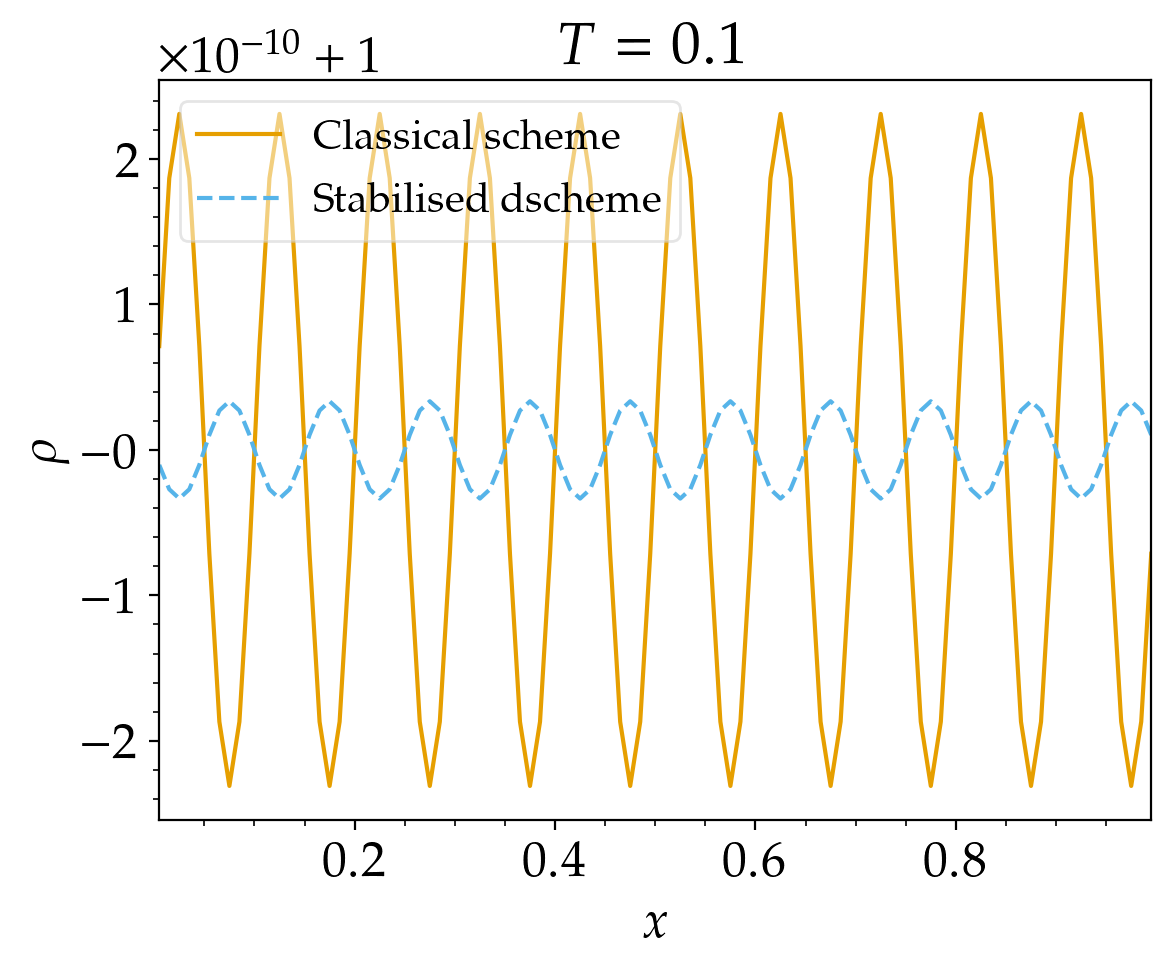}
    \includegraphics[height=0.20\textheight]{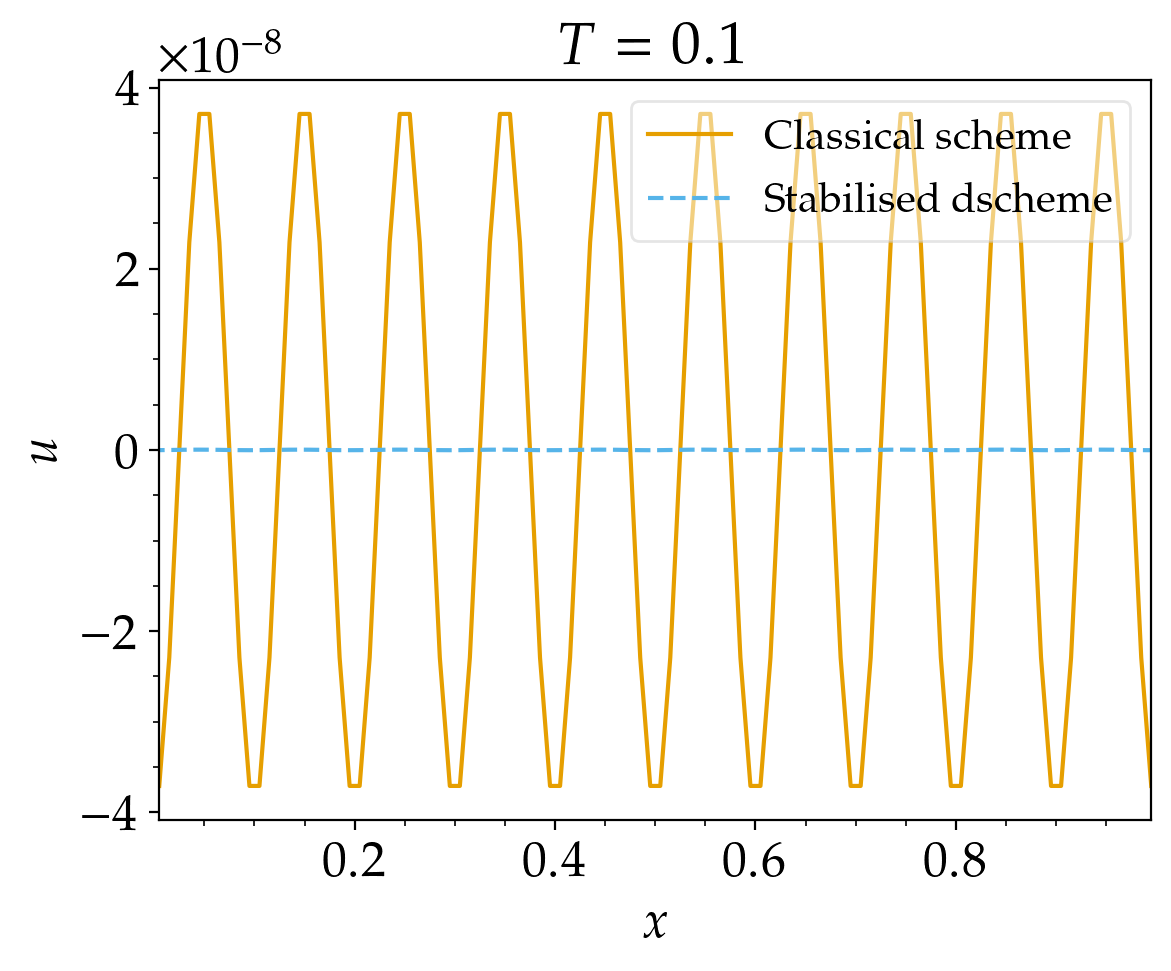}
    \includegraphics[height=0.20\textheight]{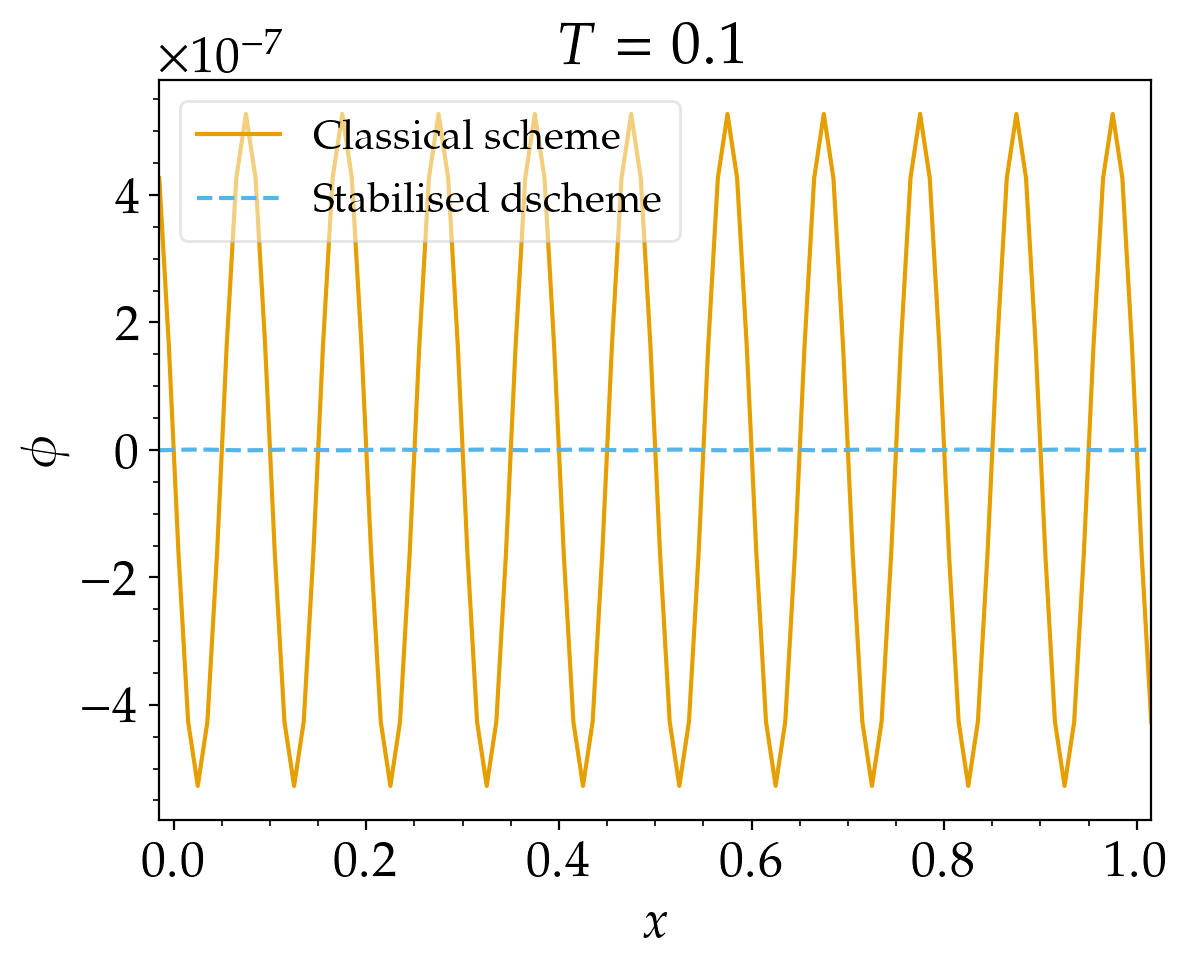}
    \caption{1D Maxwellian Perturbation of a Stationary Solution. Density, velocity and potential plots with respect to classical scheme for $\delta = \veps^2$} 
    \label{fig:1d_perturb_maxwell_eps2}
\end{figure}

\subsection{Oscillation of a Plasma Column}
\label{sec:osc_plasma}
We consider the following numerical test case from \cite{MST23} where we simulate the oscillation of a plasma column in the domain $\Omega = [0,1]\times[0,1]$. The initial conditions are given by
\begin{equation}
\rho(0, x, y)=\begin{cases}
1 - \delta, &\mathrm{if}\; x < 5.0\\
1 + \delta, &\mathrm{if}\; x \geq 0.5
\end{cases}, \;
\uu{u}(0, x, y) = (0, 0).
\end{equation}
We choose the parameter $\delta = \veps$ with $\veps = 10^{-3}$. The plasma frequency is given by $\omega = 1/\veps$ and hence the plasma period for this configuration is obtained as $t_p = 2\pi\veps$. The problem is supplemented with the homogeneous Neumann boundary condition on $\phi$, i.e.\ $\bgrd \phi\cdot\uu{\nu} = 0$ on the boundary $\partial\Omega$. Also we consider a no-flux boundary condition $\uu{u}\cdot\uu{\nu} = 0$ for the velocity component. We choose the adiabatic constant $\gamma = 1.4$. The domain is discretised using $100\times100$ grid points and we plot the cross sections of the profiles of the density, $x$-velocity and the potential at times $t=0, \frac{1}{2}t_p, t_p$ in Figure \ref{fig:osc_plasma}. We observe that the scheme is capable of simulating the plasma column oscillation while preserving the stationary contact discontinuities at $x=0.5$. As indicated by the density profile at $t = t_p$, the plasma column comes back to its initial configuration after one period.
\begin{figure}[htbp]
    \centering
    \includegraphics[height=0.155\textheight]{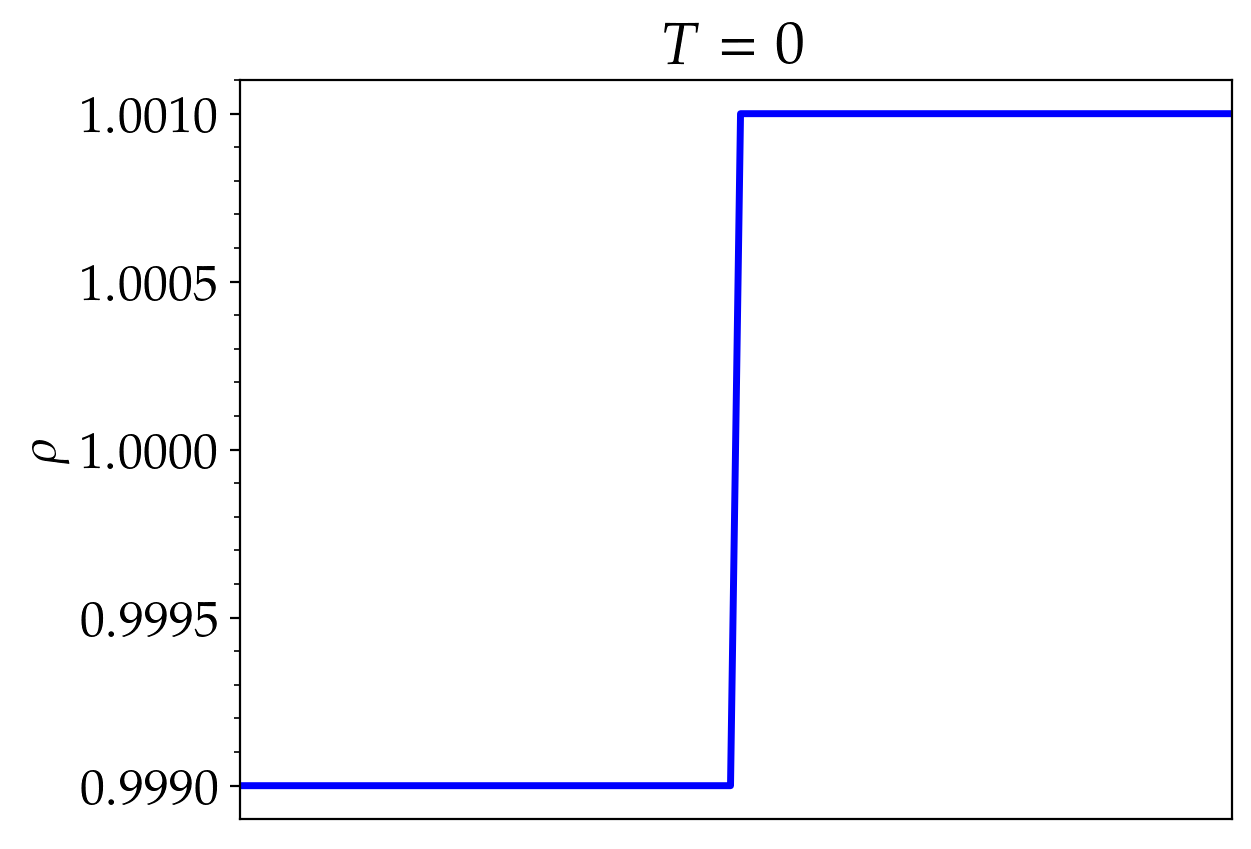}
    \includegraphics[height=0.155\textheight]{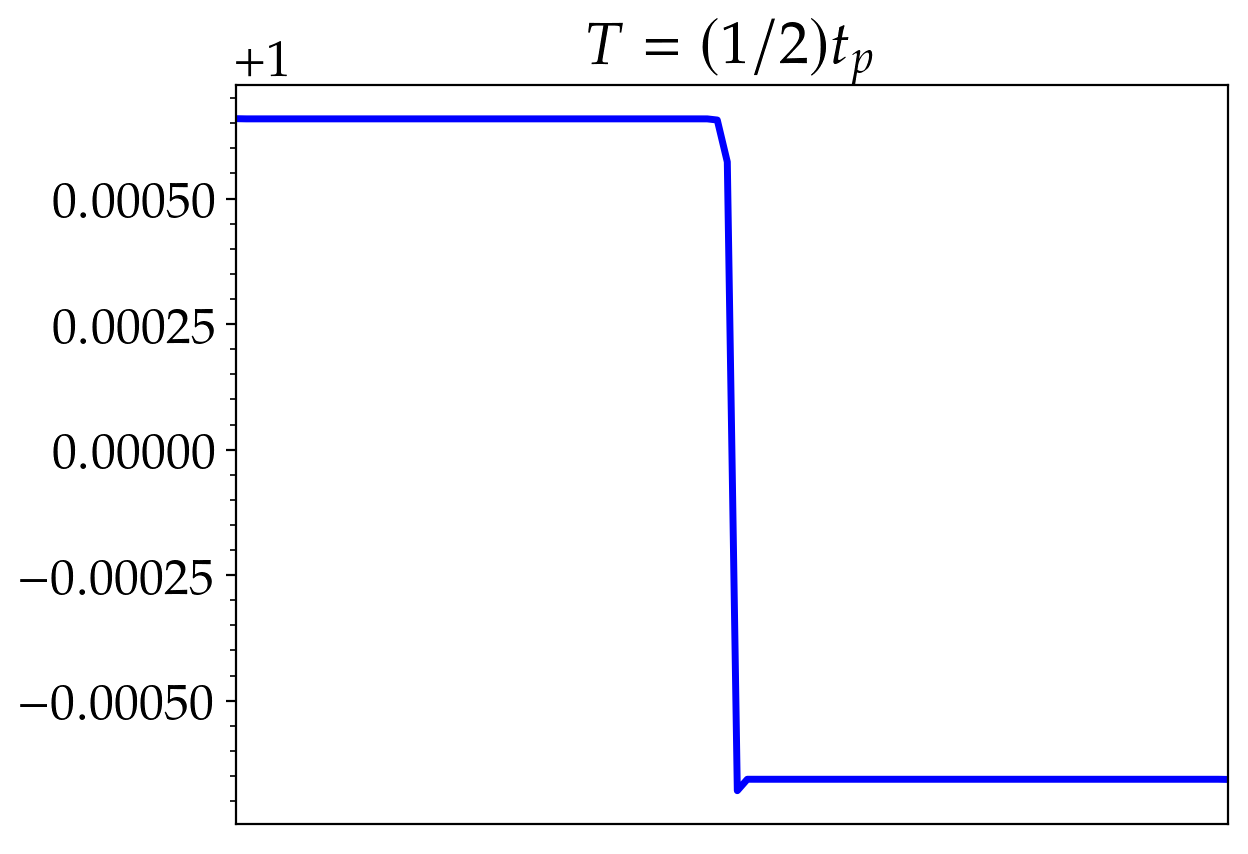}
    \includegraphics[height=0.155\textheight]{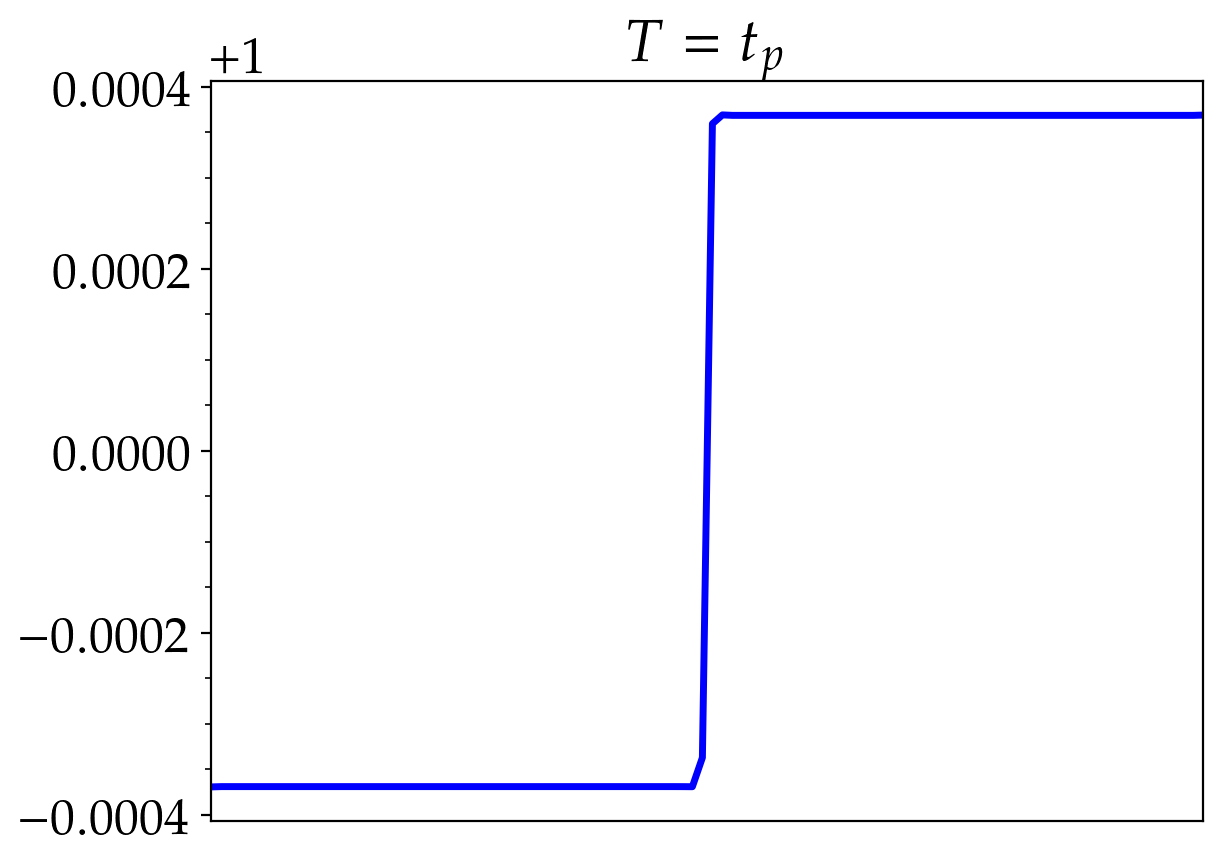}
    \includegraphics[height=0.15\textheight]{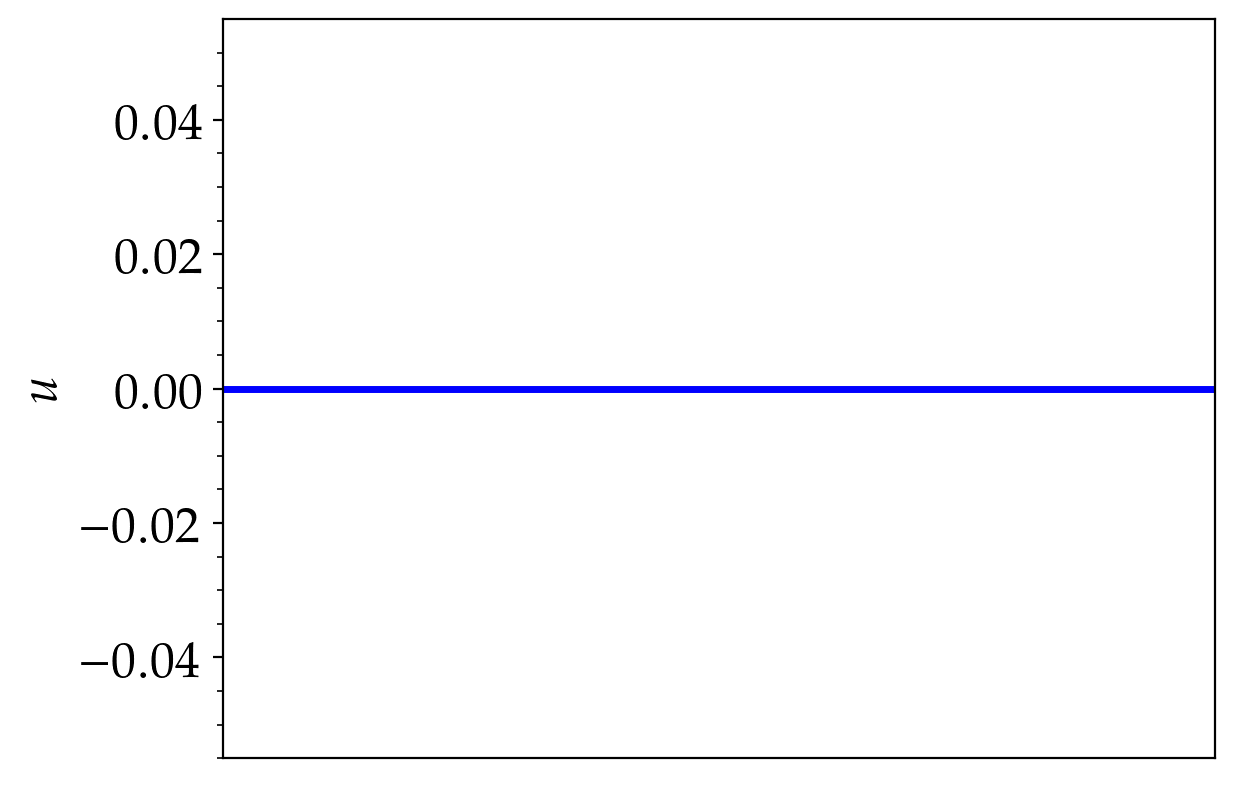}
    \includegraphics[height=0.15\textheight]{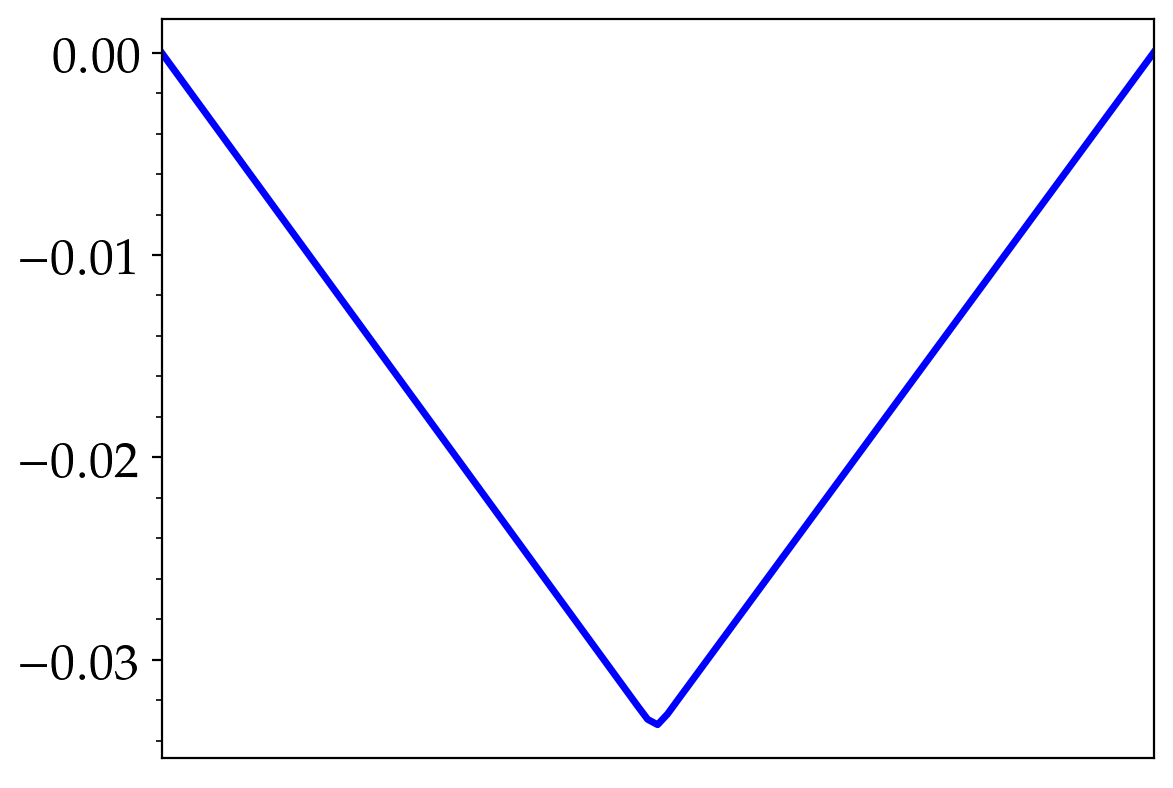} 
    \includegraphics[height=0.15\textheight]{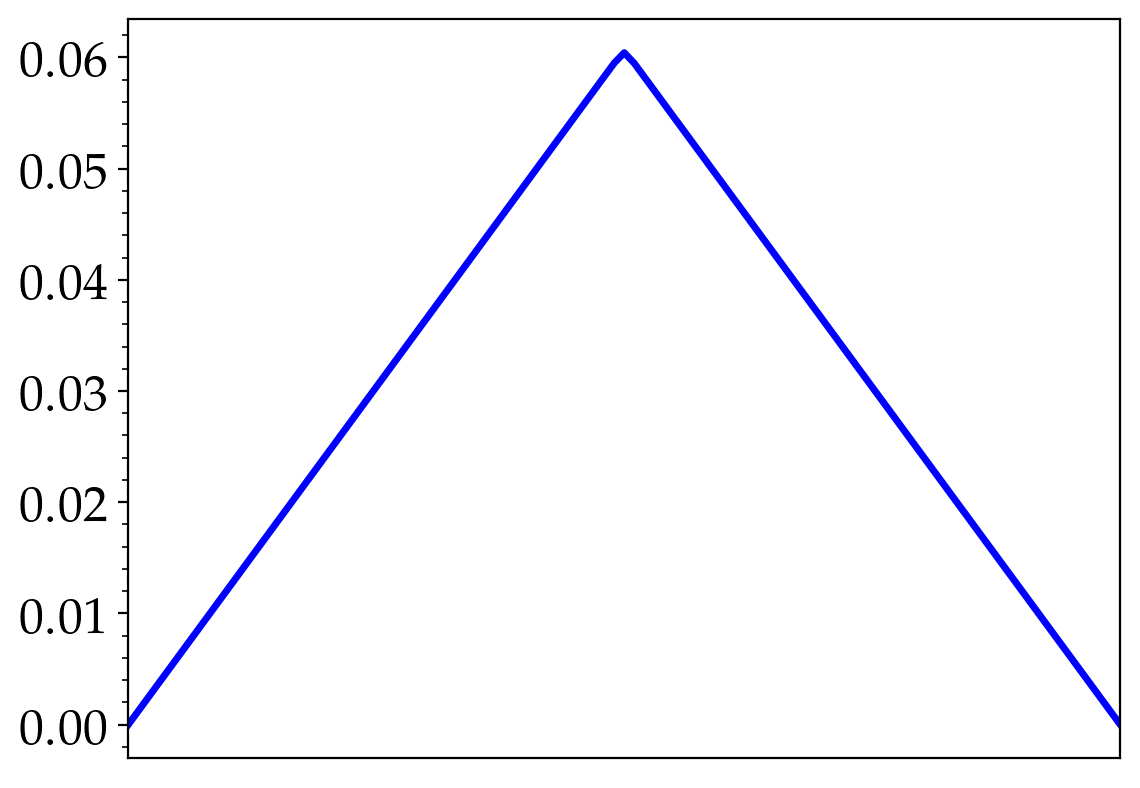}
    \includegraphics[height=0.175\textheight]{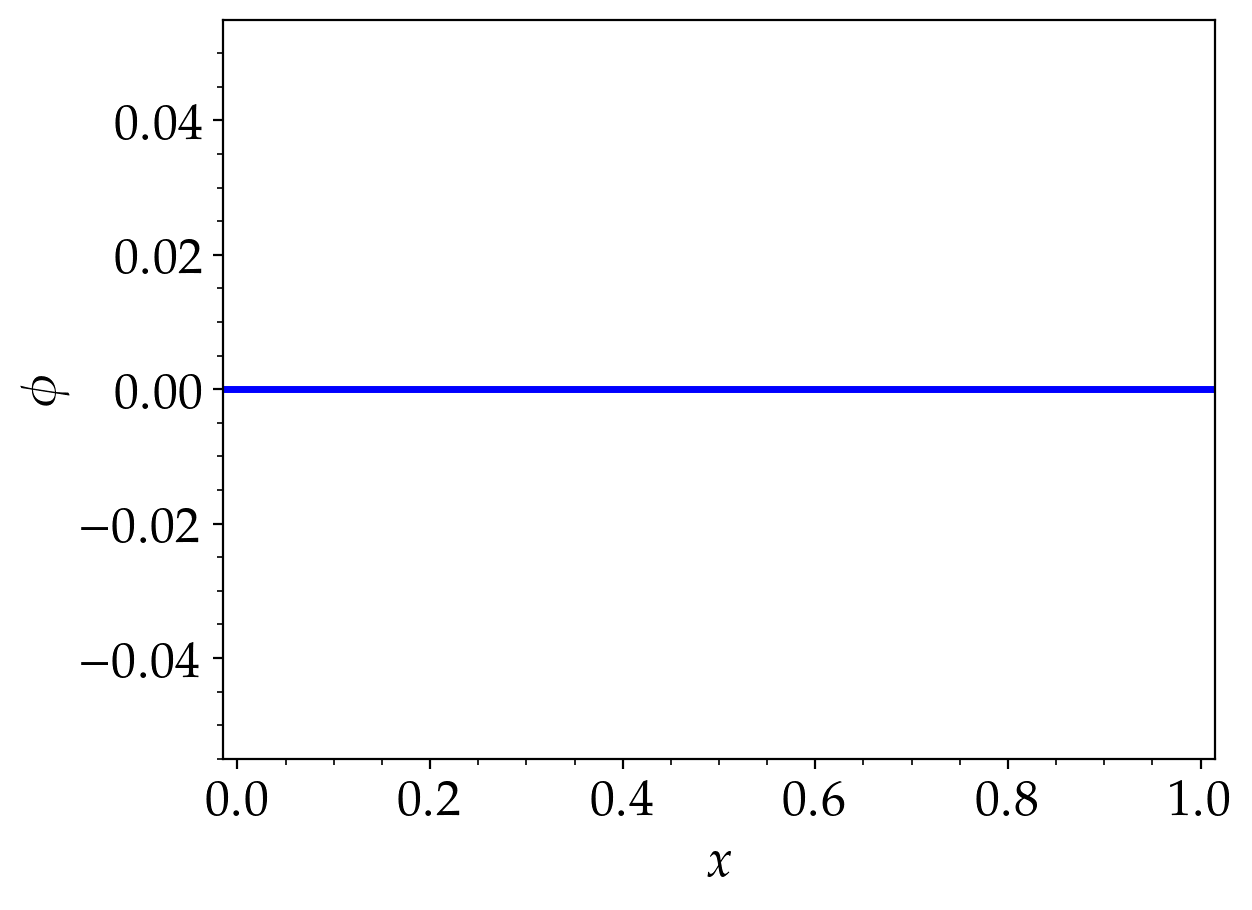}
    \includegraphics[height=0.175\textheight]{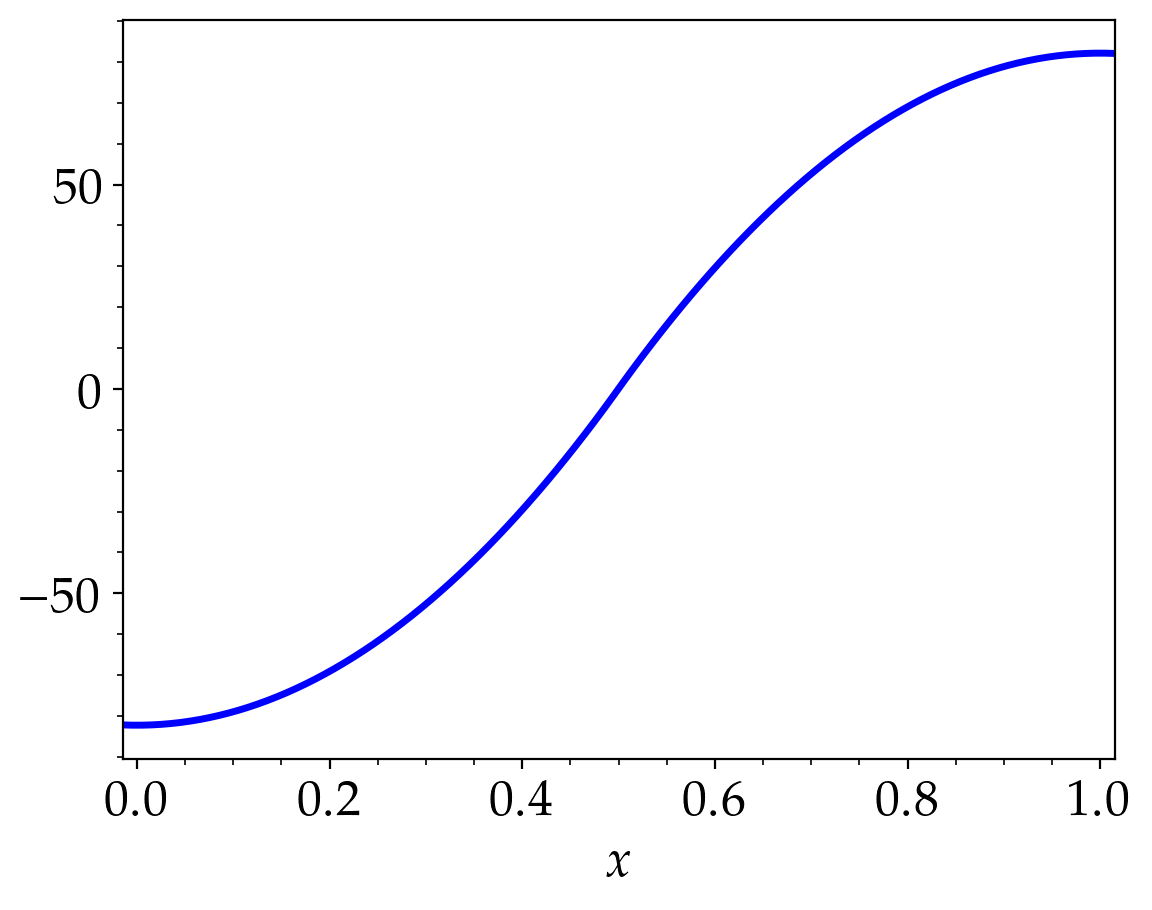}
    \includegraphics[height=0.175\textheight]{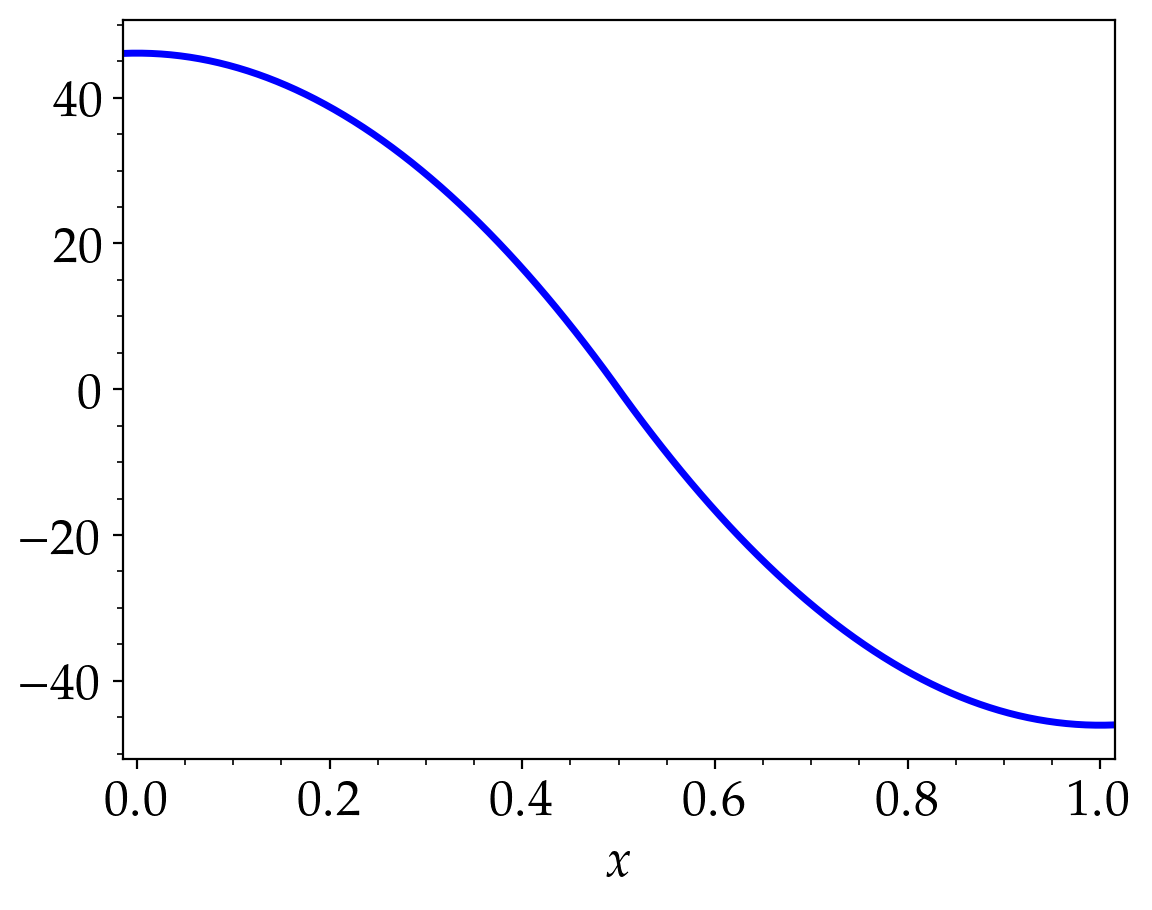}
    \caption{Oscillation of a Plasma Column. Density, velocity and potential plots at times $t=0, \frac{1}{2}t_p, t_p$ with $\veps = 10^{-3}$} 
    \label{fig:osc_plasma}
\end{figure}

\subsection{Periodic Perturbation of a Quasineutral State in 2D}
\label{sec:1d_perturb_periodic}
We consider the following initial data following \cite{CDV07, DT11} where we consider a small perturbation of a quasi-neutral state in the domain $\Omega = [0, 1]\times[0, 1]$ in $\mbb{R}^2$. The quasineutral state in 2D is described by a constant density $1$ and uniform constant velocities in both the direction. Periodic perturbations with a small amplitude is added to the quasineutral velocity field. The constant density implies that $\phi = 0$ from the Poisson Equation. The initial data reads
\begin{align}
    \rho(0, x, y) &= 1,
    \\
    u(0, x, y) &= 1 + \delta\sin(k\pi(x+y)),
    \\
    v(0, x, y) &= 1 + \delta\cos(k\pi(x+y)),
    \\
    \phi(0, x, y) &= 0,
\end{align}
where the frequency of the perturbation $\kappa$ is chosen as $\kappa = 16$. We choose the small amplitude of the periodic perturbation $\delta = \veps^2$ in order to make the data well prepared. For this problem we choose the adiabatic constant $\gamma = 2$. We implement periodic boundary conditions on all the sides of the domain which is discretised using $100\times100$ grid points. The aim of this test case is to show that the semi-implicit scheme's capability for recovering a quasineutral state in 2D. We plot the density, potential and divergence profiles in Figure \ref{fig:qn2d_per_e-2} at times $t = 0.5, 2$ for $\veps = 10^{-2}$ and observe that the scheme converges to the quasineutral limit even with a larger value of $\veps$. In Figure \ref{fig:qn2d_per_e-4} we plot the density, potential and the divergence profiles at time $t=0.005$ with $\veps=10^{-4}$ which clearly indicates that the scheme recovers a quasineutral state rather quickly with a smaller value of the amplitude of the periodic perturbation.  
\begin{figure}[htbp]
    \centering
    \includegraphics[height=0.20\textheight]{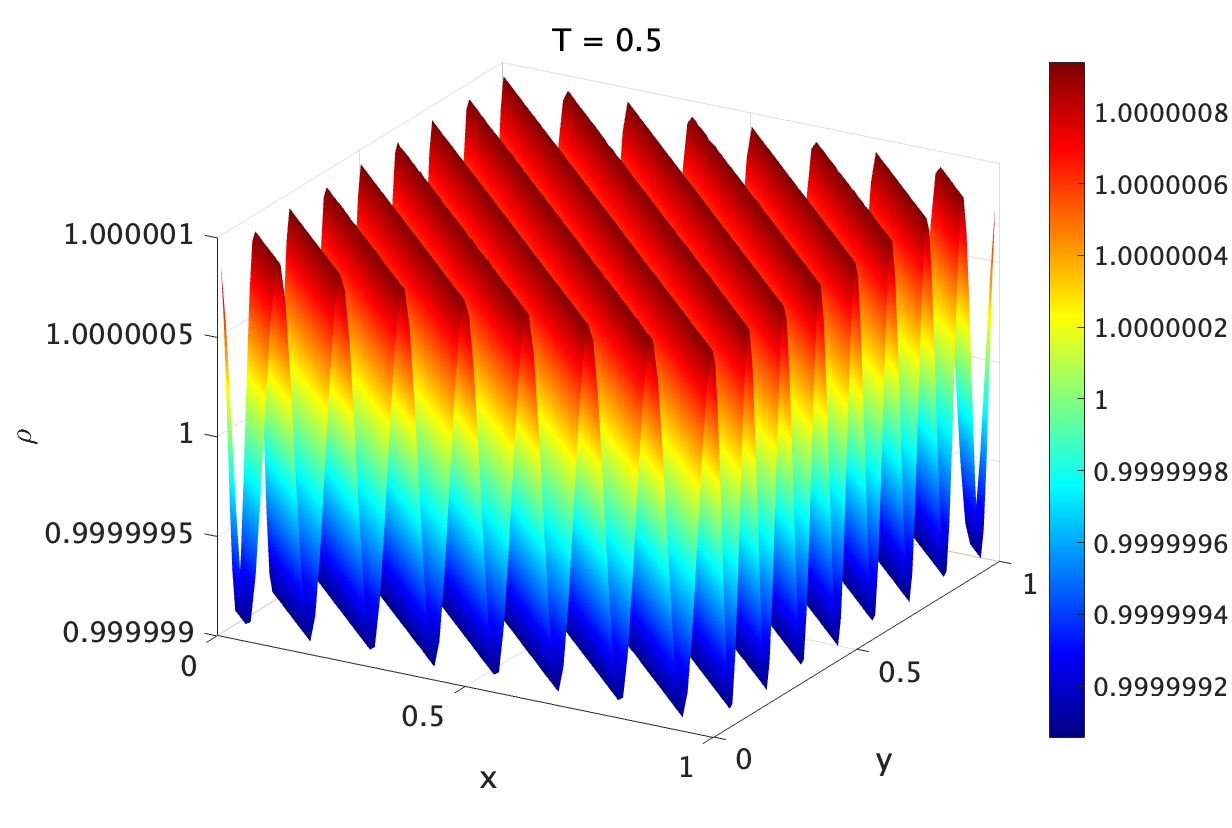}
    \includegraphics[height=0.20\textheight]{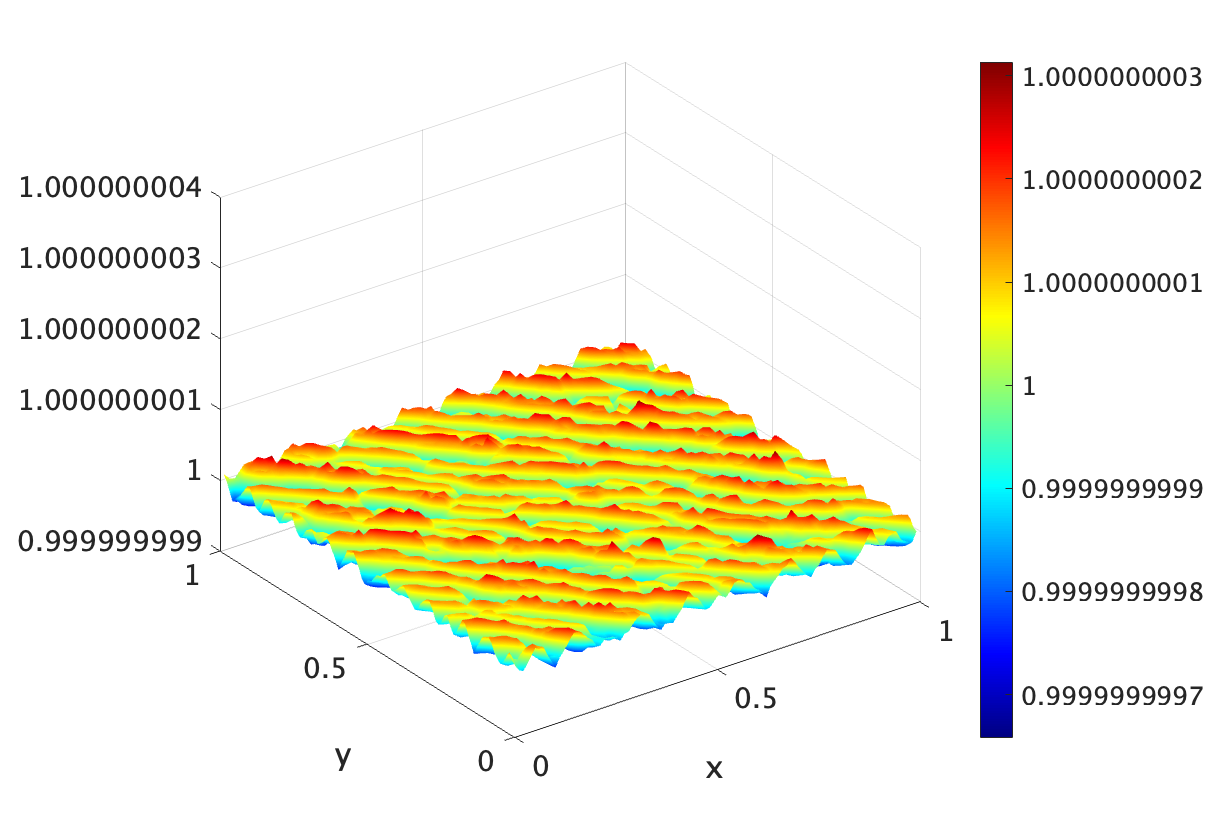}
    \includegraphics[height=0.20\textheight]{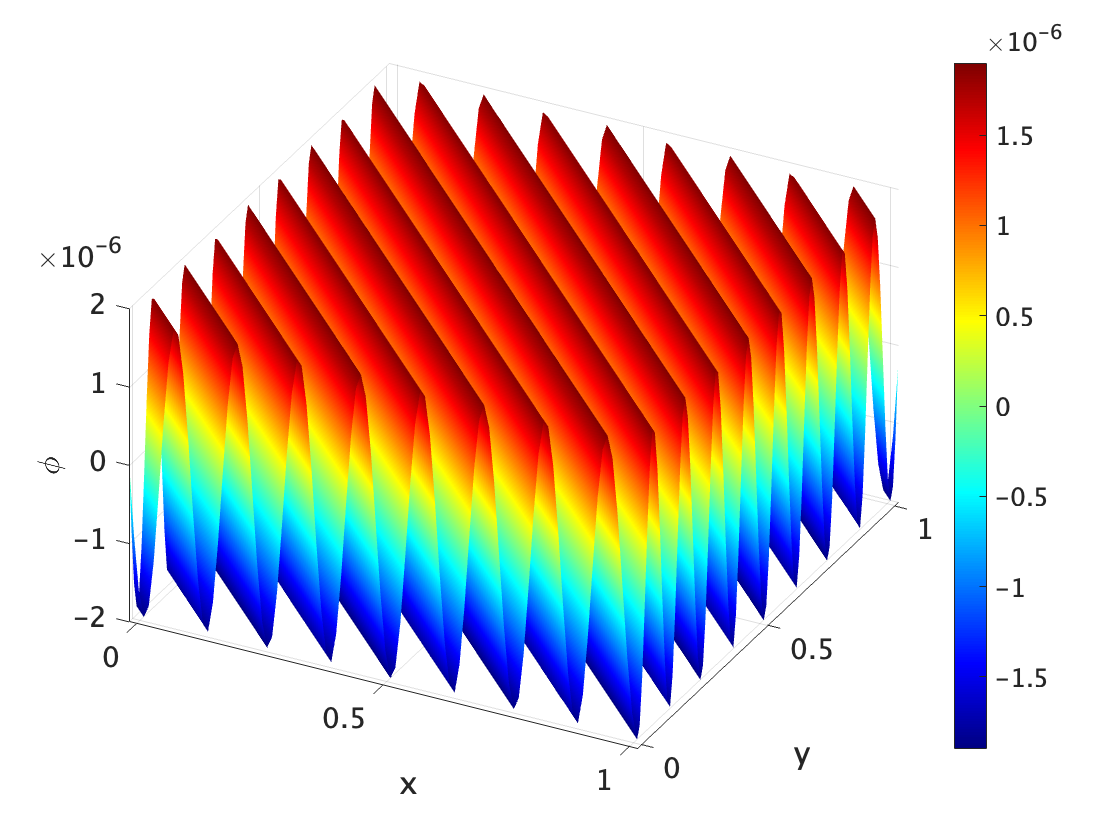}    \includegraphics[height=0.20\textheight]{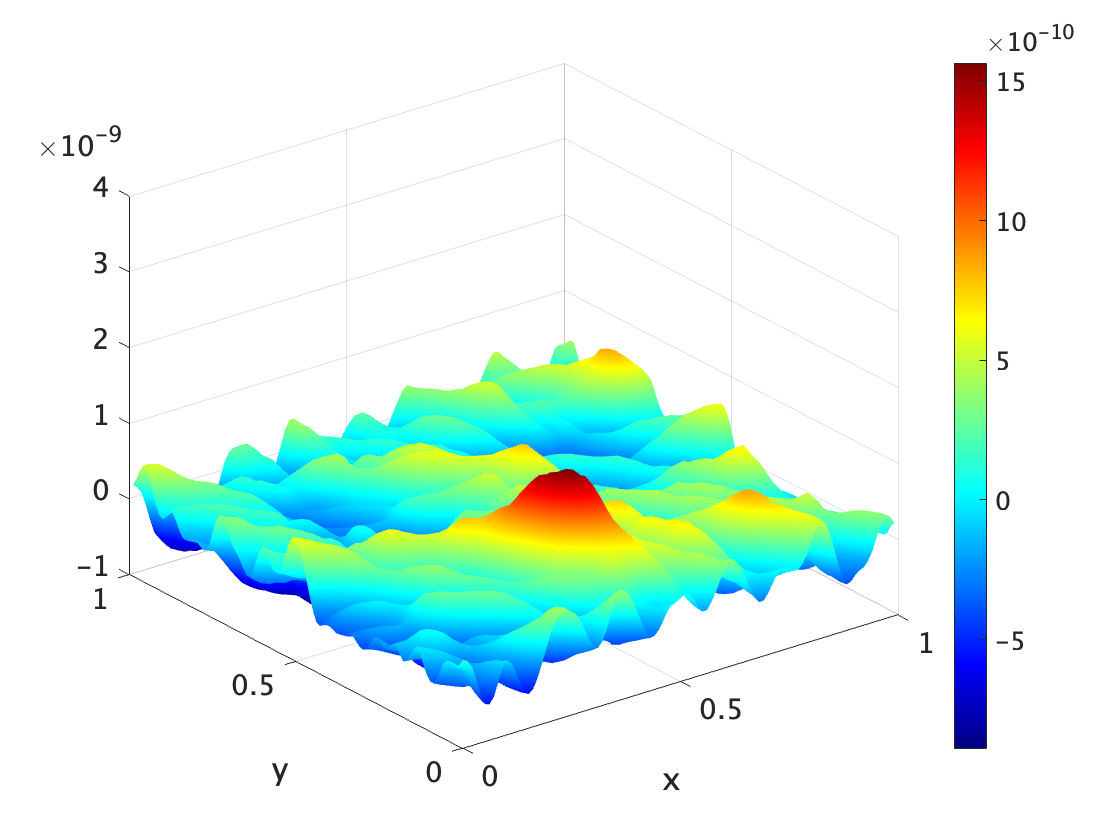}
    \includegraphics[height=0.20\textheight]{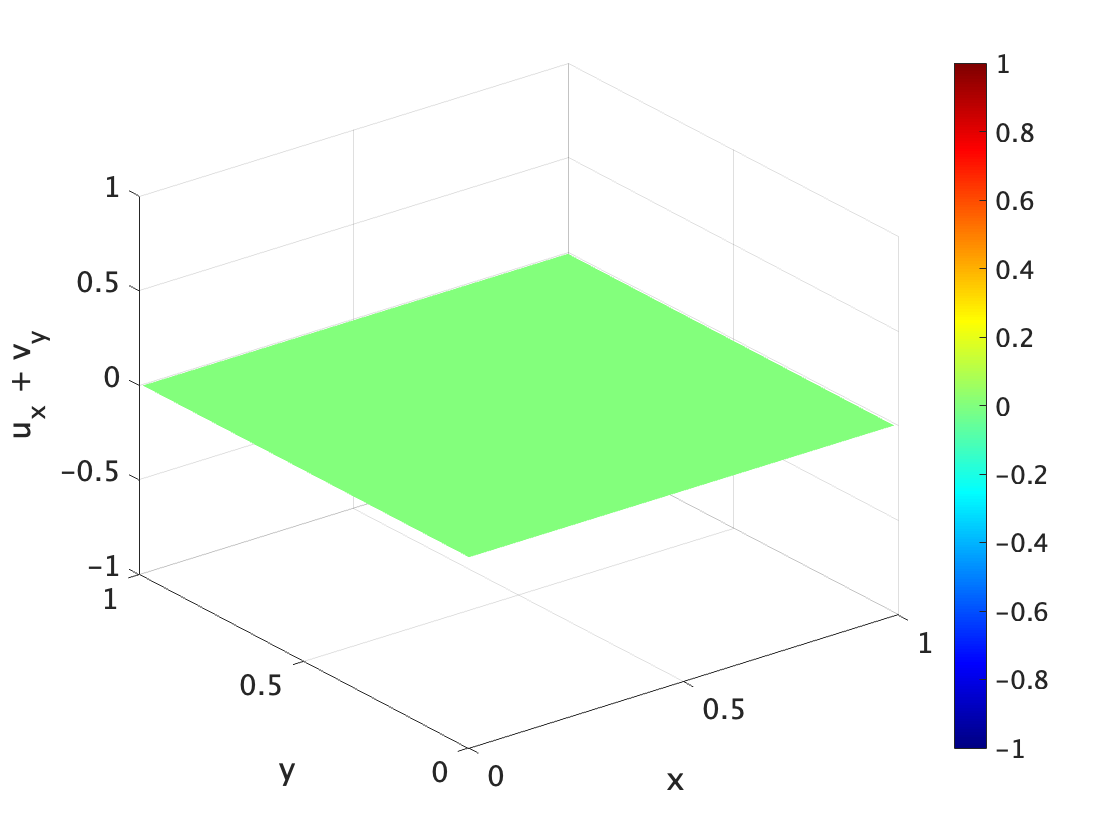} 
    \includegraphics[height=0.20\textheight]{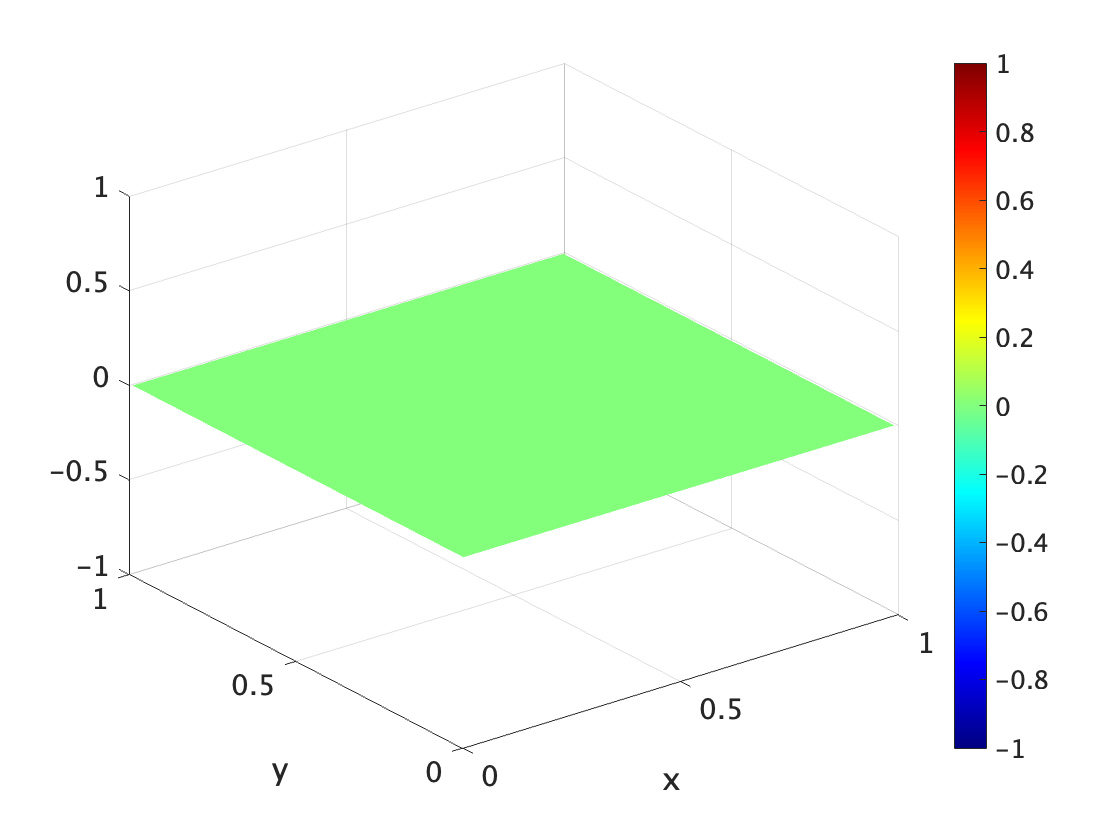}
    \caption{Periodic Perturbation of a Quasi-neutral State in 2D. Density, potential and divergence plots at times $t=0.5, 2.0$ with $\veps = 10^{-2}$} 
    \label{fig:qn2d_per_e-2}
\end{figure}

\begin{figure}[htbp]
    \centering
    \includegraphics[height=0.15\textheight]{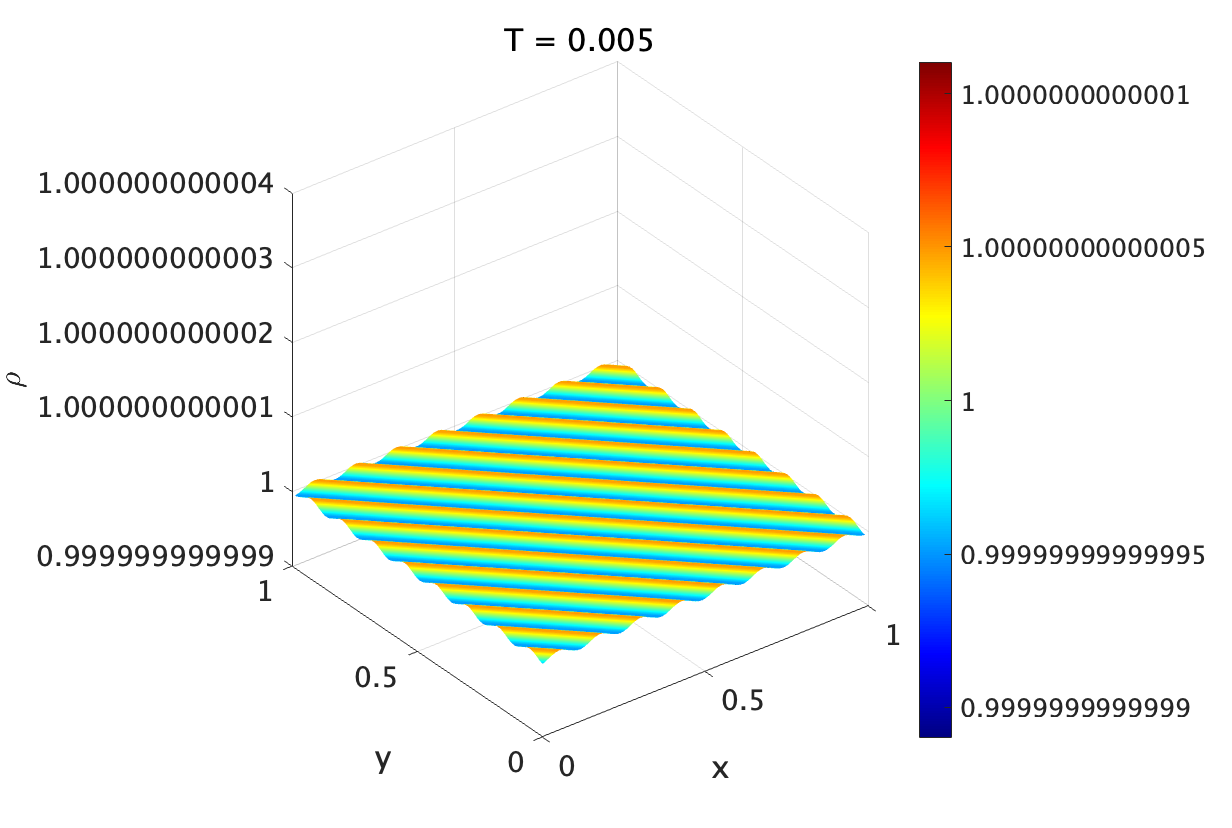}
    \includegraphics[height=0.15\textheight]{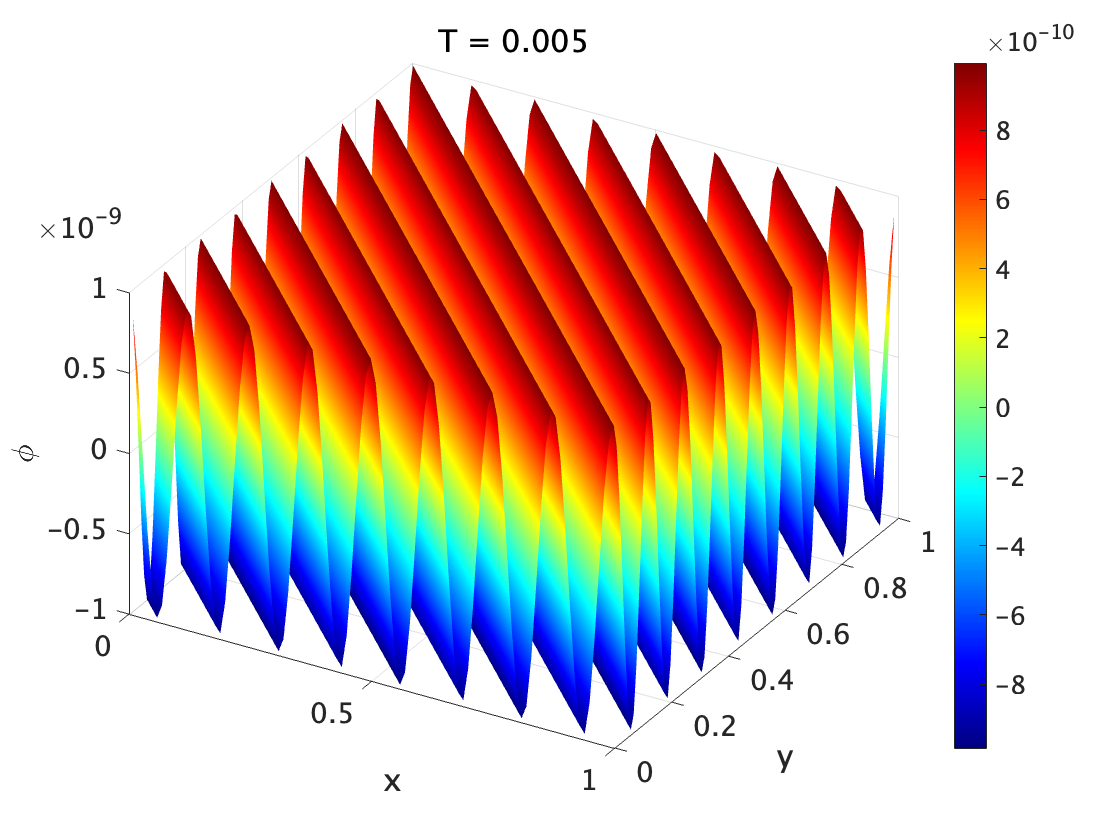}
    \includegraphics[height=0.15\textheight]{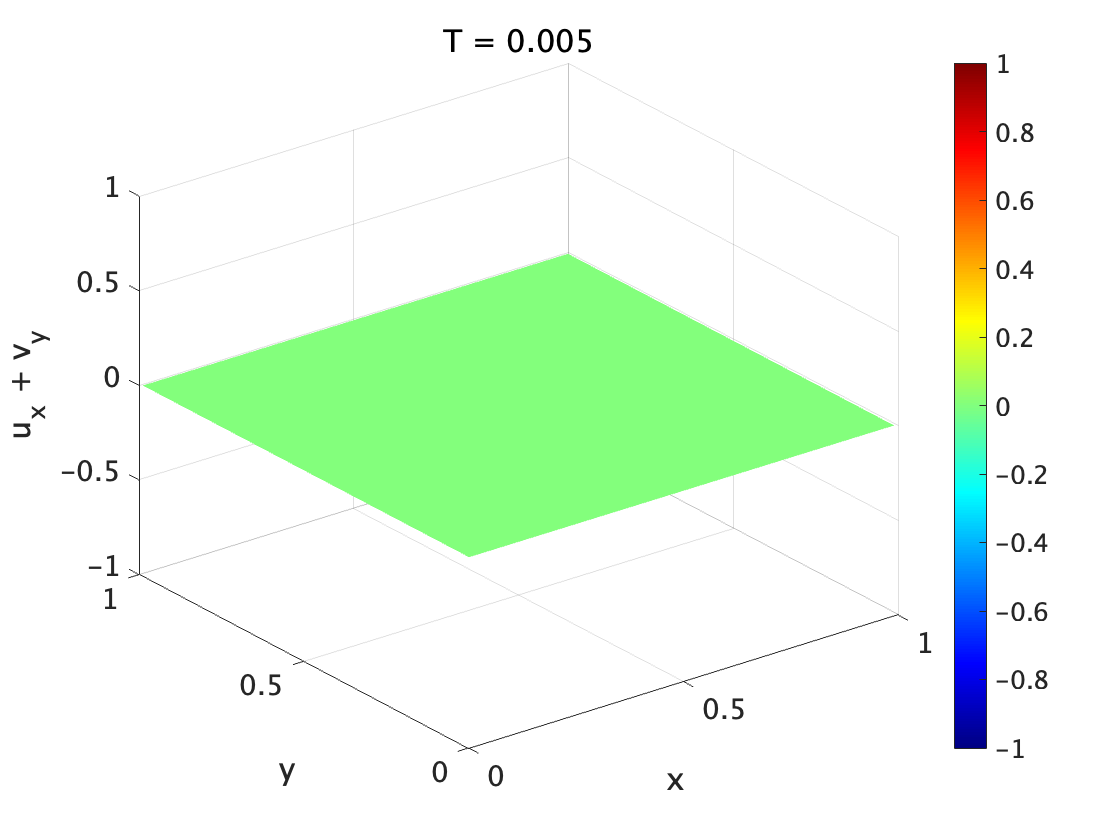}
    \caption{Periodic Perturbation of a Quasi-neutral State in 2D. Density, potential and divergence plots at time $t=0.005$ with $\veps = 10^{-4}$}     \label{fig:qn2d_per_e-4}
\end{figure}
\section{Conclusions and Future Works}
\label{sec:conclusion}
We have designed an energy stable scheme for the EP system under the quasineutral scaling. In order to achieve the decay of energy, a stabilisation technique has been used. Apriori energy bounds are established which in turn gives bounded numerical solutions. A Lax-Wendroff-type consistency of the numerical scheme with weak solutions of the continuous model as well as its consistency with the quasineutral limit have been shown. Results of numerical experiments are presented to substantiate the claims. 


\bibliographystyle{abbrv}
\bibliography{references}
\end{document}